\documentclass[11pt]{amsart}
\usepackage{amssymb,latexsym,amsmath,amscd,amsthm,amsfonts, enumerate}
\usepackage{multirow}
\usepackage{color}
\usepackage[all]{xy}
\usepackage{etex}
\usepackage{caption}
\usepackage{soul}
\usepackage{float}
\usepackage{titletoc}
\usepackage[normalem]{ulem}
\usepackage{graphicx}

\usepackage{tikz,tikz-cd}
\usepackage{mathrsfs}
\usepackage[all]{xy}
\usepackage{tikz}
\usepackage{extarrows}
\usepackage{tikz-cd}
\usetikzlibrary{calc}
\usetikzlibrary{matrix,arrows,decorations.pathmorphing}
\usetikzlibrary{snakes}
\usetikzlibrary{shapes.geometric,positioning}
\usetikzlibrary{arrows,decorations.pathmorphing,decorations.pathreplacing}
\usetikzlibrary{positioning,shapes,shadows,arrows,snakes}

\usepackage[colorlinks=true,pagebackref,hyperindex]{hyperref}
\newcommand\myshade{85}
\colorlet{mylinkcolor}{violet}
\colorlet{mycitecolor}{red}
\colorlet{myurlcolor}{cyan}

\hypersetup{
  linkcolor  = mylinkcolor!\myshade!black,
  citecolor  = mycitecolor!\myshade!black,
  urlcolor   = myurlcolor!\myshade!black,
  colorlinks = true,
}

\usepackage{tabularx}
\newcolumntype{E}{>{\hsize=0.5cm \centering\arraybackslash}X}%
\newcolumntype{C}[1]{>{\hsize=#1\hsize \centering\arraybackslash}X}%

\numberwithin{equation}{section}
\newtheorem{theorem}{Theorem}[section]
\newtheorem{proposition}[theorem]{Proposition}
\newtheorem{proposition-definition}[theorem]{Proposition-Definition}

\newtheorem{corollary}[theorem]{Corollary}
\newtheorem{lemma}[theorem]{Lemma}
\newtheorem{theoremA}{Theorem}

\theoremstyle{definition}
\newtheorem{remark}[theorem]{Remark}
\newtheorem{example}[theorem]{Example}
\newtheorem{definition}[theorem]{Definition}

\newcommand{\End}{\operatorname{End}\nolimits}
\newcommand{\thick}{\mathsf{thick}}
\newcommand{\Hom}{\mathrm{Hom}}

\newcommand{\Ext}{\mathrm{Ext}}

\newcommand{\za}{\alpha}
\newcommand{\zb}{\beta}

\newcommand{\zD}{\Delta}

\newcommand{\zg}{\gamma}
\newcommand{\zG}{\Gamma}

\newcommand{\Z}{\mathbb{Z}}
\newcommand{\aaa}{{\bf{a}}}
\newcommand{\bbb}{{\bf{b}}}
\newcommand{\ccc}{{\bf{c}}}
\newcommand{\ddd}{{\bf{d}}}

\newcommand{\pd}{\operatorname{{\rm pd }}}
\newcommand{\fd}{\operatorname{{\rm fd }}}
\newcommand{\id}{\operatorname{{\rm id }}}

\newcommand{\ma}{\operatorname{{\rm mod-A }}}%

\newcommand{\cals}{\mathcal{S}}

\newcommand{\calc}{\mathcal{C}}

\newcommand{\calh}{\mathcal{H}}

\newcommand{\cald}{\mathcal{D}}
\newcommand{\calm}{\mathcal{M}}

\newcommand{\calp}{\mathcal{P}}

\newcommand{\calk}{\mathcal{K}}
\newcommand{\bbp}{\mathbb{P}}
\newcommand{\dba}{\cald^b(A)}
\newcommand{\ka}{\calk^b(A)}
\newcommand{\kaa}{\calk^{-,b}(proj\text{-}A)}
\newcommand{\iaa}{\calk^{+,b}(inj\text{-}A)}

\renewcommand{\P}{P^\bullet}
\newcommand{\I}{I^\bullet}

\def\s{\stackrel}

\def\s{\stackrel}

\def\wt{\widetilde}

\definecolor{dark-green}{RGB}{14,150,2}
\definecolor{red}{RGB}{250,0,0}

\newcommand{\bpoint}{\circ}
\newcommand{\rpoint}{\color{red}{\bullet}}

\begin{document}

\title[Geometric models for the algebraic hearts of a gentle algebra]{Geometric models for the algebraic hearts in the derived category of a gentle algebra}

\author{Wen Chang}
\address{School of Mathematics and Statistics, Shaanxi Normal University, Xi'an 710062, China}\thanks{The author is supported by the Fundamental Research Funds
	for the Central Universities (No. GK202403003) and the NSF of China (Grant No. 12271321).}
\email{changwen161@163.com}


\thanks{}
%%\dedicatory{}

%
%\date{\today}

%(general geometric structures on low-dimensional manifolds)

\begin{abstract}
We give a geometric model for any algebraic heart in the derived category of a gentle algebra, which is equivalent to the module category of some gentle algebra. To do this, we deform the geometric model for the module category of a gentle algebra given in \cite{BC21}, and then embed it into the geometric model of the derived category given in \cite{OPS18}, in the sense that each so-called zigzag curve on the surface represents an indecomposable module as well as the minimal projective resolution of this module.
A key point of this embedding is to give a geometric explanation of the duality between the simple modules and the projective modules.
%We give a geometric model for any algebraic heart in the derived category of a gentle algebra, by embedding the geometric model of the module category of a gentle algebra into the geometric model of the associated derived category, in the sense that each so-called zigzag curve on the surface represents an indecomposable module as well as the minimal projective resolution of this module, with respect to different `coordinates'.

Such a blend of two geometric models provides us with a handy way to describe the homological properties of a module within the framework of the derived category. In particular, we realize any higher Yoneda-extension as a polygon on the surface, and realize the Yoneda-product as gluing of these polygons. As an application, we realize any algebraic heart in the derived category of a gentle algebra on the marked surface.
\end{abstract}

\maketitle
\setcounter{tocdepth}{2} %contents length

\tableofcontents

\section*{Introduction}\label{Introductions}

In recent years, topological and geometric methods have been used more and more widely in the representation theory of algebras.
There are many works on geometric models for various categories associated with an algebra.
A geometric model allows one to use surface combinatorics/topology to produce many applications, i.e.
realize arcs as indecomposables, intersections as morphisms, rotations of arcs as Auslander-Reiten translations, cuts of the surfaces as reductions of the categories, etc.

It has been shown that the surface models are powerful tools to solve a number of representation-theoretic questions on the module/derived categories of algebras, for example, classifying torsion pairs \cite{BBM14,CD20,CP16,HJR11}; giving the derived invariants of algebras \cite{AG16,ALP21,APS23,D14,O19,OZ22,LP20}; confirming the connectedness of support $\tau$-tilting graphs \cite{FGLZ23} and studying silting theory and the theory of exceptional sequences \cite{APS23,CJS22,CHS23,CS23a,CS23b} etc.

The use of a surface model to study the representations of an algebra originates from the cluster theory.
In \cite{FST08}, the authors discover a cluster algebra structure on an oriented surface with boundaries and marked points. They interpret a cluster of the cluster algebra as a triangulation of the surface and a mutation of the cluster as a flip of the triangulation.

Inspired by the surface model for a cluster algebra, the surface model for  a gentle algebra in the form of the Jacobian algebra of a triangulation is studied in \cite{ABCJP10,L09} (also in \cite{BZ11,CCS06}).
In \cite{DS12}, the authors deform Jacobian algebras of triangulations and introduce another class of gentle algebras so-called surface cut algebras.

More recently, for any gentle algebra, the geometric models for its module category and its derived category are established by using partial triangulations (or dissections) of the surfaces in \cite{BC21} and  \cite{OPS18} respectively. For a skew-gentle algebra, a generalization of a gentle algebra, the surface models for the associated module/derived categories are also established, for example, in \cite{A21,AB22,HZZ23,LSV22,QZZ22}.

Gentle algebras are classical objects in the representation theory of associative algebras \cite{AH81,AS87}.
Remarkably, they connect to many other areas of mathematics. For example, they play an important role in the homological mirror symmetry of surfaces, see e.g. \cite{B16,BD18,HKK17}. In particular, it has been shown  that the perfect derived category of a homologically smooth graded gentle algebra is triangle equivalent to the partially wrapped Fukaya category of a graded oriented smooth surface with marked points on the boundary \cite{HKK17, LP20}.

In this paper, we give a geometric model of any algebraic heart in the (bounded) derived category of a gentle algebra. A heart of a triangulated category is a full abelian subcategory, which is a key part in the construction of Brigdeland's stability conditions \cite{B07}. More precisely, a stability condition is equivalent to a heart with a stability function satisfying the so-called Harder-Narasimhan property. So a stability condition can be thought of as a continuous generalization of a heart. A heart is called an algebraic heart, if it is a length category and has finitely many isomorphism classes of simple objects.
Note that the stability conditions, and hence the hearts, on the derived category of a (graded) gentle algebra have been studied in \cite{HKK17}, while the geometric model established in this paper is from the perspective of representation theory, which is more concrete in some sense.
We also mention that the study of hearts in the bounded derived categories of a subclass of gentle algebras and the corresponding stability conditions
was carried out in \cite{BPP16, BPP17}.

More precisely, we will show that an algebraic heart in the derived category of a gentle algebra is equivalent to the module category of a gentle algebra.
Keeping this in mind, we give geometric models of the algebraic hearts by finding out module categories of various gentle algebras in the surface model of the derived category of a given gentle algebra $A$.

However, note that the geometric model of the module category $\ma$ in \cite{BC21} and the geometric model of the derived category $\dba$ in \cite{OPS18} are different, see details in \cite[Introduction and Section 2]{BC21}. For our purpose, that is, to realize any algebraic heart in the derived category, we make a  modification of the geometric model given in \cite{BC21} and unify these two models on the same surface, see more explanation for the differences between the geometric model in \cite{BC21} and this paper in Remark \ref{lem:compBC21}.
%by adding additional marked points on the boundary components, and by contracting the unmarked boundary components into punctures.
%In particular, an indecomposable module may correspond to several curves in \cite{BC21}, which are not homotopic with each other, while in our model, it corresponds to a unique curve.

Let's give a more detailed explanation.
We realize both the module category and the derived category of a gentle algebra on the same surface, with respect to different `coordinates': the \emph{simple coordinate} $\zD_s^*$ and the \emph{projective coordinate} $\zD_p^*$ (see Definitions \ref{definition:addmissable dissections} and \ref{definition:addmissable dissections in prelimilary}) respectively, where both $\zD_s^*$ and $\zD_p^*$ are kinds of partial triangulations of the surface which are related by arc-twist and dual (see Figures \ref{figure:admissible dissection-pre} and \ref{figure:dual-twist}).
Then a so-called \emph{zigzag curve} (see Definition \ref{definition:zigzag arcs}) on the surface represents an indecomposable $A$-module if we view $\zD_s^*$ as the coordinate, while it represents the minimal projective resolution of this module if we view $\zD_p^*$ as the coordinate. More precisely, we have the following theorem, where the first part has been essentially given in \cite[Theorem 2]{BC21}.

\begin{theoremA}[Theorem \ref{theorem:main arcs and objects}, Theorem  \ref{theorem:object}]\label{Mtheorem:object}
Let $(\cals,\calm,\zD_s^*)$ be a marked surface with a simple coordinate, and let $\zD_p^*$ be the associated projective coordinate. Then the gentle algebras arising from $\zD_s^*$ and $\zD_p^*$ are isomorphic, which is denoted by $A$. Moreover,
\begin{enumerate}[\rm(1)]
  \item there is a map $M$ which gives rise to a bijection between \emph{zigzag arcs} $\za$ on $(\cals,\calm,\zD^*_s)$ and string $A$-modules $M_\za$, as well as a bijection between \emph{(primitive) closed zigzag curves} $\za$ on $(\cals,\calm,\zD^*_s)$ and classes of band $A$-modules $M_{(\za,\lambda,m)}$, for $\lambda\in k^*, m\in \mathbb{N}$;
  \item the projective complex $\P_{\za}$ associated to any zigzag arc $\za$, with respect to $\zD_p^*$, is the minimal projective resolution of $M_{\za}$; the projective complex $\P_{(\za,\lambda,m)}$ associated to any zigzag closed curve $\za$, with respect to $\zD_p^*$, is the minimal projective resolution of $M_{(\za,\lambda,m)}$, for any $\lambda\in k^*, m\in \mathbb{N}$.
 \end{enumerate}
  \end{theoremA}

As one might have expected, the \emph{weighted-oriented intersections} of curves (see Definitions \ref{definition:weighted intersections1}, \ref{definition:extension in interior}, and \ref{definition:weighted intersections}) will be explained as homomorphisms/extensions of the associated modules, see Theorem \ref{theorem:main-extensions}.

%In fact, we introduce four kinds of `dissections' and `coordinates' in Definitions \ref{definition:addmissable dissections in prelimilary} and \ref{definition:addmissable dissections}: the simple dissection $\zD_s$, the simple coordinate $\zD^*_s$, the projective dissection $\zD_p$ and the projective coordinate $\zD^*_p$.
%In some sense, each of them is a kind of partial triangulation or admissible dissection used in \cite{BC21,OPS18,PPP19}.
%They are related by two operations: a duality $*$, which is an involution, and a twist $t$ (as well as the inverse $t^{-1}$), which clockwise rotates both endpoints of an arc to next endpoints on the same boundary components, see Lemma \ref{lem:comm-diag-four-opp} and Figure \ref{figure:dual-twist}.

Since on the surface model established above, a curve does not only represent an indecomposable module, as well as the minimal projective resolution of this module, thus it is convenient to study the homological properties of a module by using the surface model. For example, we get an explicit characterization for the projective resolution of any indecomposable module, see Proposition \ref{prop:proj-res} and Remark \ref{remark:inj}.
%In particular, for any indecomposable module, there are at most two terms in the $n$-th position of its minimal projective resolution when $n\leq -2$.
The projective dimension and the injective dimension of an indecomposable $A$-module, as well as the finitistic dimension of $A$ can be read directly from the surface, see
Corollaries \ref{corollary:proj-res-string} and \ref{corollary:fdim}.

Note that the syzygies and the extensions of modules over gentle algebras (and over more general string algebras) have been studied in several papers, for example \cite{BDMTY20,BS21,CS17,CPS19,CPS21,HS05,K15,S99,Z14}.
Here we focus on the higher-extensions, and explain them as polygons of the surface, see the following theorem, which can be viewed as a generalization of the phenomenon that certain triangles on the surface represent distinguished triangles as well as $1$-extension of objects in the associated category.
In particular, any higher-extension always arises from an oriented intersection at the endpoint, while an interior non-punctured oriented intersection always gives rise to a 1-extension. 
We mention that this phenomenon is also discovered in \cite{BS21}, where the authors study the higher-extensions for a gentle algebra arising from a triangulation of a marked surface.

\begin{theoremA}[Theorem \ref{theorem:yoneda}]
Let $\aaa$ be a boundary/puncture-oriented intersection from an $\bpoint$-arc $\za$ to an $\bpoint$-arc $\zb$ on a marked surface $(\cals,\calm,\zD^*_s)$, which gives rise to a map $\aaa$ in $\Ext^\omega(M_\za,M_\zb)$, where $\omega$ is the number of the (red) arcs in $\zD^*_s$ in between $\za$ and $\zb$, as shown in Figure \ref{figure:Mhigher-extension}.  
Then the Yoneda $\omega$-extension $[\aaa]$ is given as follows, where the arcs $\zg_i$ and the maps $\bbb_i$ are depicted in the distinguished $(\omega+2)$-gon $\bbp(\aaa)$ of Figure \ref{figure:Mhigher-extension}:

\begin{gather}\label{eq:Myoneda}
\xymatrix@C=0.7cm
{[\aaa]=0\ar[r]^{} & M_{\zb}\ar[r]^{\bbb_1} & M_{\zg_1}\ar[r]^{\bbb_2} & \cdots \ar[r]^{} & M_{\zg_{\omega}}\ar[r]^{\bbb_{\omega+1}} & M_{\za}\ar[r]& 0}.
\end{gather}

\begin{figure}
 \[\scalebox{1}{
\begin{tikzpicture}[>=stealth,scale=0.8]
\draw[red!50,thick] (-4,1)--(-4,3.5)--(-2,5);
\draw[red!50,thick] (4,1)--(4,3.5)--(2,5);
\draw[red!50,dashed,thick] (-2,0)--(-4,1);
\draw[red!50,dashed,thick] (2,0)--(4,1);
\draw[red!50,dashed,thick] (-2,5)--(2,5);
\draw[thick](-2,0)--(2,0);

\draw[thick] (-6,2.5)--(0,0)--(6,2.5);
\draw[dashed,thick,black!50] (-4,5)--(0,0)--(4,5);
\draw[thick,dashed,black!50] (0,0)--(0,6);

\draw[thick]plot [smooth,tension=1] coordinates {(-6,2.5) (-4,3) (-4,5)};
\draw[thick]plot [smooth,tension=1] coordinates {(6,2.5) (4,3) (4,5)};
\draw[thick]plot [smooth,tension=1] coordinates {(0,6) (-1.5,4) (-4,5)};
\draw[thick]plot [smooth,tension=1] coordinates {(0,6) (1.5,4) (4,5)};

\draw[red,thick,fill=red] (2,0) circle (0.08);
\draw[red,thick,fill=red] (4,1) circle (0.08);
\draw[red,thick,fill=red] (4,3.5) circle (0.08);
\draw[red,thick,fill=red] (2,5) circle (0.08);
\draw[red,thick,fill=red] (-2,0) circle (0.08);
\draw[red,thick,fill=red] (-4,1) circle (0.08);
\draw[red,thick,fill=red] (-4,3.5) circle (0.08);
\draw[red,thick,fill=red] (-2,5) circle (0.08);

\node at (-2,.6) {\tiny$\za$};
\node at (2,.6) {\tiny$\zb$};
\node[black!50] at (-1.8,3) {\tiny$\ell_{t+1}$};
\node[black!50] at (1.7,3) {\tiny$\ell_{t+\omega-1}$};
\node [red!50] at (-3.65,2.2) {\tiny$\ell^*_{t}$};
\node [red!50] at (3.5,2.2) {\tiny$\ell^*_{t+\omega}$};
\node [red!50] at (2.84,4.9) {\tiny$\ell^*_{t+\omega-1}$};
\node at (-5,3) {\tiny$\zg_{\omega}$};
\node at (5,3) {\tiny$\zg_1$};

\draw[thick,fill=white] (0,0) circle (0.08);
\draw[thick,bend left,->](-.5,.2)to(.5,.2);
\node [] at (.2,.5) {\tiny$\aaa$};
\node at (0,-.5) {\tiny$q$};

\draw[thick,fill=white] (4,5) circle (0.08);
\node [] at (4.3,5.3) {\tiny$p_2$};
\draw[thick,fill=white] (-4,5) circle (0.08);
\node [] at (-4.3,5.3) {\tiny$p_{\omega}$};

\draw[thick,fill=white] (0,6) circle (0.08);
\draw[thick,dashed,bend left,->](-.5,.2)to(.5,.2);

\draw[thick,fill=white] (-6,2.5) circle (0.08);
\draw[thick,bend left,->](-5,2.5)to(-5,2.11);
\node [] at (-4.4,2.3) {\tiny$\bbb_{\omega+1}$};
\node [] at (-6.6,2.6) {\tiny$p_{\omega+1}$};

\draw[thick,fill=white] (6,2.5) circle (0.08);
\draw[thick,bend left,->](5,2.11)to(5,2.5);
\node [] at (4.65,2.3) {\tiny$\bbb_1$};
\node [] at (6.4,2.6) {\tiny$p_1$};
\end{tikzpicture}}\]
\begin{center}
\caption{A distinguished $(\omega+2)$-gon $\bbp(\aaa)$ (with the bold black arcs as the edges) represents a Yoneda $\omega$-extension $[\aaa]$ given in \eqref{eq:Myoneda}, which corresponds to a map in $\Ext^\omega(M_\za,M_\zb)$ arising from an oriented intersection $\aaa$ from $\za$ to $\zb$ with weight $\omega$.}\label{figure:Mhigher-extension}
\end{center}
\end{figure}
\end{theoremA}

Furthermore, the Yoneda product of two Yoneda extensions $[\aaa]$ and $[\bbb]$ can be explained as gluing $\bbp(\aaa)$ and $\bbp(\bbb)$ to a new polygon $\bbp(\aaa\bbb)$, see Remark \ref{remark:yoneda}.

At last, to give a geometric model for any algebraic heart in the derived category of a gentle algebra $A$, we fix a marked surface $(\cals,\calm,\zD^*_A)$ with an initial projective coordinate, and we introduce a notion of \emph{simple-minded dissection} on the surface, which is a pair $(\zD_s,f)$ of a simple dissection $\zD_s$ together with a set $f$ consisting of gradings of arcs in $\zD_s$ (see Definition \ref{definition:silting dissection}). We show that a simple-minded dissection $(\zD_s,f)$ gives rise to a geometric realization of a simple-minded collection $\P_{(\zD_s,f)}$.

Given a simple-minded dissection $(\zD_s,f)$, we construct a \emph{graded simple coordinate} $(\zD^*_s,F^*)$, where $F^*$ is a grading on $\zD^*_s$ induced by $f$. The graded simple coordinate $(\zD^*_s,F^*)$ will be used as the `coordinate' to realize the heart generated by the simple-minded collection $\P_{(\zD_s,f)}$.
More precisely, let $\zG$ be the subalgebra of the graded gentle algebra associated to $(\zD^*_s,F^*)$ consisting of elements with zero gradings, then $\zG$ is a gentle algebra and we have the following

\begin{theoremA}[Theorem \ref{theorem: last}]\label{Mtheorem: last}
Let $\calh$ be an algebraic heart of $\dba$ generated by a simple-minded collection $\P_{(\zD_s,f)}$. Then $\calh$ is equivalent to the module category of the gentle algebra $\zG$. Furthermore, any indecomposable object in $\calh$ is of the form $\P_{(\za,g)}$ for a {\em zigzag $\bpoint$-arc} $\za$ on $(\cals,\calm,\zD^*_s,F^*)$ and some grading $g$, or of the form $\P_{(\za,g,\lambda,m)}$ for a {\em zigzag closed curve} $\za$ on $(\cals,\calm,\zD^*_s,F^*)$ and some grading $g$, $\lambda\in k^*$, $m\in \mathbb{N}$.
\end{theoremA}

Furthermore, the (simple) HRS-tilting of a torsion pair in a heart $\calh$ can be interpreted as certain flip of arcs in the associated graded simple coordinate, see Remark \ref{remark:last}.

The paper is organized as follows. We begin by recalling some background on marked surfaces, module/derived categories of gentle algebras, and algebraic hearts in derived categories in Section \ref{Preliminaries}. Then we give a geometric model for the module category of a gentle algebra in Section \ref{section:geo-model-module}. 
%In particular, we embed the geometric model of the module category into the geometric model of the associated derived category in subsections \ref{subsection: Modules as curves}, \ref{section:dual of simples and projectives} and \ref{subsection:Module category vs derived category}, and we give an explicit description of the minimal projective resolution of an indecomposable module in subsection \ref{subsection:projective-res}, and then a geometric characterization of the morphisms/extensions and the higher Yoneda extensions will be given in subsections \ref{subsection:morph} and \ref{section:yonada} respectively.
Section \ref{section:heart} is devoted to giving a geometric model for any algebraic heart in the derived category of a gentle algebra.
Finally, we give an example in Section \ref{section:example} to explain the above results.

\section*{Acknowledgments}
I would like to express my thanks to Ping He, Haibo Jin, Yu Qiu, Sibylle Schroll, Yu Zhou and Bin Zhu for interesting discussions and suggestions on this topic. I also would like to thank Raquel Coelho Sim\~oes for telling me the related work by Thomas Br\"{u}stle, David Pauksztello and Helene R. Tyler, and to thank Thomas Br\"{u}stle, David Pauksztello and Helene R. Tyler for sharing their ongoing preprint \cite{BPT}. Many thanks to the referee for the very careful reading of the paper
and for their valuable remarks. 
In particular, I would like to thank a referee for pointing out a gap in the proof of Theorem \ref{theorem: last} in the previous version, and for providing a
way to correct this.

\section{Preliminaries}\label{Preliminaries}

In this paper, an algebra will be assumed to be basic with finite dimension over a base field $k$ which is algebraically closed. A quiver will be denoted by $Q=(Q_0,Q_1)$, where $Q_0$ is the set of vertices and $Q_1$ is the set of arrows. For an arrow $\aaa$, $s(\aaa)$ is the source and $t(\aaa)$ is the target of it.
Arrows in a quiver are composed from left to right as follows: for arrows $\aaa$ and $\bbb$ we write $\aaa\bbb$ for the path from the source of $\aaa$ to the target of $\bbb$.
We adopt the convention that maps are also composed  from left to right, that is if $f: X \to Y$ and $g: Y \to Z$ then $fg : X \to Z$. In general, we consider right modules.
We denote by $\mathbb{Z}$ the set of integer numbers, and by $\mathbb{N}$ the set of positive integer numbers.
For a finite set $M$, we denote by $|M|$ its cardinality.

\subsection{Module categories of gentle algebras}\label{subsection: gentle algebras}

In this subsection, we recall some basic definitions and properties of a gentle algebra and its module category.

\begin{definition}\label{definition:gentle algebras}
We call an algebra $A=kQ/I$ a \emph{gentle algebra}, if $Q=(Q_0,Q_1)$ is a finite quiver and $I$ is an admissible ideal of $kQ$ satisfying the following conditions:
\begin{enumerate}[\rm(1)]
 \item Each vertex in $Q_0$ is the source of at most two arrows and the target of at most two arrows.

 \item For each arrow $\aaa$ in $Q_1$, there is at most one arrow $\bbb$ such that  ${\bf{0}} \neq \aaa\bbb\in I$; at most one arrow $\ccc$ such that  ${\bf{0}} \neq \ccc\aaa\in I$; at most one arrow $\bbb'$ such that $\aaa\bbb'\notin I$; at most one arrow $\ccc'$ such that $\ccc'\aaa\notin I$.

 \item $I$ is generated by paths of length two.
\end{enumerate}
\end{definition}

Throughout this section, $A=kQ/I$ denotes a gentle algebra. It is known that any finite dimensional indecomposable right A-module is either a string module or a band module, which is parameterized by string and band combinatorics respectively.
We briefly recall these results here and we refer the reader to \cite{BR87} for more details.

For an arrow $\aaa$, let $\aaa^{-1}$ be its \emph{formal inverse} with $s(\aaa^{-1})=t(\aaa)$ and $t(\aaa^{-1})=s(\aaa)$. A \emph{walk} is a (possibly infinite) sequence $\cdots\omega_1\omega_2\cdots\omega_m\cdots$ of arrows and inverse arrows in $Q$ such that $t(\omega_i)=s(\omega_{i+1})$ and $\omega_{i+1}\neq\omega^{-1}_i$ for each $i$.
A walk $\omega$ is called a \emph{finite walk}, if it consists of finite number of arrows and formal inverses.

A \emph{string} is a finite walk $\omega$ that avoids relations, that is, there is no subsequence of $\omega$ or of $\omega^{-1}$ belongs to $I$.
A \emph{direct string} is a string consisting of arrows, and an \emph{inverse string} is a string consisting of formal inverses.
For each vertex $v$, we associate a \emph{trivial string} $1_v$ to it.
Each string $\omega$ defines a \emph{string module} $M(\omega)$, which is given by the quiver representation of type $A$ obtained by replacing every vertex in $\omega$ by a copy of $k$ and every arrow by the identity map.
This gives a bijection between the inversion equivalent classes of strings and the isomorphism classes of string modules. In particular, $M(1_v)$ is the simple module arising from the vertex $v$ of $Q$.

A \emph{band} is a string $\omega$ such that any power $\omega^n$ is a string and $\omega$ itself is not a proper power of any string.
Each band $\omega=\omega_1\omega_2\cdots\omega_n$ defines a family of \emph{band modules} $M(\omega,m,\lambda)$ for a positive integer $m\in \mathbb{N}$ and a parameter $\lambda\in k^*$.
More precisely, each vertex of $\omega$ is replaced by a copy of the vector space $k^m$ and the action of an arrow $\omega_i$ (or $\omega^{-1}_i$) is induced by identities if $1 \leq i \leq n-1$ and by $\phi(\lambda)$ if $i=n$, where $\phi(\lambda)$ is an irreducible automorphism of $k^m$ with eigenvalue $\lambda$. Note that since we assume that the base field $k$ is algebraically closed, after choosing a basis of the vector spaces, an irreducible automorphism $\phi(\lambda)$ is given by a Jordan-block $B(\lambda)_m$ with rank $m$ and eigenvalue $\lambda$.
Furthermore, up to inversion and rotation, the bands together with $m\in \mathbb{N}$ and $\lambda\in k^*$ form an indexing set of band modules.
Note that for a given band $\omega$ and a parameter $\lambda$, all the band modules $M(\omega,m,\lambda)$ form a homogenous tube in the AR-quiver of the module category, where the module $M(\omega,1,\lambda)$ is at the mouth of tube, which will be called a \emph{quasi-simple band module}.

\subsection{Derived categories of gentle algebras}\label{subsection: Derived categories of gentle algebras}
In this subsection, we recall the classification of indecomposable objects in the derived category of a gentle algebra $A = kQ/I$, by using homotopy strings and homotopy bands introduced in \cite{B11}. We refer the reader to \cite{BM03,B11} for more details.

A \emph{homotopy string} $\sigma = \cdots\omega_1\omega_2\cdots\omega_n\cdots$ on $(Q,I)$ is a walk in $Q$ consisting of subwalks $\ldots, \sigma_1,\sigma_2, \ldots, \sigma_r, \ldots$ with  $\sigma = \cdots\sigma_1\sigma_2 \cdots \sigma_r\cdots$
and such that
\begin{enumerate}[\rm(1)]
\item each $\sigma_i$ is a direct or inverse string;
\item if both $\sigma_i, \sigma_{i+1}$ are direct (resp. inverse) strings then $\sigma_{i} \sigma_{i+1}$ (resp. $\sigma_{i+1}^{-1} \sigma_{i}^{-1}$) $\in I$;
\item if $\sigma$ is infinite from the right (resp. left), then there exists $l$ such that $\sigma_{l} \sigma_{l+1} \ldots$ (resp. $ \ldots \sigma_{l-1} \sigma_{l}$) is eventually periodic with each $\sigma_i, i\geq l$ (resp. $i\leq l$) an arrow (resp. formal inverse arrow) in~$Q$.
\end{enumerate}

For an infinite homotopy string, we will use the following notation: $$\sigma=(\sigma_{-s}\cdots\sigma_{-2}\sigma_{-1})^{\infty}\sigma_1
\sigma_2\cdots\sigma_n(\sigma_{n+1}
\sigma_{n+2}\cdots\sigma_{n+t})^{\infty},$$
where $\sigma_{-i}$ are formal inverse arrows for $1\leq i \leq s$, and $\sigma_{n+i}$ are direct arrows for $1\leq i \leq t$.

A \emph{homotopy band} is a finite homotopy string $\sigma = \sigma_1 \cdots \sigma_r$ with an equal number of direct and inverse strings $\sigma_i$ such that any power $\sigma^n$ is a homotopy string and $\sigma$ itself is not a proper power of any homotopy string.

A \emph{grading} on a homotopy string $\sigma$ is a sequence of integers $\mu=(\ldots, \mu_1, \mu_2, \ldots, \mu_r, \ldots)$ such that for each $i$,
    \[ \mu_{i+1} = \begin{cases}
                          \mu_i - 1 & \textrm{if $\sigma_{i+1}$ is a direct homotopy string;} \\
                          \mu_i + 1 & \textrm{if $\sigma_{i+1}$ is an inverse homotopy string.}
                      \end{cases}
    \]
The pair $(\sigma, \mu)$ is called a \emph{graded homotopy string}.

We similarly define a \emph{grading} on a homotopy band, and a \emph{graded homotopy band}.

A graded homotopy string $(\sigma,\mu)$ gives rises to an indecomposable object $\P_{(\sigma,\mu)}$, and a homotopy band $(\sigma,\mu)$ gives rise to a family of indecomposable objects $\P_{(\sigma,\mu,\lambda,m)}$ in $\kaa$, with $\lambda\in k^*$ and $m\in \mathbb{N}$. We briefly recall the construction in the following.

Let $(\sigma,\mu)$ be a graded homotopy string on $(Q,I)$ with $\sigma=\sigma_1\sigma_2\cdots \sigma_n$. Denote by $P_{v_i}$ the projective object associated to the vertex $v_i:=s(\sigma_i)$ for $1\leq i \leq n$ and $v_{n+1}:=t(\sigma_n)$. Then $\sigma_i$ induces a canonical non-zero map from $P_{v_{i+1}}$ to $P_{v_{i}}$ (resp. from $P_{v_{i}}$ to $P_{v_{i+1}}$), which is still denoted by $\sigma_i$ (resp. by $\sigma_i^{-1}$), if $\sigma_i$ is a direct string (resp. inverse string).
We put $P_{v_i}$ at $\mu_i$-th position of $\P_{(\sigma,\mu)}$ and put $P_{v_{i+1}}$ at $\mu_{i+1}$-th position of $\P_{(\sigma,\mu)}$, and connect them by the map $\sigma_i$ (resp. $\sigma_i^{-1}$). In this way, we obtain a complex $\P_{(\sigma,\mu)}$ by connecting all the projectives that $\sigma$ goes through.
We call $\P_{(\sigma,\mu)}$ the \emph{homotopy string object} arising from $(\sigma,\mu)$.

Let $(\sigma,\mu)$ be a graded homotopy band on $(Q,I)$, where $\sigma=\sigma_1\cdots\sigma_n$. Then similarly, $\sigma_i$ induces a canonical non-zero map between $P_{v_{i}}$ and $P_{v_{i+1}}$.
For each $1\leq i \leq n-1$, we put $P^m_{v_i}$ at $\mu_i$-th position of $\P_{(\sigma,\mu,\lambda,m)}$ and put $P^m_{v_{i+1}}$ at $\mu_{i+1}$-th position of $\P_{(\sigma,\mu,\lambda,m)}$, and connect them by the map $\sigma_iI_m$ or $\sigma_i^{-1}I_m$, depending on $\sigma_i$ is direct or inverse, where $I_m$ is the identity matrix of rank $m$. At last, we connect $P_{v_1}$ and $P_{v_n}$ by $\sigma_nB(\lambda)_m$ or $\sigma^{-1}_nB(\lambda)_m$, depending on $\sigma_n$ is direct or inverse. In this way, we obtain a complex $\P_{(\sigma,\mu,\lambda,m)}$ by connecting all the direct sums of projectives that $\sigma$ goes through.
We call $\P_{(\sigma,\mu,\lambda,m)}$ a \emph{homotopy band object} arising from $(\sigma,\mu)$.

Note that each object in $\kaa$ is either a homotopy string object or a homotopy band object. Furthermore, up to inversion, the homotopy strings  form an indexing set of homotopy string objects, and up to inversion and rotation, the homotopy bands together with $m\in \mathbb{N}$ and $\lambda\in k^*$ form an indexing set of homotopy band objects.

\subsection{Marked surfaces and surface dissections}\label{subsection: marked surfaces}

We recall some concepts about marked surfaces, and the construction of gentle algebras from dissections of the surfaces, for which there are many references such as \cite{HKK17,PPP19,LP20,BC21,PPP21}, in this paper we closely follow \cite{OPS18} and \cite{APS23}.

\begin{definition}
\label{definition:marked surface}
A \emph{marked surface} is a pair $(\cals,\calm)$, where
  \begin{enumerate}[\rm(1)]
 \item $\cals$ is an oriented surface with non-empty boundary with
  connected components $\partial \cals=\sqcup_{i=1}^{b}\partial_i \cals$;
 \item $\calm = \calm_{\bpoint} \cup \calm_{\rpoint} \cup \calp_{\bpoint}$ is a finite set of \emph{marked points} on $\cals$. The elements of~$\calm_{\bpoint}$ and~$\calm_{\rpoint}$ are on the boundary of $\cals$, which will be respectively represented by symbols~$\bpoint$ and~$\rpoint$. Each connected component $\partial_i \cals$ is required to contain at least one marked point of each colour, where the points~$\bpoint$ and~$\rpoint$ are alternating on $\partial_i \cals$. The elements in $\calp_{\bpoint}$ are in the interior of $\cals$. We refer to these points as \emph{punctures}, and we will also represent them by the symbol $\bpoint$.
  \end{enumerate}
\end{definition}

Note that in \cite{APS23}, the punctures in the surface is represented by $\rpoint$-point, while here we use $\bpoint$-point, since such choose is more convenient for our purpose, that is, we will focus on the module category of a gentle algebra, rather than its derived category.

On the surface, all curves are considered up to homotopy with respect to the boundary components and the punctures, and all intersections of curves are required to be transversal.
We fix the clockwise orientation as the orientation of the surface, that is, when we following the boundary, the interior of the surface is on the right.

\begin{definition}
\label{definition:arcs}
Let $(\cals,\calm)$ be a marked surface.
\begin{enumerate}[\rm(1)]
 \item An \emph{arc} is a non-contractible curve, with endpoints in~$\calm_{\bpoint}\cup \calp_{\bpoint}\cup \calm_{\rpoint}$.
  \item A \emph{loop} is an arc whose endpoints coincide.
 \item An \emph{$\bpoint$-arc} is a non-contractible curve, with endpoints in~$\calm_{\bpoint}\cup \calp_{\bpoint}$.
 \item An \emph{$\rpoint$-arc} is a non-contractible curve, with endpoints in $\calm_{\rpoint}$.
 \item A \emph{simple arc} is an arc without interior self-intersections.
 \item An \emph{exceptional arc} is a simple arc without boundary self-intersections, that is, has two distinct endpoints.
 \item A \emph{closed curve} is a non-contractible curve in the interior of $\cals$ whose endpoints coincide. We always assume a closed curve to be \emph{primitive}, that is, it is not a non-trivial power of a closed curve in the fundamental group of $\cals$.
 \end{enumerate}
\end{definition}

In order for some definitions and notations to be well-defined in the case of a loop, we will treat the unique endpoint of a loop as two distinct endpoints.

\begin{definition}\label{definition:addmissable dissections in prelimilary}
Let $(\cals,\calm)$ be a marked surface.
\begin{enumerate}[\rm(1)]
 \item A collection of simple arcs with endpoints in $\calm_{\bpoint}$ is called a \emph{projective dissection}, if the arcs have no interior intersections and they cut the surface into polygons each of which contains exactly one marked point from $\calm_{\rpoint}\cup \calp_{\bpoint}$.
 \item A collection of simple arcs with endpoints in $\calm_{\rpoint}\cup\calp_{\bpoint}$ is called a \emph{projective coordinate}, if the arcs have no interior intersections and they cut the surface into polygons each of which contains exactly one marked point from $\calm_{\bpoint}$.
  \end{enumerate}
\end{definition}

A projective dissection will be denoted by $\zD_p$. Then it is not hard to see that (cf. \cite{OPS18} for example) for each $\zD_p$, there is a unique (up to homotopy) projective coordinate on $(\cals,\calm)$, which we will denote by $\zD^*_p$, such that each arc $\ell^*$ in $\zD^*_p$ intersects exactly one arc $\ell$ of $\zD_p$. We call $\zD^*_p$ the \emph{dual} (\emph{projective coordinate}) of $\zD_p$, and call $\zD_p$ the \emph{dual} (\emph{projective dissection}) of $\zD^*_p$. 

We mention that a projective coordinate introduced above is called an $\rpoint$-admissible dissection in \cite{APS23}, and  corresponds to a lamination considered in \cite{OPS18}. Note that in \cite{OPS18}, empty  boundary components without marked points are allowed, which are contracted as $\rpoint$-punctures in this paper. There is only one kind of marked points used in  \cite{OPS18}, which corresponds to boundary-$\bpoint$-points in this paper.
Each endpoint of an arc in a lamination locates on a segment between two adjacent marked points, or on an empty boundary component. For a lamination, by contracting the endpoints on the same segment (or empty boundary component) into a point, which corresponds to a red marked point (or a puncture) used in this paper, we get a projective coordinate defined above. 

\begin{remark}
In Definition \ref{definition:addmissable dissections}, we will introduce the simple dissection and the simple coordinate. The terminologies we used here will be justified in Theorem \ref{theorem:main arcs and objects} and Theorem \ref{theorem:projectives from simple coordinates}, that is, the projective coordinate and the simple coordinate will be respectively viewed as kinds of `coordinates' when we study the derived category and the module category of a gentle algebra, and the arcs in the dual projective dissection and the dual simple dissection will respectively represent indecomposable projective modules and simple modules of this gentle algebra.
\end{remark}

\begin{definition}\label{definition:oriented intersection}
Let $q$ be a common endpoint of arcs $\ell_i^*, \ell_j^*$ in $\zD_p^*$, an \emph{oriented intersection} from $\ell_i^*$ to $\ell_j^*$ at $q$ is a clockwise angle locally from $\ell_i^*$ to $\ell_j^*$ based at $q$ such that the angle is in the interior of the surface. We call an oriented intersection a \emph{minimal oriented intersection} if it is not a composition of two oriented intersections of arcs from $\zD_p^*$.
\end{definition}

\begin{definition}\label{definition:gentle algebra from dissection}
We define the algebra $A(\zD_p^*)$ as the quotient of the path algebra $kQ(\zD_p^*)$ of the quiver $Q(\zD_p^*)$ by the ideal $I(\zD_p^*)$ defined as follows:
\begin{enumerate}[\rm(1)]
	\item The vertices of $Q(\zD_p^*)$ are given by the arcs in $\zD_p^*$.
	\item Each minimal oriented intersection $\aaa$ from $\ell^*_i$ to $\ell^*_j$ gives rise to an arrow from $\ell^*_i$ to $\ell^*_j$, which is still denoted by $\aaa$.
	\item The ideal $I(\zD_p^*)$ is generated by paths $\aaa\bbb:\ell^*_i\rightarrow \ell^*_j\rightarrow \ell^*_k$, where the common endpoint of $\ell^*_i$ and $\ell^*_j$, and the common endpoint of $\ell^*_j$ and $\ell^*_k$ that respectively gives rise to $\aaa$ and $\bbb$ coincide.
\end{enumerate}
\end{definition}

Then it is straightforward to check that $A(\zD_p^*)$ is a gentle algebra. In fact it is proved in \cite{BC21,OPS18,PPP19} that any gentle algebra arises from this way. So this establishes a bijection between the set of triples  $(\cals,\calm,\zD_p^*)$ of homeomorphism classes of marked surfaces together with projective coordinates and the set of isomorphism classes of gentle algebras $A(\zD_p^*)$.
\begin{remark}
	We should mention that when we define the algebras from a dissection on a surface, we take the opposite orientation to that used in \cite{OPS18} and \cite{BC21}. 
\end{remark}
By definition, $\zD_p^*$ cuts the surface into polygons each of which has a unique $\bpoint$-point from $\calm_{\bpoint}$. Let $\za$ be an $\bpoint$-arc or a closed curve on $(\cals,\calm)$, and we choose a direction of $\za$. Denote by $\za \cap \zD_p^*$ the totally ordered set of intersection points of~$\za$ with arcs in $\zD_p^*$, where the order is induced by the direction of $\za$.
A \emph{grading} of $\za$ is a function $f: \za \cap \zD_p^* \longrightarrow \mathbb{Z},$ which is defined as follows:
If~$p_1$ and~$p_2$ are in~$\za \cap \zD_p^*$ and~$p_2$ is the direct successor of~$p_1$, then~$\za$ enters a polygon $\bbp$ formed by arcs in $\zD_p^*$ via~$p_1$ and leaves it via~$p_2$. If the $\bpoint$-point~ in $\bbp$ is to the left of~$\za$, then~$f(p_2) = f(p_1)-1$; otherwise,~$f(p_2) = f(p_1)+1$. We call the pair $(\za,f)$ a \emph{graded curve} on  $(\cals,\calm,\zD_p^*)$.

It is not hard to see that the definition of the grading is independent on the choice of the direction of the arc. However, if $\za$ is a closed curve, then it may not be gradable, that is there may not exist a grading of $\za$ since in this case $f$ might not be well-defined.

\subsection{Geometric models for the derived categories of gentle algebras}\label{subsection: derived categories and derived categories}

In this subsection we recall the geometric model of the derived category of a gentle algebra established in \cite{OPS18}. In the following, let $(\cals,\calm,\zD_p^*)$ be a marked surface with a projective coordinate, and let $A(\zD_p^*)$ be the associated gentle algebra.

We first recall how to construct a graded homotopy string/band $(\sigma, \mu)$ of $A(\zD_p^*)$ from a graded $\bpoint$-arc/closed curve $(\za,f)$ on $(\cals,\calm,\zD_p^*)$.

Let $\za$ be an $\bpoint$-arc or a closed curve on a the surface. After choosing a direction of $\za$, denote by $\bbp_0, \bbp_1,\cdots, \bbp_{n+1}$ the ordered polygons of $\zD^*_p$ that successively intersect with $\za$, whenever $\za$ is a closed curve the index is considered modulo $n+1$, and in particular $\bbp_0=\bbp_{n+1}$. Denote by $\ell^*_0, \ell^*_1,\cdots, \ell^*_{n}$ the ordered arcs in $\zD^*_p$ that successively intersect with $\za$ such that $\ell^*_i$ belongs to $\bbp_{i}$ and $\bbp_{i+1}$ for each $0 \leq i \leq n$.

(1) If $\za$ is an $\bpoint$-arc with both endpoints in $\calm_{\bpoint}$.
Then for each $1\leq i \leq n$, there is a canonical non-zero path $\sigma_i$ in $A(\zD^*_p)$ from $\ell^*_{i-1}$ to $\ell^*_i$ (from $\ell^*_i$ to $\ell^*_{i-1}$ resp.), which is the composition of arrows arising from the minimal oriented intersections in $\bbp_i$, if the (unique) $\bpoint$-point in $\bbp_i$ is to the left (right resp.) of $\za$. We denote by $\overrightarrow{\sigma}_i=\sigma_i$ ($=\sigma^{-1}_i$ resp.) if $\sigma_i$ is a path from $\ell^*_{i-1}$ to $\ell^*_i$ (from $\ell^*_i$ to $\ell^*_{i-1}$ resp.).
Let $\sigma=\overrightarrow{\sigma}_1\overrightarrow{\sigma}_2\cdots\overrightarrow{\sigma}_n$, which is a finite homotopy string. We define the grading $\mu$ on $\sigma$ by setting $\mu_i=f(\ell^*_i)$. Then $(\sigma,\mu)$ is a finite homotopy string of $A(\zD^*_p)$.

(2) If $\za$ is a closed curve, then the construction of $(\sigma,\mu)$ is similar, noticing that since $\za$ is gradable, the numbers of direct and inverse strings are equal, thus $\sigma$ is a homotopy band, which is gradable.

(3) If $\za$ is an $\bpoint$-arc which has at least one endpoint $q$ from $\calp_{\bpoint}$. Then we firstly construct an infinite arc $\widetilde{\za}$ by redrawing $\za$ locally in a way such that it wraps around $q$ infinitely many times in the clockwise direction. If both endpoints of $\za$ are punctures, then we redraw $\za$ locally around both endpoints. Furthermore, we expand $f$ in a unique way to a grading $\widetilde{f}$ on $\widetilde{\za}$ such that the values of $f$ and $\widetilde{f}$ on the common intersections of $\za$ and $\wt{\za}$ with $\zD^*_p$ are the same. We associated the new graded arc $(\widetilde{\za},\widetilde{f})$ with a graded homotopy string $(\sigma,\mu)$ in a similar way that we do as above for those $\bpoint$-arcs with both endpoints in $\calm_{\bpoint}$. Then $(\sigma,\mu)$ is an infinite homotopy string of $A(\zD^*_p)$.

The above constructions give a bijection between the set of graded $\bpoint$-arcs/closed curves $(\za,f)$ on $(\cals,\calm,\zD_p^*)$ and the set of graded homotopy string/band $(\sigma, \mu)$ of $A(\zD_p^*)$.
We will denote by $\P_{(\za,f)}$/$\P_{(\za,f,\lambda,m)}$ the homotopy string/band object associated to $(\za,f)$, that is, the object arising from $(\sigma, \mu)$.

Let $n\in \mathbb{Z}$, for a graded curve $(\za,f)$, we define the \emph{shift} $f[n]$ by setting $f[n](q)=f(q)-n$ for any intersection $q\in \za\cap \zD_p^*$. Then the shift of the grading of the graded curves corresponds to the shift of the associated objects, that is, $\P_{(\za,f[n])}=\P_{(\za,f)}[n]$ and $\P_{(\za,f[n],\lambda,m)}=\P_{(\za,f,\lambda,m)}[n]$.

Furthermore, based on the results in \cite{ALP16}, it is proved in \cite[Theorem 3.3, Remark 3.8]{OPS18} that the morphisms in $\kaa$ can be interpreted as the oriented intersections of graded curves on the surface, which will be restated in the following. We start with introducing a notion of oriented intersections for general curves on the surface (compare Definition \ref{definition:oriented intersection}).

\begin{definition}\label{definition:oriented intersections}
Let $\za$ and $\zb$ be two $\bpoint$-arcs or closed curves on $(\cals,\calm)$ which intersect at a point $q$. An \emph{oriented intersection} from $\za$ to $\zb$ at $q$ is the clockwise angle from $\za$ to $\zb$ based at $q$ such that the angle is in the interior of the surface with the convention that if $q$ is a non-punctured interior point of the surface then the opposite angles are considered equivalent.
\end{definition}

Then there are three possibilities for the position of $q$: $q$ is a boundary marked point, a non-punctured interior point, or a puncture, see respectively the pictures in Figure \ref{interior intersection}. A boundary intersection gives rise to only one oriented intersection $\aaa$, see the left picture, while a non-punctured interior intersection gives rise to two oriented intersections, see the middle picture. A punctured interior intersection gives rise to infinite oriented intersections, see the right picture, where we draw the minimal oriented intersection from $\za$ to $\zb$ by $\aaa$, and draw the minimal oriented intersection from $\zb$ to $\za$ by $\bbb$. Then any oriented intersection from $\za$ to $\zb$ arising from $q$ is obtained from $\aaa$ by making a further clockwise $2m\pi$ rotation, for some $m\in \mathbb{N}$, which we will denoted by $(\aaa,m)$. In particular, we have $(\aaa,m)=\aaa$.
Similarly, the oriented intersections from $\zb$ to $\za$ arising from $q$ will be denoted by $(\bbb,m)$.

To consider the gradings over $\za$ and $\zb$, we introduce some special intersections $p_1$ and $p_2$ between them and the arcs in $\zD^*_p$.
More precisely, in the pictures of Figure \ref{interior intersection}, the green lines are segments of the arcs in $\zD_p^*$.
We denote by $p_1$ and $p_2$ the nearest intersections to $q$ in the set $\zD_p^*\cap \za$ and $\zD_p^*\cap \zb$ respectively.
In particular, to produce the intersections $p_1$ and $p_2$ when $q$ is a puncture, we replace the arcs by infinite arcs (the dotted arcs), see the right picture. Note that in all cases, the choices of $p_1$ and $p_2$ are unique, even for a non-punctured interior intersection, where the $\bpoint$-point in the polygon $\bbp$ which contains $q$ is on the right, see the middle picture.

\begin{figure}[ht]
		{
\begin{tikzpicture}[>=stealth,scale=0.3]
			\draw[thick,-] (7.95,2.5) -- (1.85,6.85);
			\draw[thick,-] (7.95,2.5) -- (1.85,-1.85);
		
			\draw (1.3,7) node {\tiny$\zb$};
            \draw [bend right,thick] (8.7,5) to (8.7,0);
%            \draw (9.25,1) node {\tiny$\rpoint$};
%			\draw (9.25,4) node {\tiny$\rpoint$};

%			\draw (3,6.5) node {\tiny$\rpoint$};
			\draw[thick,dark-green!50] (3,4.5) -- (5,6.5);

%			\draw (3,-1.5) node {\tiny$\rpoint$};
			\draw[thick,dark-green!50] (3,.5) -- (5,-1.5);

			\draw[dashed,thick,dark-green!50] (3,.5) -- (3,4.5);

			\draw[dashed,thick,dark-green!50] (5,6.5) -- (8.5,4.5);
			\draw[dashed,thick,dark-green!50] (5,-1.5) -- (8.5,.5);
			\draw (1.3,-2.2) node {\tiny$\za$};
			
			\draw (4,6.4) node {\tiny$p_2$};
			\draw (4,-1.2) node {\tiny$p_1$};
			\draw (8.6,2.5) node {\tiny$q$};
			\draw (5,2.5) node {\tiny$\bbp$};
            \draw [bend left,thick,->] (7.2,2) to (7.2,3);
			\draw (6.5,2.5) node {\tiny$\aaa$};
            \draw[thick,fill=white] (7.95,2.5) circle (0.2);
			\draw (12,0) node {};
			\end{tikzpicture}
\begin{tikzpicture}[>=stealth,scale=0.3]
			\draw[thick,-] (8,-.5) -- (0.15,6.85);
			\draw[thick,-] (8,5.5) -- (0.15,-1.85);

			\draw (-.5,7) node {\tiny$\zb$};

            \draw [bend right,thick] (8.7,5) to (8.7,0);
%            \draw (9.25,1) node {\tiny$\rpoint$};
%			\draw (9.25,4) node {\tiny$\rpoint$};

%%			\draw (3,6.5) node {\tiny$\rpoint$};
%			\draw[thick,dark-green] (7,6) -- (4,7.5);
%
%%			\draw (3,-1.5) node {\tiny$\rpoint$};
%			\draw[thick,dark-green] (7,-1) -- (4,-2.5);
			\draw[dashed,thick,dark-green!50] (0,5) -- (0,-.2);

			\draw[thick,dark-green!50] (0,5) -- (4,7.5);
			\draw[thick,dark-green!50] (0,-.2) -- (4,-2.5);

			\draw[dashed,thick,dark-green!50] (8,4) -- (4,7.5);
			\draw[dashed,thick,dark-green!50] (8,1) -- (4,-2.5);

			\draw (-.5,-2) node {\tiny$\za$};
			
			\draw (1.3,6.5) node {\tiny$p_2$};
			\draw (1.3,-1.6) node {\tiny$p_1$};
%			\draw (4.7,0) node {\tiny$p$};
			\draw (2,2.5)node {\tiny$\bbp$};

            \draw [bend left,thick,->] (4.2,2) to (4.2,3);
			\draw (3.5,2.5) node {\tiny$\aaa$};
            \draw [bend left,thick,->] (5.4,3) to (5.4,2);
			\draw (6.1,2.5) node {\tiny$\aaa$};
            \draw [bend left,thick,->] (4.2,3) to (5.4,3);
			\draw (4.7,3.7) node {\tiny$\bbb$};
            \draw [bend left,thick,->] (5.4,2) to (4.2,2);
			\draw (4.7,1.3) node {\tiny$\bbb$};

            \draw[thick,fill=white] (7.95,2.5) circle (0.2);

			\end{tikzpicture}
\begin{tikzpicture}[>=stealth,scale=0.8]

\draw[dark-green!50,thick](2,0)--(5,-1.6);
\draw[dark-green!50,thick](2,0)--(-1,-1.6);
\draw[dark-green!50,thick](2,0)--(2,-2.5);
\draw[dark-green!50,thick](2,0)--(5,1.6);
\draw[dark-green!50,thick](2,0)--(-1,1.6);
\draw[dark-green!50,thick](2,0)--(2,2.5);

\draw[thick](2,0)--(5,0);
\draw[thick](2,0)--(-1,0);
%\node [] at (0,-.6) {\tiny$\bbp$};

\node [] at (-.8,0.25) {\tiny$\za$};
\node [] at (4.8,-0.25) {\tiny$\zb$};

\node [] at (.95,.9) {\tiny$p_1$};
\node [] at (1.25,.2) {\tiny$p_2$};

            \draw[thick] (2,0) circle (.3);
            \draw [thick,->] (2.27,.1) to (2.31,0);
            \draw [thick,->] (1.73,-.1) to (1.69,0);

			\draw (2.3,.5) node {\tiny$\aaa$};
			\draw (2.3,-.5) node {\tiny$\bbb$};

\node [dark-green!50] at (-.8,1.1) {\tiny$\ell^*_1$};
\node [dark-green!50] at (1.75,2) {\tiny$\ell^*_2$};
%\node [dark-green] at (4.8,-1.15) {\tiny$\ell_i$};
\node [dark-green!50] at (-.8,-1.2) {\tiny$\ell^*_n$};

			    \draw[dotted,thick]plot [smooth,tension=1] coordinates {(0,0) (.6,0.3) (2,1.2) (3,0.8)};
			    \draw[dotted,thick]plot [smooth,tension=1] coordinates {(+4,0) (-.6+4,-0.3) (-2+4,-1.2) (.8,-.2) (1.5,.6)};
%			\draw (1.8,-.3) node {\tiny$p$};

            \draw[thick,fill=white] (2,0) circle (0.08);
			\draw (-3,0) node {};
\end{tikzpicture}}
			\caption{Oriented intersections give rise to morphisms between the associated objects in the derived category.
}\label{interior intersection}
			\end{figure}

Let $\P_{(\za,f)}$ and $\P_{(\zb,g)}$ be two objects in $\kaa$ associated with two graded curves $(\za,f)$ and $(\zb,g)$ on the surface. For convenience, in order to state the results, in the following we will denote a band object $\P_{(\za,f,\lambda,m)}$ by $\P_{(\za,f)}$ if $m=1$ and the statement does not depend on the choice of $\lambda$.
 Then the oriented intersections of $\za$ and $\zb$ give rise to morphisms in $\kaa$. More precisely,

(1) the oriented intersection $\aaa$ arising from a boundary intersection depicted in the left picture of Figure \ref{interior intersection} gives rise to a morphism:
\begin{equation}\label{equ:morph1}
\psi_\aaa: \P_{(\za,f)}\rightarrow\P_{(\zb,g)}[g(p_2)-f(p_1)];
\end{equation}

(2) the oriented intersections $\aaa$ and $\bbb$ arising from a non-punctured interior intersection depicted in the middle picture of Figure \ref{interior intersection} respectively give rise to two morphisms:
\begin{equation}\label{equ:morph2}
\psi_\aaa: \P_{(\za,f)}\rightarrow\P_{(\zb,g)}[g(p_2)-f(p_1)],
\end{equation}
\begin{equation}\label{equ:morph22}
\psi_\bbb:\P_{(\zb,g)} \rightarrow \P_{(\za,f)}[f(p_1)-g(p_2)+1];
\end{equation}

(3) the oriented intersections $(\aaa,m)$ and $(\bbb,m)$ arising from a punctured interior intersection depicted in the right picture of Figure \ref{interior intersection} respectively give rise to morphisms:
\begin{equation}\label{equ:morph3}
\psi_{(\aaa,m)}: \P_{(\za,f)}\rightarrow\P_{(\zb,g)}[g(p_2)-f(p_1)+mn],
\end{equation}
\begin{equation}\label{equ:morph33}
\psi_{(\bbb,m)}:\P_{(\zb,g)} \rightarrow \P_{(\za,f)}[f(p_1)-g(p_2)+mn],
\end{equation}
where $m\in \mathbb{N}$ and $n$ is the number of the arcs in $\zD_p^*$ connected to the puncture $\bpoint$, and we set $0\leq g(p_2)-f(p_1),f(p_1)-g(p_2)\leq n-1$ by modulo $n$ if necessary.

\begin{proposition}\label{prop:obj-in-derived-cat}\cite[Theorem 3.3]{OPS18}
The morphisms $\psi_\aaa$, $\psi_\bbb$, $\psi_{(\aaa,m)}$ and $\psi_{(\bbb,m)}$ constructed above form a basis of the $\Hom$-space $$\Hom^\bullet(\P_{(\za,f)},\P_{(\zb,g)}):=\bigoplus\limits_{n\in \mathbb{Z}}\Hom(\P_{(\za,f)},\P_{(\zb,g)}[n]),$$
unless $\za$ and $\zb$ coincide.

(1) If $\za$ and $\zb$ are the same $\bpoint$-arc, then above morphisms together with the identity map form a basis of the $\Hom$-space.

(2) If $\za$ and $\zb$ are the same closed curve, and $\P_{(\za,f)}$ and $\P_{(\zb,g)}$ have the same parameters, then $\P_{(\za,f)}\cong\P_{(\zb,g)}[f-g]$, and the above morphisms together with the identity map and a canonical map $\xi$ that appears in an Auslander-Reiten triangle
\[\tau \P_{(\za,f)}\longrightarrow E \longrightarrow \P_{(\za,f)} \s{\xi}\longrightarrow \tau \P_{(\za,f)}[1]=\P_{(\zb,g)}[f-g+1],\]
form a basis of the $\Hom$-space, where $f-g\in \mathbb{Z}$ is defined as the difference $f(p)-g(p)$ for any intersection $p$ of $\za=\zb$ with an arc in $\zD^*_p$.
\end{proposition}

\subsection{t-structures, hearts and simple-minded collections}\label{subsection: t-structures and hearts}
For a finite dimensional algebra $A$, it is well-known that the bounded derived category $\cald^b(A)$ of finite generated right $A$-module is triangle equivalent to the homotopy category $\kaa$. In the following, we will not distinguish these two categories. It is also well-known that the category $\ma$ of finite generated right $A$-module can be embedded in $\kaa$ as a full subcategory which maps a module $M$ to its projective resolution $\P_M$. Under this embedding, the extension spaces of modules coincide with the spaces of homomorphisms of the associated projective complexes, that is, we have a canonical isomorphism $\Ext^\omega(M,N)\cong \Hom(\P_M,\P_N[\omega])$ for any modules $M$ and $N$ in $\ma$, and any non-negative integer $\omega$.

In fact, $\ma$ is a typical example of a more general concept so called the heart of $\cald^b(A)$, that is, the heart of the standard t-structure $(\cald^{\leq 0},\cald^{\geq 0})$, where $\cald^{\leq 0}$ consists of complexes with
vanishing cohomologies in positive degrees, and $\cald^{\geq 0}$ consists of complexes with vanishing cohomologies in negative degrees.

In general, a \emph{$t$-structure} on a triangulated category $\calc$ with suspension functor $[1]$ (\cite{BBD82}) is a pair $(\calc^{\leq 0},\calc^{\geq 0})$ of full subcategories of $\calc$ such that
\begin{enumerate}[\rm(1)]
\item $\calc^{\leq 0}[1]\subseteq\calc^{\leq 0}$ and
$\calc^{\geq 0}[-1]\subseteq\calc^{\geq 0}$;
\item $\Hom(M,N[-1])=0$ for $M\in\calc^{\leq 0}$
and $N\in\calc^{\geq 0}$;
\item for each $M\in\calc$ there is a triangle $M'\rightarrow
M\rightarrow M''\rightarrow M'[1]$ in $\calc$ with $M'\in\calc^{\leq
0}$ and $M''\in\calc^{\geq 0}[-1]$.
\end{enumerate}

We call $\calc^{\leq 0}~\cap~\calc^{\geq 0}$ the \emph{heart} of the t-structure, which is
always abelian.
The $t$-structure $(\calc^{\leq 0},\calc^{\geq 0})$ is said to be
\emph{bounded} if
$$\bigcup_{n\in\mathbb{Z}} \calc^{\leq
0}[n]=\calc=\bigcup_{n\in\mathbb{Z}}\calc^{\geq 0}[n].$$

A t-structure is called \emph{algebraic} if it is bounded and its heart is \emph{algebraic}, that is, the heart has finitely many isomorphism classes of simple objects and each object of the heart is both artinian and noetherian. The algebraic t-structures are closely related to collections of certain special objects in the triangulated category: the simple-minded collections.

A collection $S_1,\ldots,S_n$ of objects in a triangulated category $\calc$ is said to be
\emph{simple-minded} if the following conditions hold for
$1\leq i, j\leq n:$
\begin{enumerate}[\rm(1)]
\item $\Hom(S_i,S_j[m])=0,~~\forall~m<0$,
\item {$\End(S_i)$ is a division algebra and $\Hom(S_i,S_j)$ vanishes for $i\neq j$,}
%$\Hom(S_i,S_j)=\begin{cases} K & \text{if\ }i=j,\\
                          %                 0 & \text{otherwise},
                          %                 \end{cases}$
\item $S_1,\ldots,S_n$ generate
$\calc$ (i.e. $\calc=\thick(S_1,\ldots,S_n)$).
\end{enumerate}

Note that for any algebra $A$, the collection of simple A-modules consist a simple-minded collection of $\cald^b(A)$, and the set of simple-minded collections bijectively correspond to the set of algebraic t-structures, and thus correspond to the set of algebraic hearts, see \cite{KY14}. The following proposition shows how to construct an algebraic heart from a simple-minded collection.

\begin{proposition}\cite[Proposition 5.4, Theorem 6.1]{KY14}\label{prop:simpleminded-to-t-str}
Let $S_1,\ldots,S_n$ be a simple-minded collection of
$\dba$. Let $\cald^{\leq 0}$ (respectively, $\cald^{\geq
0}$) be  the extension closure of $\{S_i[m]\mid i=1,\ldots,n,
m\geq 0\}$ (respectively, $\{S_i[m]\mid i=1,\ldots,n, m\leq 0\}$)
in $\dba$.
Then the pair $(\cald^{\leq 0},\cald^{\geq 0})$ is a bounded $t$-structure on $\dba$.

Conversely, any algebraic heart in $\dba$ arises from a simple-minded collection in the above way.
\end{proposition}

\section{A geometric model for the module category of a gentle algebra}\label{section:geo-model-module}
\subsection{Modules as curves}\label{subsection: Modules as curves}

In this section, let $(\cals,\calm)$ be a marked surface, where $\calm=\calm_{\bpoint}\cup \calp_{\bpoint}\cup \calm_{\rpoint}$ is the set of marked points on $\cals$. We will introduce simple dissections/coordinates on $(\cals,\calm)$ and associate them with gentle algebras. Then we interpret the indecomposable modules of the gentle algebras as certain curves on the marked surface.
This can be viewed as a modification of the geometric model of the module category established in \cite{BC21}. 

\begin{definition}\label{definition:addmissable dissections}
\begin{enumerate}[\rm(1)]
 \item A collection of simple $\bpoint$-arcs is called a \emph{simple dissection}, if the arcs have no interior intersections and they cut the surface into polygons each of which contains exactly one marked point from $\calm_{\rpoint}$.
 \item A collection of simple $\rpoint$-arcs is called a \emph{simple coordinate}, if the arcs have no interior intersections and they cut the surface into polygons each of which contains exactly one $\bpoint$-marked point from $\calm_{\bpoint}\cup \calp_{\bpoint}$.
\end{enumerate}
\end{definition}

Note that if $\calp_{\bpoint}$ is empty, that is there exists no puncture on the surface, then the projective dissection (projective coordinate resp.) coincides with the simple dissection (projective coordinates resp.), otherwise, they are different, see the following example.

\begin{example}\label{example:admissible dissection}
	The pictures in Figure \ref{figure:admissible dissection-pre} give examples of coordinates and dissections on a disk with one puncture.
		\begin{figure}
		\begin{center}
			\begin{tikzpicture}[scale=0.35]
				\draw[thick,fill=white] (0,0) circle (4cm);
				
				\draw[thick](0,0)to(3.5,-2);
				\draw[thick](-3.5,-2)to(0,0);
				\draw[thick](0,4)to(0,0);	
				
				\draw (0,-4) node {$\rpoint$};
				\draw (-3.5,2) node {$\rpoint$};
				\draw (3.5,2) node {$\rpoint$};
				
				\draw (0.6,2.6) node {\tiny$\ell_1$};
				\draw (1.4,-1.4) node {\tiny$\ell_2$};
				\draw (-1.4,-1.4) node {\tiny$\ell_3$};
				\draw[thick,fill=white] (0,0) circle (.2cm);
				\draw[thick,fill=white] (0,4)  circle (.2cm);
				\draw[thick,fill=white] (-3.5,-2) circle (.2cm);
				\draw[thick,fill=white] (3.5,-2) circle (.2cm);					
				\draw (0,-5.5) node {$\zD_s$};
				\draw (6,0) node {};
			\end{tikzpicture}
			\begin{tikzpicture}[scale=0.35]
				
				\draw[thick,fill=white] (0,0) circle (4cm);
				
				\draw[red,thick](-3.5,2)to(3.5,2);
				\draw[red,thick](-3.5,2)to(0,-4);
				\draw[red,thick](0,-4)to(3.5,2);		
				\draw (0,-4) node {$\rpoint$};
				\draw (-3.5,2) node {$\rpoint$};
				\draw (3.5,2) node {$\rpoint$};
				
				\draw (0,2.5) node {\tiny$\ell_1^*$};
				\draw (2.3,-1) node {\tiny$\ell_2^*$};
				\draw (-2.3,-1) node {\tiny$\ell_3^*$};		
				\draw (0,-5.5) node {$\zD^*_s$};				
				
				\draw[thick,fill=white] (0,0) circle (.2cm);
				\draw[thick,fill=white] (0,4)  circle (.2cm);
				\draw[thick,fill=white] (-3.5,-2) circle (.2cm);
				\draw[thick,fill=white] (3.5,-2) circle (.2cm);	
				\draw (6,0) node {};		
			\end{tikzpicture}
			\begin{tikzpicture}[scale=0.35]
				
				\draw[thick,fill=white] (0,0) circle (4cm);
				
				\draw[thick](-3.5,-2)to(3.5,-2);
				\draw[thick](-3.5,-2)to(0,4);
				\draw[thick](0,4)to(3.5,-2);		
				\draw (0,-4) node {$\rpoint$};
				\draw (-3.5,2) node {$\rpoint$};
				\draw (3.5,2) node {$\rpoint$};
				
				\draw (0,-2.6) node {\tiny$t(\ell_2^*)$};
				\draw (2.6,1.4) node {\tiny$t(\ell_1^*)$};
				\draw (-2.5,1.4) node {\tiny$t(\ell_3^*)$};
				\draw[thick,fill=white] (0,0) circle (.2cm);
				\draw[thick,fill=white] (0,4)  circle (.2cm);
				\draw[thick,fill=white] (-3.5,-2) circle (.2cm);
				\draw[thick,fill=white] (3.5,-2) circle (.2cm);	
				\draw (0,-5.5) node {$\zD_p=t(\zD^*_s)$};
				\draw (6,0) node {};	
			\end{tikzpicture}
			\begin{tikzpicture}[scale=0.35]
				
				\draw[thick,fill=white] (0,0) circle (4cm);
				
				\draw[dark-green,thick](0,0)to(3.5,2);
				\draw[dark-green,thick](-3.5,2)to(0,0);
				\draw[dark-green,thick](0,-4)to(0,0);	
				
				\draw (0,-4) node {$\rpoint$};
				\draw (-3.5,2) node {$\rpoint$};
				\draw (3.5,2) node {$\rpoint$};
				
				\draw (1,-2.6) node {\tiny$t(\ell_2)$};
				\draw (2.4,.6) node {\tiny$t(\ell_1)$};
				\draw (-2.4,.6) node {\tiny$t(\ell_3)$};
				\draw[thick,fill=white] (0,0) circle (.2cm);
				\draw[thick,fill=white] (0,4)  circle (.2cm);
				\draw[thick,fill=white] (-3.5,-2) circle (.2cm);
				\draw[thick,fill=white] (3.5,-2) circle (.2cm);				
				\draw (0,-5.5) node {$\zD^*_p=t(\zD_s)$};
				
			\end{tikzpicture}
		\end{center}
		\begin{center}
			\caption{Examples of dissections and coordinates on a disk with one puncture. From left to right: simple dissection, simple coordinate, projective dissection and projective coordinate. Note that the simple/projective dissection and the simple/projective coordinate are dual with each other, in the sense that an arc in the dissection intersects exactly one arc in the associated coordinate. The projective dissection/coordinate is obtained from a simple coordinate/dissection by clockwise rotating the $\rpoint$/$\bpoint$-boundary endpoints of the arcs to the next $\bpoint$/$\rpoint$-boundary points on the same boundary component.}\label{figure:admissible dissection-pre}
		\end{center}
	\end{figure}
\end{example}

We denote a simple dissection by $\zD_s$. Then similar to the case of projective dissection, for any $\zD_s$, there is a unique (up to homotopy) \emph{dual simple coordinate} $\zD^*_s$ on $(\cals,\calm)$, such that each arc in $\zD^*_s$ intersects exactly one arc of $\zD_s$. See the left two pictures in Figure \ref{figure:admissible dissection-pre} for an example.

One may associate $\zD^*_s$ with a gentle algebra $A(\zD^*_s)=kQ(\zD_s^*)/I(\zD_s^*)$, which is dual to the case for a projective coordinate given in Definition \ref{definition:gentle algebra from dissection}.
More precisely, the vertices of $Q(\zD^*_s)$ are given by the arcs in $\zD^*_s$, and the arrows are given by the \emph{minimal oriented intersections} between the arcs. However, dual to the case for the projective coordinate, here we need $\ell^*_j$ follows $\ell^*_i$ \emph{anti-clockwise} at a common endpoint for an oriented intersection $\aaa$ from $\ell^*_i$ to $\ell^*_j$, and the ideal $I(\zD_s^*)$ is generated by paths $\aaa\bbb:\ell^*_i\rightarrow \ell^*_j\rightarrow \ell^*_k$, where the common endpoint of $\ell^*_i$ and $\ell^*_j$, and the common endpoint of $\ell^*_j$ and $\ell^*_k$ that respectively gives rise to $\aaa$ and $\bbb$ are different.
See Example \ref{example:quiver of admissible dissection} for an example.

\begin{example}\label{example:quiver of admissible dissection}
The left picture in Figure \ref{figure:quiver of admissible dissection} is a marked surface with a simple coordinate, while the right picture is the associated quiver with relations.
	\begin{figure}[ht]
		\begin{center}
			\begin{tikzpicture}[scale=0.7]
				\draw[thick,fill=white] (0,0) circle (4cm);
				\draw[thick,fill=gray] (-.5,0) circle (.5cm);

				\draw[thick,red](0,-4)to(0,0);	
			    \draw[red,thick]plot [smooth,tension=1] coordinates {(0,-4) (-2.7,0) (-1.1,2) (0,0)};
			    \draw[red,thick]plot [smooth,tension=1] coordinates {(0,-4) (2.7,0) (1.1,2) (0,0)};

                \draw[bend right,thick,->](.6,-3.4)to(0,-3.2);
                \draw[bend right,thick,->](0,-3.2)to(-.6,-3.4);
                \draw[bend right,thick,->](.3,1)to(-.3,1);
                \draw[bend right,thick,->](0,-.6)to(.2,.6);
                \draw (.5,0) node {$\ddd$};
                \draw (.5,-2.8) node {$\aaa$};
                \draw (-.5,-2.8) node {$\bbb$};
                \draw (0,1.5) node {$\ccc$};

                \draw[red] (-1.4,-2) node {$1$};
                \draw[red] (2,-2) node {$3$};
                \draw[red] (.4,-2) node {$2$};
				
				\draw (0,-5.5) node {$(\cals,\calm,\zD^*_s)$};

				\draw[thick,fill=white] (-1,0) circle (.1cm);
                \draw[thick,fill=white] (0,4)  circle (.1cm);
                \draw[thick,fill=white] (1.5,0) circle (.08cm);
				\draw[thick,red,fill=red] (0,0) circle (.08cm);				
				\draw[thick,red,fill=red] (0,-4) circle (.08cm);	
			\end{tikzpicture}
			\begin{tikzpicture}[scale=0.7,>=stealth]
                \draw (-2.5,0) node {$3$};
                \draw (2.5,0) node {$2$};
				\draw (0,-3) node {$1$};
				
                \draw (0,-.4) node {$\aaa$};
                \draw (1.7,-1.5) node {$\bbb$};
				\draw (-1.7,-1.5) node {$\ccc$};
                \draw (0,1.5) node {$\ddd$};

                \draw [thick,->] (-2,0) -- (2,0);
                \draw [thick,->] (-2,-.5) -- (-.5,-2.5);
                \draw [thick,->] (2,-.5) -- (.5,-2.5);

                \draw[bend right,thick,->](2,.5)to(-2,.5);
				\draw[dotted,thick,bend left](1.5,0) to (1.8,0.6);					
                \draw[dotted,thick,bend right](-1.5,0) to (-1.7,0.6);	
			
					\draw (0,-6) node {$(Q(\zD^*_s),I(\zD^*_s))$};
					\draw (-6,-5) node {};
			\end{tikzpicture}
         \end{center}
        \begin{center}
			\caption{The left picture is a simple coordinate on a marked annulus, and the right picture is the associated quiver with relations, where we denote the relations by dotted lines.}\label{figure:quiver of admissible dissection}
		\end{center}
	\end{figure}
\end{example}
%	\begin{figure}[ht]
%		\begin{center}
%			{\begin{tikzpicture}[scale=0.5]
%					\draw (0,2.6) node {$\za^*$};
%					\draw (2.5,-1.4) node {$\zb^*$};
%					\draw (-2.5,-1.4) node {$\ell^*$};	
%					
%					\draw [thick,->] (0.3,2.2) -- (2.2,-1);
%					\draw [thick,->] (2,-1.4) -- (-2.2,-1.4);
%					\draw [thick,->] (-2.4,-1) -- (-0.4,2.2);
%					
%					\draw[dotted,thick](.3,1.8) to [out=-90,in=-90] (-.4,1.8);
%					\draw[dotted,thick](1.4,-1.2) to [out=90,in=180] (1.9,-0.6);
%					\draw[dotted,thick](-1.5,-1.2) to [out=90,in=0] (-2,-0.6);					
%					\draw (0,-3) node {\tiny$(Q(\zD^*_s),I(\zD^*_s))$};
%					
%			\end{tikzpicture}}
%		\end{center}
%		\begin{center}
%			\caption{The quiver with relations associated to the simple coordinate in Figure \ref{figure:admissible dissection}, where we denote the relations by dotted lines.}\label{figure:quiver of admissible dissection}
%		\end{center}
%	\end{figure}

Then $A(\zD^*_s)$ is a gentle algebra, and similar to the case for projective coordinates, this establishes a bijection between the set of triples $(\cals,\calm,\zD^*_s)$ of homeomorphism classes of marked surfaces together with simple coordinates and the set of isomorphism classes of gentle algebras $A(\zD^*_s)$, see also \cite[Theorem 1]{BC21}.

Note that $\zD^*_s$ cut the surface into polygons $\bbp$ each of which has exactly one marked point from $\calm_{\bpoint}$ or from $\calp_{\bpoint}$, see the pictures in Figure \ref{figure:two polygons}. These polygons will be called the \emph{polygons of $\zD^*_s$}.

\begin{figure}
 \[\scalebox{1}{
\begin{tikzpicture}[>=stealth,scale=0.7]
\draw[red,thick] (0,0)--(-1,2)--(2,3)--(5,2)--(4,0);
\draw[thick](4,0)--(0,0);
\draw[red,thick,fill=red] (0,0) circle (0.1);
\draw[red,thick,fill=red] (-1,2) circle (0.1);
\draw[red,thick,fill=red] (2,3) circle (0.1);
\draw[red,thick,fill=red] (5,2) circle (0.1);
\draw[red,thick,fill=red] (4,0) circle (0.1);

\draw[,thick,fill=white] (2,0) circle (0.1);
\node[red] at (.2,2.8) {\tiny$\ell^*_{i_2}$};
\node[red] at (5,.8) {\tiny$\ell^*_{i_m}$};
\node[red] at (-.8,.8) {\tiny$\ell^*_{i_1}$};
\node at (.8,1.5) {\tiny$\bbp$};
\node at (7,1.5) {};
\end{tikzpicture}

\begin{tikzpicture}[>=stealth,scale=0.7]
\draw[red,thick] (0,0)--(-1,2)--(2,3)--(5,2)--(4,0);
\draw[red,thick](4,0)--(0,0);
\draw[red,thick,fill=red] (0,0) circle (0.1);
\draw[red,thick,fill=red] (-1,2) circle (0.1);
\draw[red,thick,fill=red] (2,3) circle (0.1);
\draw[red,thick,fill=red] (5,2) circle (0.1);
\draw[red,thick,fill=red] (4,0) circle (0.1);

\draw[thick,fill=white] (2,1.5) circle (0.1);
\node[red] at (0.2,2.8) {\tiny$\ell^*_{i_2}$};
\node[red] at (5.2,.8) {\tiny$\ell^*_{i_{m-1}}$};
\node[red] at (2,.4) {\tiny$\ell^*_{i_m}$};
\node[red] at (-.8,.8) {\tiny$\ell^*_{i_1}$};
\node at (.8,1.5) {\tiny$\bbp$};
\end{tikzpicture}}\]
\begin{center}
\caption{Two types of polygon $\bbp$ formed by arcs in a simple coordinate and boundary segments.}\label{figure:two polygons}
\end{center}
\end{figure}

In the following we introduce a special kind of curves on $(\cals,\calm,\zD^*_s)$, which will be interpreted as indecomposable module of $A(\zD^*_s)$.
We denote a polygon $\bbp$ of $\zD^*_s$ by $(\ell^*_{i_1},\cdots,\ell^*_{i_m})$, the ordered set of arcs in $\zD^*_s$ which form $\bbp$, where the arcs are ordered clockwise and where the index $1\leq j \leq m$ is considered modulo $m$ if $\bbp$ has a puncture. For any $\ell^*_{i_j}\in \bbp$, we call $\ell^*_{i_{j-1}}$ (if exists) the \emph{predecessor} of $\ell^*_{i_{j}}$ in $\bbp$ and call $\ell^*_{i_{j+1}}$ (if exists) the \emph{successor} of $\ell^*_{i_{j}}$ in $\bbp$, see Figure \ref{figure:two polygons}.
In particular, $\ell^*_{i_{1}}$ has no predecessor and $\ell^*_{i_{m}}$ has no successor if $\bbp$ has a marked point from $\calm_{\bpoint}$, while $\ell^*_{i_{1}}$ has the predecessor $\ell^*_{i_{m}}$ and $\ell^*_{i_{m}}$ has the successor $\ell^*_{i_{1}}$ if $\bbp$ has a puncture.

{\bf Setting 1}: {\it Let $\za$ be an $\bpoint$-arc or a closed curve on a marked surface $(\cals,\calm)$. After choosing a direction of $\za$, denote by $\bbp_0, \bbp_1,\cdots, \bbp_{n+1}$ the ordered polygons of $\zD^*_s$ that successively intersect with $\za$, whenever $\za$ is a closed curve the index is considered modulo $n+1$, and in particular, we set $\bbp_0=\bbp_{n+1}$. Denote by $\ell^*_0, \ell^*_1,\cdots, \ell^*_{n}$ the ordered arcs in $\zD^*_s$ that successively intersect with $\za$ such that $\ell^*_i$ belongs to $\bbp_{i}$ and $\bbp_{i+1}$ for each $0 \leq i \leq n$. Note that $\{\ell^*_0, \ell^*_1,\cdots, \ell^*_{n}\}$ is an ordered multiset, that is, the $\ell^*_{i}$ need not be distinct.}

\begin{definition}\label{definition:zigzag arcs}
Let $\za$ be an $\bpoint$-arc or a closed curve on $(\cals,\calm,\zD^*_s)$. Under the notations in {\bf Setting 1} above,
\begin{enumerate}[\rm(1)]
	\item we call $\za$ a \emph{zigzag curve} on $(\cals,\calm,\zD^*_s)$ if in each polygon $\bbp_{i+1}, 0 \leq i \leq n+1$, $\ell^*_{i+1}$ is the predecessor or the successor of $\ell^*_{i}$. Furthermore, if $\bbp_{i+1}$ is a polygon which contains a puncture and has $m$ edges with $m\neq 2$, then we also need the puncture is not in the (unique) triangle formed by the segments of $\ell^*_{i}$, $\ell^*_{i+1}$ and $\za$, see the right picture of Figure \ref{figure:def-zigzag};
  \item we call an $\bpoint$-arc a \emph{zigzag arc} if it is a zigzag curve;
  \item we call a (primitive) closed curve a \emph{closed zigzag curve} if it is a zigzag curve.
\end{enumerate}
%\begin{enumerate}[\rm(1)]
%  \item for any $1 \leq i \leq n$, $\ell^*_{i+1}$ is the predecessor or the successor of $\ell^*_{i}$ in the polygon $P_{i+1}$.
%  \item the puncture is not in the triangle formed by the segments of $\za$, $\ell^*_{i}$ and $\ell^*_{i+1}$, if $P_{i+1}$ is a polygon contains a puncture which has at least three edges.
%\end{enumerate}
\begin{figure}
 \[\scalebox{1}{
\begin{tikzpicture}[>=stealth,scale=0.7]
\draw[red,thick] (0,0)--(-1,2)--(2,3)--(5,2)--(4,0);
\draw[thick](4,0)--(0,0);
\draw[red,thick,fill=red] (0,0) circle (0.1);
\draw[red,thick,fill=red] (-1,2) circle (0.1);
\draw[red,thick,fill=red] (2,3) circle (0.1);
\draw[red,thick,fill=red] (5,2) circle (0.1);
\draw[red,thick,fill=red] (4,0) circle (0.1);

\draw[thick,fill=white] (2,0) circle (0.1);
\node at (.8,.5) {\tiny$\bbp_{i+1}$};
\node at (.8,1.5) {\tiny$\za$};
\node[red] at (3.9,2.7) {\tiny$\ell^*_{i+1}$};
\node[red] at (0,2.7) {\tiny$\ell^*_{{i}}$};

\draw[bend right,thick](-1,2.5)to(5,2.5);
\draw[bend right,thick,->](1.4,2.8)to(2.6,2.8);
\node at (2,2.3) {\tiny$\aaa_{i+1}$};
\node at (7,1.5) {};
\end{tikzpicture}

\begin{tikzpicture}[>=stealth,scale=0.7]
\draw[red,thick] (0,0)--(-1,2)--(2,3)--(5,2)--(4,0);
\draw[red,thick](4,0)--(0,0);
\draw[red,thick,fill=red] (0,0) circle (0.1);
\draw[red,thick,fill=red] (-1,2) circle (0.1);
\draw[red,thick,fill=red] (2,3) circle (0.1);
\draw[red,thick,fill=red] (5,2) circle (0.1);
\draw[red,thick,fill=red] (4,0) circle (0.1);

\draw[thick,fill=white] (2,1) circle (0.1);
\node at (.8,.5) {\tiny$\bbp_{i+1}$};
\node at (.8,1.5) {\tiny$\za$};
\node[red] at (3.9,2.7) {\tiny$\ell^*_{i+1}$};
\node[red] at (0,2.7) {\tiny$\ell^*_{{i}}$};
\draw[bend right,thick](-1,2.5)to(5,2.5);
\draw[bend right,thick,->](1.4,2.8)to(2.6,2.8);
\node at (2,2.3) {\tiny$\aaa_{i+1}$};

\end{tikzpicture}}\]
\begin{center}
\caption{A zigzag curve $\za$ passes through a polygon $\bbp_{i+1}$ of $\zD_s^*$. If $\bbp_{i+1}$ contains a puncture and does not contain a bigon, then the puncture need not belong to the (unique) triangle formed by the segments of $\ell^*_{i}$, $\ell^*_{i+1}$ and $\za$, see the right picture. The arcs $\ell^*_{i}$ and $\ell^*_{i+1}$ may coincide, that is, $\bbp_{i+1}$ is a once punctured monogon surrounded by a $\rpoint$-arc $\ell^*_{i}=\ell^*_{i+1}$.}\label{figure:def-zigzag}
\end{center}
\end{figure}
\end{definition}

\begin{remark}
Note that the definition of the zigzag curve is independent on the choice of the direction of the curve. If $\bbp_{i+1}$ is a bigon with a puncture, then there are two triangles formed by the segments of $\ell^*_{i}$, $\ell^*_{i+1}$ and of the curve. In this case, the puncture is allowed to belong to any of these two triangles. However, different choices yield different zigzag curves which are not homotopic.
\end{remark}

{\bf Construction 1. String/band of a zigzag $\bpoint$-arc/closed curve.}

Let $\za$ be a zigzag curve on $(\cals,\calm,\zD^*_s)$, we associate a string $\omega(\za)$ of $A(\zD_s^*)$ to it in the following way.

(1) If $\za$ is a zigzag arc with notation as in {\bf Setting 1}.
Then for each $0\leq i \leq n-1$, there is an arrow in $A(\zD_s^*)$ from $\ell_i^*$ to $\ell_{i+1}^*$ arising from a corner of $\bbp_{i+1}$, which is a minimal oriented intersection and we denote it by $\aaa_{i+1}$, see the left picture in Figure \ref{figure:def-zigzag}.
We associate a walk $\omega(\za)=\omega_1\cdots\omega_n$ to $\za$, where $\omega_{i+1}=\aaa_{i+1}$ if $\za$ enter $\bbp_{i+1}$ through $\ell_i^*$ and leave through $\ell_{i+1}^*$, or $\omega_{i+1}=\aaa^{-1}_{i+1}$ if $\za$ enter $\bbp_{i+1}$ through $\ell_{i+1}^*$ and leave through $\ell_i^*$. Since $\za$ is zigzag, it is straightforward to see that $\omega(\za)$ is a string of $A(\zD_s^*)$.

(2) If $\za$ be a closed zigzag curve on $(\cals,\calm,\zD_s^*)$ with notations as in {\bf Setting 1}.
Then in a similar way, we associate a walk $\omega(\za)=\omega_1\cdots\omega_n$ to $\za$.
It is easy to see that $\omega(\za)$ is a band of $A(\zD_s^*)$, since $\za$ is a closed curve, and noticing that we need any closed curve to be primitive, so $\omega(\za)$ is not a power of any band.

\begin{example}\label{example:arcs and objects}
	The pictures in Figure \ref{figure:arcs and objects2} give examples of strings and bands arising from zigzag curves of a marked surface with the simple coordinate given in Figure \ref{figure:quiver of admissible dissection}.

	\begin{figure}[ht]
		\begin{center}
			\begin{tikzpicture}[scale=0.7,>=stealth]
				\draw[thick,fill=white] (0,0) circle (4cm);
				\draw[purple,fill=white,very thick] (-0.5,0) circle (1.4cm);				
                \draw[thick,fill=gray] (-0.5,0) circle (0.5cm);

				\draw[thick,red!50](0,-4)to(0,0);	
			    \draw[red!50,thick]plot [smooth,tension=1] coordinates {(0,-4) (-2.7,0) (-1.1,2) (0,0)};
			    \draw[red!50,thick]plot [smooth,tension=1] coordinates {(0,-4) (2.7,0) (1.1,2) (0,0)};

                \draw[bend right,thick,->,black!50](.6,-3.4)to(0,-3.2);
                \draw[bend right,thick,->,black!50](0,-3.2)to(-.6,-3.4);
                \draw[bend right,thick,->,black!50](.3,1)to(-.3,1);
                \draw[bend right,thick,->,black!50](0,-.6)to(.2,.6);
                \draw[black!50] (.5,0) node {\tiny$\ddd$};
                \draw[black!50] (.5,-2.8) node {\tiny$\aaa$};
                \draw[black!50] (-.5,-2.8) node {\tiny$\bbb$};
                \draw[black!50] (0,1.5) node {\tiny$\ccc$};

                \draw (1.2,-1.2) node {$\za'$};
                 \draw [purple](-1.6,0) node {$\za$};
                \draw[very thick]plot [smooth,tension=1] coordinates {(-1,0) (-0.5,-1) (1,0) (-.5,2) (-2.5,.5) (-.5,-2) (1.5,0)};
%
%                \draw (-1.4,-2) node {\tiny$1$};
%                \draw (2,-2) node {\tiny$3$};
%                \draw (.4,-2) node {\tiny$2$};
				\draw[thick,fill=white] (-1,0) circle (.1cm);
                \draw[thick,fill=white] (0,4)  circle (.1cm);
                \draw[thick,fill=white] (1.5,0) circle (.1cm);		
				\draw[thick,red,fill=red] (0,0) circle (.1cm);				
				\draw[thick,red,fill=red] (0,-4) circle (.1cm);						
%				\draw (0,-5.5) node {$(\cals,\calm,\zD^*_s)$};
				\draw (6,0) node {};
			\end{tikzpicture}
{\begin{tikzpicture}[scale=0.7,>=stealth]
				\draw[thick,fill=white] (0,0) circle (4cm);

                \draw[thick,fill=gray] (-0.5,0) circle (0.5cm);

                \draw[purple,fill=white,very thick] (-0.5,0) circle (2.6cm);	
				\draw[thick,red!50](0,-4)to(0,0);	
			    \draw[red!50,thick]plot [smooth,tension=1] coordinates {(0,-4) (-2.7,0) (-1.1,2) (0,0)};
			    \draw[red!50,thick]plot [smooth,tension=1] coordinates {(0,-4) (2.7,0) (1.1,2) (0,0)};

                \draw[bend right,thick,->,black!50](.6,-3.4)to(0,-3.2);
                \draw[bend right,thick,->,black!50](0,-3.2)to(-.6,-3.4);
                \draw[bend right,thick,->,black!50](.3,1)to(-.3,1);
                \draw[bend right,thick,->,black!50](0,-.6)to(.2,.6);
                \draw[very thick]plot [smooth,tension=1] coordinates {(-1,0) (-.3,-1) (1.5,0) (-.5,2) (-2,.5) (-.5,-1.3) (1,0)};

                \draw[black!50] (.5,0) node {\tiny$\ddd$};
                \draw[black!50] (.5,-2.8) node {\tiny$\aaa$};
                \draw[black!50] (-.5,-2.8) node {\tiny$\bbb$};
                \draw[black!50] (0,1.5) node {\tiny$\ccc$};

                \draw (1.2,-1) node {$\zb'$};
                 \draw [purple](-3.4,0) node {$\zb$};

%                \draw (-1.4,-2) node {\tiny$1$};
%                \draw (2,-2) node {\tiny$3$};
%                \draw (.4,-2) node {\tiny$2$};
				
%				\draw (0,-5.5) node {$(\cals,\calm,\zD^*_s)$};
				\draw[thick,fill=white] (-1,0) circle (.1cm);
                \draw[thick,fill=white] (0,4)  circle (.1cm);
                \draw[thick,fill=white] (1,0) circle (.1cm);
				\draw[thick,red,fill=red] (0,0) circle (.1cm);				
				\draw[thick,red,fill=red] (0,-4) circle (.1cm);	
			\end{tikzpicture}}
		\end{center}
		\begin{center}
			\caption{In the left picture, $\za$ is a closed zigzag curve which gives a band $\omega(\za)=\bbb\ccc^{-1}\ddd^{-1}$, while $\za'$ is a zigzag arc which gives a string $\omega(\za')=\bbb\ccc^{-1}\ddd^{-1}$. In the right picture, $\zb$ is a closed zigzag curve which gives a band $\omega(\zb)=\bbb\ccc^{-1}\aaa$, while $\zb'$ is a zigzag arc which gives a string $\omega(\zb')=\bbb\ccc^{-1}\aaa$.}\label{figure:arcs and objects2}
		\end{center}
	\end{figure}
\end{example}

\begin{definition}\label{prop-def:string and band from curves}
Let $(\cals,\calm,\zD^*_s)$ be a marked surface with a simple coordinate.
\begin{enumerate}[\rm(1)]
	\item  For a zigzag $\bpoint$-arc $\za$ on the surface, we call the string module $M_\za$ of $A(\zD_s^*)$ arising from $\omega(\za)$ the \emph{string module of $\za$}.

   \item For a zigzag closed curve $\za$ on the surface, we call the family of band modules $M_{(\za,\lambda,m)}$ of $A(\zD_s^*)$ arising from $\omega(\za)$ the \emph{band modules of $\za$}, for $\lambda\in k^*$ and $m\in \mathbb{N}$.
\end{enumerate}
	\end{definition}

Note that the string module defined in above is independent of the choice of the direction of $\za$. If we choose another direction of $\za$, then we obtain a string which is the inverse of $\omega(\za)$. So both $M_\za$ and $M_{(\za,\lambda,m)}$ are well-defined.

The following theorem is given in \cite[Theorem 2]{BC21}, up to a modification of the surface model, which establishes a bijection between the homotopy classes of zigzag curves on the surface and the isomorphism classes of (families) of indecomposable modules over the associated gentle algebra. For the convenience of the reader, we include a proof here, see Remark \ref{lem:compBC21} for more explanation of the relations between the following theorem and \cite[Theorem 2]{BC21}.

\begin{theorem}\label{theorem:main arcs and objects}
Let $(\cals,\calm,\zD_s^*)$ be a marked surface with a simple coordinate.
\begin{enumerate}[\rm(1)]
  \item The map $M: \za\mapsto M_\za$ gives a bijection between zigzag arcs on $(\cals,\calm,\zD^*_s)$ and string modules of $A(\zD^*_s)$. In particular, $M$ induces a bijection between arcs in the dual simple dissection $\zD_s$ and simple modules of $A(\zD^*_s)$.
  \item The map $M: \za\mapsto M_{(\za,\lambda,m)}$ for $\lambda\in k^*$ and $m\in \mathbb{N}$ gives a bijection between closed zigzag curves on $(\cals,\calm,\zD^*_s)$ and classes of band modules of $A(\zD^*_s)$.
\end{enumerate}
\end{theorem}
\begin{proof}
(1) Note that $M$ is injective. In fact, if we have $M_{\za_1}=M_{\za_2}$ for two zigzag arcs $\za_1$ and $\za_2$, then $\omega(\za_1)=\omega(\za_2)$, where we change the direction of $\za_2$ if necessary. So $\za_1$ and $\za_2$ pass the same sequence of polygons of $\zD_s^*$. In particular, they have the same endpoints in $\calm_{\bpoint}$. Then $\za_1$ and $\za_2$ are homotopic with respect to the boundary components and the punctures.

Now we show that $M$ is surjective. Let $M=M(\omega)$ be a string module with $\omega=\omega_1\cdots\omega_n$, where for each $1\leq i \leq n$, $\omega_i=\aaa_i$ or $\omega_i=\aaa_i^{-1}$ for some arrow $\aaa_i$ of $Q(\zD^*_s)$. Then by the definition of $Q(\zD^*_s)$, each $\aaa_i$ connects two arcs $\ell_i^*$ and $\ell_{i+1}^*$in $\zD_s^*$, such that both of them belong to a polygon $\bbp_i$ of $\zD_s^*$, and $\ell_i^*$ is the predecessor or the successor of $\ell_{i+1}^*$.
In particular, $\ell^*_1$ and $\ell_2^*$ share a common polygon $\bbp_1$. We denote by $\bbp_0$ the other polygon of $\zD_s^*$ that contains $\ell_1^*$. Similarly, let $\bbp_{n+1}$ be the other polygon of $\zD_s^*$ (rather than $\bbp_n$) which contains $\ell_n^*$.
We construct a (unique) zigzag arc $\za$ whose endpoints are the (unique) marked points or punctures in $\bbp_0$ and $\bbp_{n+1}$, and successively pass through $\aaa_i$ or $\aaa_i^{-1}$ in polygons $\bbp_i, 1 \leq i \leq n$. Then we have $M_\za=M(\omega)$.

(2) The proof is similar to the proof of the first statement.
\end{proof}

\begin{remark}\label{lem:compBC21}
The surface model of the module category of a gentle algebra established in this paper can be viewed as a small modification of the surface model given in \cite{BC21}. 
The difference is that in \cite{BC21} unmarked boundary components except those contained in monogons or digons are removed, while in this paper, they are retained and contracted to punctures. Besides that, additional marked $\bpoint$-points are added to each boundary component forming the side of a polygon in the partial triangulation given by the simple coordinate. Then a zigzag curve is called a permissible curve in \cite{BC21}, up to the modification we made above.
It is these	changes which mean that one indecomposable module corresponds to an equivalence class of curves in \cite[Theorem 2]{BC21} while in above theorem it corresponds to a unique curve. 
\end{remark}

\subsection{Simples versus projectives}\label{section:dual of simples and projectives}

In this section, we study the relations among simple coordinates, projective coordinates, simple dissections and projective dissections on a marked surface. We will show that each of them can be obtained from the other one by dual and certain twist defined in the following.

\begin{definition}\label{definition:twist}
Let $(\cals,\calm)$ be a marked surface.
\begin{enumerate}[\rm(1)]
  \item Let $\ell$ be an arc on $(\cals,\calm)$. We call an arc the \emph{twist} of $\ell$ if it is obtained from $\ell$ by clockwise rotating each $\bpoint$-endpoint ($\rpoint$-endpoint resp.) on a boundary component to the next $\rpoint$-endpoint ($\bpoint$-endpoint resp.) of this component, and keeping the $\bpoint$-endpoint which is a puncture.
  We denote the twist of $\ell$ by $t(\ell)$.
  \item Dually, we define the \emph{anti-twist} of an arc $\ell$ by anti-clockwise rotating the endpoints of $\ell$, and denote it by $t^{-1}(\ell)$.
  \item Let $\zD^*_s$ be a simple coordinate on $(\cals,\calm)$. We call $t(\zD^*_s)$ ($t^{-1}(\zD^*_s)$ resp.) the \emph{twist} (\emph{anti-twist} resp.) of $\zD^*_s$, which is the collection of the twists (anti-twists resp.) of arcs in $\zD^*_s$. Similarly we define the twist as well as the anti-twist of a simple dissection, a projective dissection and a projective coordinate.

\end{enumerate}
\end{definition}

\begin{proposition}\label{proposition:twist}
Let $(\cals,\calm)$ be a marked surface.
\begin{enumerate}[\rm(1)]
  \item The twist and the anti-twist of a simple coordinate of $(\cals,\calm)$ are both projective dissections of $(\cals,\calm)$, and vice-verse.
  \item The twist and the anti-twist of a simple dissection on $(\cals,\calm)$ are both projective coordinates on $(\cals,\calm)$, and vice-verse.
\end{enumerate}
\end{proposition}

\begin{proof}
We only prove the case for the twist of a simple coordinate $\zD^*_s$, the proof of the other cases are similar. Note that any arc $\ell^*$ in $\zD^*_s$ is a simple $\bpoint$-arc, and the twist $t(\ell^*)$ is a simple $\rpoint$-arc.
On the other hand, since we simultaneously twist the arcs in $\zD^*_s$, the new arcs in $t(\zD^*_s)$ has no interior intersections and they still cut the surface into polygons, and each of which contains exactly one $\bpoint$-puncture or $\rpoint$-boundary point. Thus $t(\zD^*_s)$ is a projective dissection.
\end{proof}

See Figure \ref{figure:admissible dissection-pre} for examples of the twists of simple/projective dissections and simple/projective coordinates.
The following theorem justifies the choice of the name projective dissections, that is, the arcs in the projective dissection $t(\zD^*_s)$ give rise to projective modules of the algebra $A(\zD^*_s)$.
\begin{theorem}\label{theorem:projectives from simple coordinates}
Let $(\cals,\calm, \zD_s^*)$ be a marked surface with a simple coordinate, where $\zD_s^*=\{\ell^*_1,\cdots,\ell^*_n\}$. Then for each $1 \leq i \leq n,$
 \begin{enumerate}[\rm(1)]
  \item the arc $t(\ell^*_i)$ is a zigzag arc which gives rise to the indecomposable projective module $P_i$ of $A(\zD^*_s)$ associated to the idempotent $\ell^*_i$, that is, we have $M_{t(\ell^*_i)}=P_i$;
  \item the arc $t^{-1}(\ell^*_i)$ is a zigzag arc which gives rise to the indecomposable injective module $I_i$ of $A(\zD^*_s)$ associated with the idempotent $\ell^*_i$, that is, we have $M_{t^{-1}(\ell^*_i)}=I_i$.
\end{enumerate}
\end{theorem}
\begin{proof}
(1) The proof directly follows from Figure \ref{figure:simple to projective1}, where the horizontal red arc is $\ell^*_i$ and the black arc is the twist $t(\ell^*_i)$.
In fact, by the definition of the relation set $I(\zD^*)$ of $A(\zD^*)$, both strings $\aaa_u^{-1}\cdots \aaa_1^{-1}$ and $\bbb_1\cdots \bbb_{u+v}$ in the picture are maximal in the sense that the compositions $\aaa_u\aaa$ and $\bbb_{u+v}\bbb$ are zero for any non-trivial arrows $\aaa$ and $\bbb$. So the string $\aaa_u^{-1}\cdots \aaa_1^{-1}\bbb_1\cdots \bbb_{u+v}$ gives rise to the projective module $P_i$, whose radical is the sum of (at most) two maximal uniserial submodules arising from $\aaa_u^{-1}\cdots \aaa_2^{-1}$ and $\bbb_2\cdots \bbb_{u+v}$.
	\begin{figure}[ht]
		\begin{center}
			\begin{tikzpicture}[scale=0.5,>=stealth]

                \draw [very thick,red] (-7.8,0) -- (7.8,0);

                \draw [thick,red!50] (-7.8,0) -- (1.5,1.4);
                \draw [thick,red!50] (-7.8,0) -- (-3,4);
                \draw [thick,red!50] (-7.8,0) -- (-6,4);
                \draw[bend right,thick,->,black!50](-3,0)to(-3.3,0.6);
                \draw[bend right,thick,->,black!50](-6.4,1.2)to(-7.1,1.5);
                \draw[black!50] (-2.2,0.3) node {\tiny$\aaa_1$};
                \draw[black!50] (-6.4,1.8) node {\tiny$\aaa_u$};
                \draw[red!50] (0,1.8) node {\tiny$\ell^*_1$};
                \draw[red!50] (-4.8,3.3) node {\tiny$\ell^*_{u-1}$};
                \draw[red!50] (-7,3.3) node {\tiny$\ell^*_{u}$};

                \draw [thick,red!50] (7.8,0) -- (-1.5,-1.4);
                \draw [thick,red!50] (7.8,0) -- (3,-4);
                \draw [thick,red!50] (7.8,0) -- (6,-4);
                \draw[bend right,thick,->,black!50](3,0)to(3.3,-0.6);
                \draw[bend right,thick,->,black!50](6.4,-1.2)to(7.1,-1.5);
                \draw[black!50] (2.2,-0.4) node {\tiny$\bbb_1$};
                \draw[black!50] (6.4,-1.8) node {\tiny$\bbb_v$};
                \draw[red!50] (0,-1.8) node {\tiny$\ell^*_{u+1}$};
                \draw[red!50] (2.7,-3) node {\tiny$\ell^*_{u+v-1}$};
                \draw[red!50] (5.5,-3) node {\tiny$\ell^*_{u+v}$};

                \draw[red] (-3,-0.8) node {$\ell^*_i$};
                \draw[] (-3,1.8) node {$t(\ell^*_i)$};
                \draw[bend left,thick](-9,4)to(-9,-4);
                \draw[bend right,thick](9,4)to(9,-4);
				\draw [very thick] (-8.5,3) -- (8.5,-3);

                \draw (-7.8,0) node {$\rpoint$};
                \draw (7.8,0) node {$\rpoint$};
                \draw[thick,fill=white] (-8.5,3) circle (.1cm);
                \draw[thick,fill=white] (8.5,-3) circle (.1cm);

			\end{tikzpicture}
         \end{center}
        \begin{center}
			\caption{The black $\bpoint$-arc is the twist $t(\ell_i^*)$ of an $\rpoint$-arc $\ell_i^*$ from a simple coordinate $\zD_s^*$, which gives rise to the indecomposable projective module of $A(\zD^*_s)$ associated to $\ell_i^*$.}\label{figure:simple to projective1}
		\end{center}
	\end{figure}

(2) The proof is dual to the proof of the first statement.
\end{proof}

{\bf Setting 2:} {\it Let $\zD_s^*$ be a simple coordinate on a marked surface $(\cals,\calm)$. We fix the following notations in the rest of this section: denote by $\zD_p$ the twist of $\zD^*_s$ and denote by $\zD_i$ the inverse twist of $\zD^*_s$. The dual of $\zD_p$ will be denoted by $\zD^*_p$ and the dual of $\zD^*_s$ will be denoted by $\zD_s$. For any curve, we always assume that it is at the minimal position with respect to both coordinates $\zD^*_s$ and $\zD^*_p$.}

Now for the `coordinates' and `dissections' we have three operations, the twist $t(-)$, the anti-twist $t^{-1}(-)$ defined above, and the dual $(-)^*$. It is easy to see that the twist and the anti-twist are inverse with each other. It is also easy to see that the dual is an involution which commutes with both the twist and the anti-twist. That is we have the following

\begin{lemma}\label{lem:comm-diag-four-opp}
The following equalities hold:
\[(-)^{**}=id(-),~ t^{-1}t(-)=tt^{-1}(-)=id(-),~ (t(-))^*=t((-)^*),~ (t^{-1}(-))^*=t^{-1}((-)^*),\]
that is, the diagram in Figure \ref{figure:dual-twist} commutes.

\begin{figure}[ht]
	\begin{center}
		\begin{tikzpicture}[scale=0.4,>=stealth]
			\draw (-4,4) node {$\zD^*_s$};
			\draw (4,4) node {$\zD_s$};
			\draw (4,-4) node {$\zD^*_p$};
			\draw (-4,-4) node {$\zD_p$};
			\draw (-10,0) node {$\zD_i$};
			\draw (0,4.5) node {$*$};
			\draw (0,-4.5) node {$*$};
			\draw (-4.7, 0) node {$t$};
			\draw (-3, 0) node {$t^{-1}$};
			\draw (4.7, 0) node {$t$};
			\draw (3.2, 0) node {$t^{-1}$};
			\draw (-7, 2.5) node {$t^{-1}$};
			\draw (-7.2, -2.7) node {$t^{-2}$};

			\draw [thick,<->] (-3.4,4) -- (3.4,4);
			\draw [thick,<->] (-3.4,-4) -- (3.4,-4);
			\draw [thick,->] (-3.8,-3.4) -- (-3.8,3.4);
			\draw [thick,<-] (-4.2,-3.4) -- (-4.2,3.4);
			\draw [thick,->] (3.8,-3.4) -- (3.8,3.4);
			\draw [thick,<-] (4.2,-3.4) -- (4.2,3.4);
			\draw [thick,->] (-4.5,3.5) -- (-9.7,0.3);
			\draw [thick,->] (-4.5,-3.5) -- (-9.5,-0.5);
		\end{tikzpicture}
	\end{center}
	\begin{center}
		\caption{The dissections and coordinates are related by a duality $*$, a twist $t$ and an anti-twist $t^{-1}$, where $\zD_i$ is a dissection whose arcs correspond the indecomposable injective modules, while the arcs in $\zD_s$ and $\zD_p$ respectively give rise to simple modules and indecomposable projective modules, see Theorems \ref{theorem:main arcs and objects} and \ref{theorem:projectives from simple coordinates}.}\label{figure:dual-twist}
	\end{center}
\end{figure}
\end{lemma}

\subsection{Module categories versus derived categories}\label{subsection:Module category vs derived category}

As recalled in subsection \ref{subsection: derived categories and derived categories}, $\zD^*_p$ plays an important role in the geometric model of the derived category of $A(\zD^*_p)$, where $\zD^*_p$ is a kind of `coordinate' of projective complexes. In this subsection we will connect the geometric model of the module category established above (as a modification of the model given in \cite{BC21}) and the geometric model of the derived category established in \cite{OPS18}.

We start with giving a direct description how $\zD^*_p=t((\zD_s^*)^*)$ is obtained directly from $\zD^*_s$, recalling that we are in {\bf Setting 2}.

Note that the simple coordinate $\zD^*_s$ cuts the surface into polygons, each of which contains a marked point from $\calm_{\bpoint}$ or from $\calp_{\bpoint}$, see the pictures in Figure \ref{figure:two polygons}. Then the pictures in Figure \ref{figure:simple to projective} depict the operation $t((-)^*)$ locally for each polygon of $\zD^*_s$, where after gluing the green segments together, we get the final arcs in $\zD^*_p$.

\begin{figure}[ht]
 \[\scalebox{1}{
\begin{tikzpicture}[>=stealth,scale=0.8]
\draw[red,thick] (0,0)--(-1,2)--(2,3)--(5,2)--(4,0);
\draw[thick](4,0)--(0,0);

\draw[black!50] (2,0)--(-2.5,1.2);
\draw[black!50] (2,0)--(-.3,3.5);
\draw[black!50] (2,0)--(-2.5+8,1);
\draw[black!50] (2,0)--(-.3+4,3.5);

\draw[dark-green,thick] (0,0)--(-2.3,2);
\draw[dark-green,thick] (0,0)--(.6,4);
\draw[dark-green,thick,bend left] (0,0)to(5,2.75);
\draw[dark-green,thick,bend left] (0,0)to(5.8,0);

\draw[thick,fill=white] (2,0) circle (0.1);
\node[red] at (1.2,2.4) {\tiny$\ell^*_{i}$};
\node[dark-green] at (.7,4.5) {\tiny$t(\ell_{i})$};
\node[black!50] at (-.5,3.7) {\tiny$\ell_{i}$};
\node[black!50] at (6,1.2) {\tiny$\ell_{m}$};
\node[red] at (4.2,1.2) {\tiny$\ell^*_{m}$};
\node[dark-green] at (6.5,0) {\tiny$t(\ell_{m})$};
\node[red] at (-.3,1.2) {\tiny$\ell^*_{1}$};
\node[dark-green] at (-2.3,2.3) {\tiny$t(\ell_{1})$};
\node[black!50] at (-2.8,1.2) {\tiny$\ell_{1}$};
%\node at (.8,1.5) {\tiny$\BBP$};
\node at (2,-.3) {\tiny$q$};
%\node at (2,3.5) {\tiny$p_2$};
\draw[red,thick,fill=red] (0,0) circle (0.08);
\draw[red,thick,fill=red] (-1,2) circle (0.08);
\draw[red,thick,fill=red] (2,3) circle (0.08);
\draw[red,thick,fill=red] (5,2) circle (0.08);
\draw[red,thick,fill=red] (4,0) circle (0.08);
\node at (7,1.5) {};
\end{tikzpicture}

\begin{tikzpicture}[>=stealth,scale=0.8]
\draw[red,thick] (0,0)--(-1,2)--(2,3)--(5,2)--(4,0);
\draw[red,thick](4,0)--(0,0);
\draw[red,thick,fill=red] (0,0) circle (0.08);
\draw[red,thick,fill=red] (-1,2) circle (0.08);
\draw[red,thick,fill=red] (2,3) circle (0.08);
\draw[red,thick,fill=red] (5,2) circle (0.08);
\draw[red,thick,fill=red] (4,0) circle (0.08);

\draw[black!50] (2,1.5)--(2,-1);
\draw[black!50] (2,1.5)--(-1.5,0.5);
\draw[black!50] (2,1.5)--(-1,3.5);
\draw[black!50] (2,1.5)--(6.5,0.5);
\draw[black!50] (2,1.5)--(4,3.5);

\draw[dark-green,thick]plot [smooth,tension=1] coordinates {(2,1.5) (1.9,0) (1.65,-1)};
\draw[dark-green,thick]plot [smooth,tension=1] coordinates {(2,1.5) (-0.5,0.99) (-1.6,0.9)};
\draw[dark-green,thick]plot [smooth,tension=1] coordinates {(2,1.55) (0.23,3) (-.5,4)};
\draw[dark-green,thick]plot [smooth,tension=1] coordinates {(2,1.5) (3.2,2.5) (5,3.5)};
\draw[dark-green,thick]plot [smooth,tension=1] coordinates {(2,1.5) (4.5,0.8) (6.5,-.2)};

%\node at (0.2,2.8) {\tiny$\ell_{2}$};
%\node at (5.2,.8) {\tiny$\ell_{m-1}$};
%\node at (2,.4) {\tiny$\ell_{m}$};
\node[red] at (0,2) {\tiny$\ell^*_{i}$};
\node[dark-green] at (.3,3.7) {\tiny$t(\ell_{i})$};
\node[black!50] at (-.4,2.8) {\tiny$\ell_{i}$};
%\node at (2,3.5) {\tiny$p_2$};
\node at (2,1.8) {\tiny$q$};
\draw[thick,fill=white] (2,1.5) circle (0.1);

%\node at (.8,1.5) {\tiny$\BBP$};
\end{tikzpicture}}\]
\begin{center}
\caption{The pictures show how the operation $t((-)^*)$ acts locally on the arcs in $\zD^*_s$, where an arc $\ell^*_i$ from $\zD^*_s$ is red, an arc $\ell_i$ from $\zD_s$ is black, and an arc $t(\ell_i)$ from $\zD^*_p$ is green. We obtain $t(\ell_i)$ by firstly drawing the dual arc $\ell_i$ of $\ell^*_i$, and then clockwise rotating the $\bpoint$-endpoint to the next $\rpoint$-point of the boundary component if the endpoint of $\ell_i$ is on a boundary component, see the left picture, or just keeping the endpoint of $\ell_i$ if it is a puncture, see the right picture. }
\label{figure:simple to projective}
\end{center}
\end{figure}

Now we prove the following

\begin{lemma}\label{lemma:iso of algebras}
We have a canonical isomorphism $A(\zD^*_s)\cong A(\zD^*_p)$, which sends the idempotent associated to $\ell^*_i$ to the idempotent associated to $t(\ell_i)$.
\end{lemma}
\begin{proof}
At first, note that there is a canonical isomorphism between quivers $Q(\zD^*_s)$ and $Q(\zD^*_p)$ which sends the vertex $\ell^*_i$ to the vertex $t(\ell_i)$.
This can be directly checked by the construction of the operation $t((-)^*)$, as locally depicted in Figure \ref{figure:simple to projective}, and the construction of the arrows in $Q(\zD^*_s)$ and $Q(\zD^*_p)$ respectively from subsection \ref{subsection: Modules as curves} and Definition \ref{definition:gentle algebra from dissection}.
%Let $\ell^*_i$ and $\ell^*_j$ be two arcs in $\zD^*_s$. the arcs$\ell_j^*$ and $\ell_i^*$ share an endpoint if and only if $t(\ell_j)$ directly follows $t(\ell_i)$ clockwise at the endpoint if and only if $\ell_j^*$ directly follows $\ell_i^*$ anti-clockwise at the corresponding endpoint. So there is an arrow from $t(\ell_i)$ to $t(\ell_j)$ if and only if there is a corresponding arrow from $\ell_i^*$ to $\ell_j^*$. 
Moreover, note that two consecutive arrows in $Q(\zD^*_p)$ arise from a common endpoint if and only if the corresponding arrows in $Q(\zD^*_s)$ comes from different endpoints. So $I(\zD^*_s)$ and $I(\zD^*_p)$ also coincide by comparing the definitions. Therefore we have a canonical isomorphism $A(\zD^*_s)\cong A(\zD^*_p)$, which sends the idempotent associated to $\ell^*_i$ to the idempotent associated to $t(\ell_i)$.
\end{proof}

Thanks to above lemma, in the following we will not distinguish $A(\zD^*_p)$ and $A(\zD^*_s)$, and denote them by $A$. In particular, we will not distinguish the idempotents associated to $t(\ell_i)$ and $\ell_i^*$, and view an arrow from $t(\ell_i)$ to $t(\ell_j)$ in $A(\zD^*_p)$ as an arrow from $\ell_i^*$ to $\ell_j^*$ in $A(\zD^*_s)$.

In the following, we will show that a zigzag curve also represents the minimal projective resolution of the associated module. We begin with introducing a canonical grading for a zigzag curve.

\begin{definition}\label{definition:canonical complex}
 Let $(\cals,\calm,\zD_p^*)$ be a marked surface with a projective coordinate.
 \begin{enumerate}[\rm(1)]
  \item For any $\bpoint$-arc $\za$ on $(\cals,\calm)$, let $f_0$ be a grading of $\za$ such that $f_0(q)\leq 0$ for any $q\in \za\cap \zD^*_p$ and such that there exists at least one $q\in \za\cap \zD^*_p$ with $f_0(q)=0$. Denote by $\P_\za$ the projective complex associated to $(\za,f_0)$. We call $f_0$ and $\P_\za$ the \emph{canonical grading} and the \emph{canonical projective complex} of $\za$ (w.r.t $\zD_p^*$) respectively.

  \item For any gradable closed curve $\za$ on $(\cals,\calm,\zD^*_p)$ and any $\lambda\in k^*, m\in \mathbb{N}$, we similarly define the  \emph{canonical grading} $f_0$ and the \emph{canonical projective complex} $\P_{(\za,\lambda,m)}$.
 \end{enumerate}
  \end{definition}

We mention that $f_0$ is unique and thus $\P_\za$ and $\P_{(\za,\lambda,m)}$ are well-defined.
The following is a main result in this subsection, whose proof will be separated into several propositions.

\begin{theorem}\label{theorem:object}
Let $(\cals,\calm)$ be a marked surface with a simple coordinate $\zD_s^*$, and let $\zD_p^*$ be the corresponding projective coordinate of $\zD_s^*$.
 \begin{enumerate}[\rm(1)]
  \item For any zigzag arc $\za$ on the surface, $\P_{\za}$ is the projective resolution of $M_{\za}$.

  \item Any zigzag closed curve $\za$ on the surface is gradable with respect to $\zD^*_p$, and  $\P_{(\za,\lambda,m)}$ is the projective resolution of $M_{(\za,\lambda,m)}$ for any $\lambda\in k^*$ and $m\in \mathbb{N}$.
 \end{enumerate}
  \end{theorem}

\begin{remark}\label{remark:coincide-objects}
The above theorem shows that a zigzag curve on $(\cals,\calm,\zD^*_s)$ does not only represent a module in the module category of $A(\zD_s^*)$, but also represents the projective resolution of this module in the derived category of $A(\zD_s^*)$. This means that for any gentle algebra, on the level of the objects, the geometric model $(\cals,\calm,\zD^*_s)$ of its module category established in above subsection can be naturally embedded into the geometric model $(\cals,\calm,\zD^*_p)$ of the derived category established in \cite{OPS18}. This fact allows us to use the derived category to study the module category, in particular, we will give a complete description of the homomorphisms as well as the extensions between two indecomposable modules in the next subsections.
\end{remark}

\begin{proposition}\label{proposition:object-simple}
For any arc $\ell$ in $\zD_s$, the canonical projective complex $\P_{\ell}$ with respect to $\zD^*_p$ is the projective resolution of the simple module $M_{\ell}$.
\end{proposition}
\begin{proof}
We start with an algebraic construction of the projective resolution of the simple module $M_\ell$, and then go back to the geometric model.
Note that the projective resolution of a simple module over a monomial algebra is well-known, see for example in \cite{GHZ85}. Here we explicitly construct the resolution for the convenience of the reader.

Denote by $P_{\ell^*}$ the projective module associated to the vertex $\ell^*$ of $Q(\zD^*_s)$.
Since $A(\zD_s^*)$ is a gentle algebra, the radical $rad(P_{\ell^*})$ of $P_{\ell^*}$ is a direct sum of at most two uniserial submodules $U_{i_0}$ and $U_{j_0}$, where $U_{i_0}$ and $U_{j_0}$ maybe zero. We assume both of them are non-zero in the following, while the proof of other cases are similar. So we have a short exact sequence
\[0\longrightarrow U_{i_0}\oplus U_{j_0} \s{
%\begin{pmatrix}
%        \begin{smallmatrix}
%            a_0 \\
%            b_0
%        \end{smallmatrix}
%    \end{pmatrix}
}\longrightarrow P_{\ell^*} \s{\pi_0}\longrightarrow M_{\ell}\longrightarrow 0,\]
where $\pi_0$ is the canonical projection.

Let $\ell^*_{i_0}$ be the (unique) vertex of $Q(\zD_s^*)$ corresponding to the top of $U_{i_0}$, which is a simple module since $U_{i_0}$ is uniserial. Then the canonical injection from $U_{i_0}$ to $P_{\ell^*}$ is given by an arrow $\aaa_0$ from $\ell^*$ to $\ell^*_{i_0}$. Denote by $P_{i_0}$ the projective module associated to $\ell^*_{i_0}$. Then $\aaa_0$ induces a map from $P_{i_0}$ to $P_{\ell^*}$, which is in fact the composition of the canonical projection from $P_{i_0}$ to $U_{i_0}$ and the canonical injection from $U_{i_0}$ to $P_{\ell^*}$. We still use $\aaa_0$ to denote this map. On the other hand, by using a similar discussion for $U_{j_0}$, we have a projective module $P_{j_0}$ as well as a canonical map $\bbb_0$ from $P_{j_0}$ to $P_{\ell^*}$. Then we have the first two terms in the projective resolution of $M_\ell$, which corresponds to the homotopy string $\aaa_0^{-1}\bbb_0$ in $A(\zD_s^*)$:
$$\xymatrix@C=0.7cm{P_{i_0}\oplus P_{j_0} \ar[r]^{ \begin{pmatrix}
        \begin{smallmatrix}
            \aaa_0 \\
            \bbb_0
        \end{smallmatrix}
    \end{pmatrix}} & P_{\ell^*}\ar[r]^{\pi_0} & M_\ell\ar[r] & 0}.$$

Note that again by the second condition in the definition of a gentle algebra, we have $rad(P_{i_0})=U_{i_1}\oplus rad(U_{i_0})$ for some uniserial module $U_{i_1}$. If $U_{i_1}=0$, then $rad(P_{i_0})=rad(U_{i_0})$ and thus $P_{i_0}=U_{i_0}$. In this case $\aaa_0$ is an injection and our construction is already done for $U_{i_0}$ side, and we continue the construction from $U_{j_0}$.

Now we assume $U_{i_1}$ is non-zero.
Let $\ell^*_{i_1}$ be the vertex corresponding to the top of $U_{i_1}$.
Similar to above discussion, there is an arrow $\aaa_1$ from $\ell^*_{i_0}$ to $\ell^*_{i_1}$ which gives rise to a canonical map $\aaa_1$ from $P_{i_1}$ to $P_{i_0}$.
Then we have $\aaa_0\aaa_1=0$, otherwise, $U_{i_1}$ is a submodule of $U_{i_0}$, which contradicts to the equality $rad(P_{i_0})=U_{i_1}\oplus rad(U_{i_0})$.

Without loss of generality, for $U_{j_0}$, we also have a submodule $U_{j_1}$ as well as arrows $\bbb_0$ and $\bbb_1$ which have similar property as for $U_{i_0}$.
Then we have a homotopy string $\aaa_1^{-1}\aaa_0^{-1}\bbb_0\bbb_1$ of $A(\zD_s^*)$ with $\aaa_0\aaa_1=\bbb_0\bbb_1=0$, which gives rise to the first three terms in the projective resolution of $M_\ell$:
$$\xymatrix@C=0.7cm{P_{i_1}\oplus P_{j_1} \ar[r]^{ \begin{pmatrix}
        \begin{smallmatrix}
            \aaa_1 & 0 \\
            0 & \bbb_1
        \end{smallmatrix}
    \end{pmatrix}} & P_{i_0}\oplus P_{j_0} \ar[r]^{ \begin{pmatrix}
        \begin{smallmatrix}
            \aaa_0 \\
            \bbb_0
        \end{smallmatrix}
    \end{pmatrix}} & P_{\ell^*}\ar[r]^{\pi_0} & M_\ell\ar[r] & 0}.$$

Continue this process, if it stops in finite steps, then we get a maximal homotopy string $\sigma=\aaa_m^{-1}\cdots \aaa_0^{-1}\bbb_0\cdots \bbb_n$ such that $\aaa_u\aaa_{u+1}=\bbb_v\bbb_{v+1}=0$ for any $0 \leq u \leq m-1, 0 \leq v \leq n-1$. Then $\sigma$ gives rise to a projective complex $\P(\sigma)$ which is exactly the projective resolution of $M_\ell$.
If the process does not stop in finitely many steps, then we get an infinite homotopy string $\sigma$, which is in fact eventually periodic by the structure of an infinite homotopy string given in subsection \ref{subsection: Derived categories of gentle algebras}.

We go back to the geometric model $(\cals,\calm,\zD_p^*)$.
Set $\P_\ell:=\P_{(\ell,f_0)}$. To prove that $\P_\ell$ is the projective resolution of $M_\ell$, we have to show that $\P_\ell=\P(\sigma)$.
To do this, it is enough to show the claim that
$\sigma$ is exactly the homotopy string $\sigma(\ell)$ associated to the complex $\P_\ell$,
noticing that then the gradings automatically coincide, since both $\P(\sigma)$ and $\P_\ell$ are projective resolutions of some modules.

Now we prove the claim, where we consider the case that one endpoint of $\ell$ is a marked point on the boundary and the other one is a puncture, see the picture in Figure \ref{figure:projective resolution of simple}.
The other cases can be proved similarly.
At first, we label the arrows $\aaa_u$ and $\bbb_v$ with $0 \leq u \leq m-1$ and $0 \leq v \leq n-1$, which appear in $\sigma$, as oriented intersections, that is, the black arrows in the picture.
Then we locally draw the projective coordinate $\zD^*_p$, i.e. the green arcs, which is obtained from $\zD_s^*$ by using the operator $t((-)^*)$, cf. Figure \ref{figure:simple to projective}.
In particular, the oriented intersections between the arcs in $\zD^*_p$ are the green arrows, where the labels inherit the ones in $\zD_s^*$.
Then it is clear from the picture that the homotopy string $\sigma(\ell)$ associated to $\ell$ is $\sigma$.

\begin{figure}[ht]
 \[\scalebox{1}{
\begin{tikzpicture}[>=stealth,scale=1.2]
\draw[red!50,thick] (-4,1)--(-2,2)--(0,1)--(0,1)--(2,2)--(4,1)--(4,-1)--(2,-2)--(0,-1)--(-1,-2)--(-3,-2)--(-4,-1);
\draw[thick](-4,1)--(-4,-1);
\draw[red!50,thick](0,1)--(0,-1);

\draw[very thick](-4,0)--(2,0);

\draw[dark-green,thick](-4,1)--(2,0);
\draw[dark-green,thick](-4,1)--(0,-2);
\draw[dark-green,thick](-4,1)--(-1.5,-2.5);
\draw[dark-green,thick,bend left](-4,1)to(-4,-2);

\draw[dark-green,thick](2,0)--(0.5,2.2);
\draw[dark-green,thick](2,0)--(3.5,2.2);
\draw[dark-green,thick](2,0)--(0.5,-2.2);
\draw[dark-green,thick](2,0)--(3.5,-2.2);
\draw[dark-green,thick](2,0)--(4.5,0);

\draw[dark-green,thick,bend left,->](1.8,-0.3)to(1.6,0.05);
\node [dark-green] at (1.4,-.2) {\tiny$\bbb_n$};
\draw[dark-green,thick,bend left,->](1.6,0.05)to(1.8,0.3);
\node [dark-green] at (1.4,.3) {\tiny$\bbb_0$};
\draw[thick,bend right,->,black!50](0,0.7)to(0.3,1.15);
\node[black!50] at (.5,.8) {\tiny$\bbb_0$};
\draw[thick,bend left,<-,black!50](0,-0.7)to(0.3,-1.15);
\node[black!50] at (.5,-.8) {\tiny$\bbb_n$};

\draw[thick,bend right,->,black!50](0,-0.7)to(-0.25,-1.25);
\node[black!50] at (-.4,-.9) {\tiny$\aaa_0$};
\draw[thick,bend right,->,black!50](-2.7,-2)to(-3.2,-1.8);
\node[black!50] at (-2.8,-1.65) {\tiny$\aaa_m$};
\draw[thick,bend left,dark-green,->](-2.7,0.8)to(-2.9,0.2);
\node [dark-green] at (-2.45,0.4) {\tiny$\aaa_0$};
\draw[thick,bend left,dark-green,->](-3.15,-.2)to(-3.55,-.4);
\node [dark-green] at (-3.2,-.6) {\tiny$\aaa_m$};

\node at (-1,-.2) {\tiny$\ell$};
\node [dark-green] at (-1,0.8) {\tiny$t(\ell)$};
\node [red!50] at (0.3,-.3) {\tiny$\ell^*$};
			    \draw[dotted,very thick]plot [smooth,tension=1] coordinates {(0,0) (.6,0.3) (2,1.2) (3,0) (1.8,-.8) (1,0) (1.5,0.7) (2,0.7)};

\node [red!50] at (-.1,-1.5) {\tiny$\ell^*_{i_0}$};
\node [dark-green] at (-1.4,-1.35) {\tiny$t(\ell_{i_0})$};
\node [red!50] at (1.3,2) {\tiny$\ell^*_{j_0}$};
\node [dark-green] at (0.35,1.7) {\tiny$t(\ell_{j_0})$};

\draw[thick,fill=white] (-4,0) circle (0.06);
\draw[thick,fill=white] (2,0) circle (0.06);
\draw[red,thick,fill=red] (-4,1) circle (0.06);
\draw[red,thick,fill=red] (-4,-1) circle (0.06);
\draw[red,thick,fill=red] (-2,2) circle (0.06);
\draw[red,thick,fill=red] (0,1) circle (0.06);
\draw[red,thick,fill=red] (0,1) circle (0.06);
\draw[red,thick,fill=red] (2,2) circle (0.06);
\draw[red,thick,fill=red] (4,1) circle (0.06);
\draw[red,thick,fill=red] (4,-1) circle (0.06);
\draw[red,thick,fill=red] (2,-2) circle (0.06);
\draw[red,thick,fill=red] (0,-1) circle (0.06);
\draw[red,thick,fill=red] (-1,-2) circle (0.06);
\draw[red,thick,fill=red] (-3,-2) circle (0.06);
\end{tikzpicture}}\]
        \begin{center}
			\caption{An arc $\ell$ in $\zD_s$ represents a simple module with respect to the geometric model $(\cals,\calm,\zD_s^*)$, as well the projective resolution of this simple module with respect to the geometric model $(\cals,\calm,\zD_p^*)$. Since one endpoint of $\ell$ is a puncture, to get $\P_\ell$, one should redraw $\ell$ locally in a way such that it wraps around the puncture infinitely many times in the clockwise direction, see the dotted line. The projective resolution of $M_\ell$ is given by the homotopy string $\sigma=\aaa_m^{-1}\cdots \aaa_0^{-1}(\bbb_0\cdots \bbb_n)^\infty$ (see the proof of Proposition \ref{proposition:object-simple}), which is exactly the homotopy string $\sigma(\ell)$ given by the arc $\ell$, with respect to $\zD^*_p$.}
\label{figure:projective resolution of simple}
\end{center}
\end{figure}
\end{proof}

\begin{proposition}\label{proposition:object-any arc}
%Let $\zD_s^*$ be a simple coordinate of a marked surface $(\cals,\calm)$.
For any zigzag $\bpoint$-arc $\za$ on $(\cals,\calm,\zD_s^*)$, $\P_{\za}$ is the projective resolution of $M_{\za}$.
\end{proposition}
\begin{proof}
Denote by $\sigma(\za)$ the string of $A(\zD^*_s)$ associated to $\za$. We use induction on the length $n$ of $\sigma(\za)$ to prove the statement.
By Proposition \ref{proposition:object-simple}, the statement holds for the case when $n=0$.

Now assume that the statement holds for $\za$ with the length of $\sigma(\za)$ smaller than $n$. Suppose $\sigma(\za)$ starts at a vertex $\ell^*$ of $Q(\zD^*_s)$. We construct a new string obtained from $\sigma(\za)$ by deleting the first arrow $\aaa$ or inverse arrow $\aaa^{-1}$ starting at $\ell^*$. Denote by $\zb$ the zigzag arc associated to the new string, which we denote by $\sigma(\zb)$. Let $q$ be the intersection of $\ell$ and $\zb$, which is a boundary marked point or a puncture. Note that in both cases, $\za$ is obtained by smoothing $\zb$ and $\ell$ at $q$, see Figure \ref{figure:projective resolution of arc} for an example when $q$ is a puncture and $\sigma(\za)$ starts at an inverse arrow $\aaa^{-1}$.

\begin{figure}[ht]
 \[\scalebox{1}{
\begin{tikzpicture}[>=stealth,scale=1]
\draw[red!50,thick] (-4,1)--(-2,2)--(0,1)--(0,1)--(2,2)--(4,1)--(4,-1)--(0,-1)--(-1,-2)--(-3,-2)--(-4,-1);
\draw[thick,black](-4,1)--(-4,-1);
\draw[red!50,thick](0,1)--(0,-1);

\draw[very thick](-4,0)--(2,0);

\node at (1,.2) {\tiny$\ell$};
\node at (2.2,0.1) {\tiny$q$};
\node at (1.6,-.6) {\tiny$\za$};
\node at (3,-.5) {\tiny$\zb$};
\draw[very thick]plot [smooth,tension=1] coordinates {(-4,0) (1,-.6) (4.5,-2)};
\draw[very thick](2,0)--(4.5,-1.9);

\draw[thick,bend right,->](.3,-1)to(0,-.8);
\node [] at (.3,-.7) {\tiny$\aaa$};
\node [red] at (0.2,.3) {\tiny$\ell^*$};
\draw[thick,fill=white] (-4,0) circle (0.06);
\draw[thick,fill=white] (2,0) circle (0.06);
\draw[red,thick,fill=red] (-4,1) circle (0.06);
\draw[red,thick,fill=red] (-4,-1) circle (0.06);
\draw[red,thick,fill=red] (-2,2) circle (0.06);
\draw[red,thick,fill=red] (0,1) circle (0.06);
\draw[red,thick,fill=red] (0,1) circle (0.06);
\draw[red,thick,fill=red] (2,2) circle (0.06);
\draw[red,thick,fill=red] (4,1) circle (0.06);
\draw[red,thick,fill=red] (4,-1) circle (0.06);
\draw[red,thick,fill=red] (0,-1) circle (0.06);
\draw[red,thick,fill=red] (-1,-2) circle (0.06);
\draw[red,thick,fill=red] (-3,-2) circle (0.06);
\end{tikzpicture}}\]
        \begin{center}
			\caption{An arc $\za$ is the smoothing of $\zb$ and $\ell$ at the intersection $q$, where the string $\sigma(\za)$ associated to $\za$ starts at $\aaa^{-1}$, and the string $\sigma(\zb)$ associated to $\zb$ is obtained from $\sigma(\za)$ by deleting $\aaa^{-1}$.}
\label{figure:projective resolution of arc}
\end{center}
\end{figure}
Such smoothing gives rise to a short exact sequence
\[0\longrightarrow M_\zb \longrightarrow M_\za \longrightarrow M_\ell \longrightarrow 0,\]
if $\aaa$ starts at $\ell^*$, or a short exact sequence
\[0\longrightarrow M_\ell \longrightarrow M_\za \longrightarrow M_\zb \longrightarrow 0,\]
if $\aaa$ ends at $\ell^*$. Since the length of $\sigma(\zb)$ is $n-1$, by the assumption, $\P_\zb$ is the projective resolution of $M_\zb$. On the other hand, $\P_\ell$ is the projective resolution of $M_\ell$. Denote by $\P$ the projective resolution of $\za$, then we have a distinguished triangle
\[\P_\zb \longrightarrow \P \longrightarrow \P_\ell \longrightarrow \P_\zb[1]\]
or
\[\P_\ell \longrightarrow \P \longrightarrow \P_\zb \longrightarrow \P_\ell[1]\]
in the derived category $\kaa$ of $A(\zD_s^*)$, which are respectively the lifts of above two short exact sequences in the module category.

On the other hand, it follows from the description of the mapping cones in $\kaa$ stated in \cite[Theorem 4.1]{OPS18} that the oriented intersection at $q$ gives rise to above distinguished triangles, and the $\bpoint$-arc associated to $\P$ is obtained from smoothing $\zb$ and $\ell$ at $q$, which is exactly the arc $\za$. Thus $\P_\za=\P$, which is the projective resolution of $M_\za$.
\end{proof}

\begin{proposition}\label{proposition:object-closed curve}
%Let $\zD_s^*$ be a simple coordinate of a marked surface $(\cals,\calm)$.
Any zigzag closed curve $\za$ on $(\cals,\calm,\zD^*_s)$ is gradable with respect to $\zD^*_p$, and $\P_{(\za,\lambda,m)}$ is the projective resolution of $M_{(\za,\lambda,m)}$ for any $\lambda\in k^*$ and $m\in \mathbb{N}$.
\end{proposition}
\begin{proof}

Since for different choices of $\lambda\in k^*$ and $m\in \mathbb{N}$, the projective resolutions of $M_{(\za,\lambda,m)}$ arises from the same homotopy band, see for example \cite[Proposition 1 (2)]{HS05}.
Thus we will fix a $\lambda\in k^*$ and assume that $m=1$.

Let $\sigma$ be the band associated to $\za$, then $\sigma=\sigma_1\sigma^{-1}_2\cdots\sigma_{2n-1}\sigma^{-1}_{2n}$, where $\sigma_i, 1\leqslant i \leqslant 2n$, are direct strings. 
Note that $\sigma$ alternates between direct and inverse homotopy letters.
Thus $\sigma$ can also be viewed as a homotopy band,  which is
denoted by $\overline{\sigma}$ to avoid ambiguity, and it gives rise to the projective resolution  $\P$ of $M_{(\za,\lambda,1)}$, see
\cite[Corollary 2.12]{CPS21}.
%\begin{center}
%	\begin{tikzpicture}[scale=.5,xscale=3,yscale=3,ar/.style={->,thick}]
%		\draw(-.7,0)node(v2){$\bigoplus\limits_{i=1}^{n} P_{s(\sigma_{2i-1})}$}
%		(1,0)node(v1){$\bigoplus\limits_{i=1}^{n} P_{t(\sigma_{2i})}$,};
%		\draw[ar](v2)edge(v1);
%		\draw[black]
%		($(v2)!.5!(v1)$) node[above] {\tiny$d$};
%	\end{tikzpicture}
%\end{center}
%where we denote by $t(\sigma_i)$ and $s(\sigma_i)$ the target and the source of $\sigma_i$ respectively.
To prove that $\za$ is gradable and $\P_{(\za,\lambda,1)}=\P$, it is enough to prove that $\overline{\sigma}$ is the homotopy band associated to $\za$.
Assume that $\sigma_i=\aaa_{i,1}\cdots\aaa_{i,r_i}$, where $\aaa_{i,j}$ is an arrow from $\ell^*_{i,j-1}$ to $\ell^*_{i,j}$.
Since $\sigma$ is the band associated to $\za$, $\za$ successively intersects the arcs $\ell^*_{i,j}$, see the local configuration in Figure \ref{figure:band}, recalling the construction of a band from a closed curve given in {\bf Construction 1}. Therefore $\za$ successively intersects two arcs $t(\ell_{i,0})$ and $t(\ell_{i,r_i})$ in $\zD^*_p$, recall that $\zD^*_p$ is obtained from $\zD^*_s$ by using the operator $t(-)^*$, see the pictures in Figure \ref{figure:simple to projective}.
Finally, by the construction of the homotopy band associated to a closed curve given in the beginning of subsection \ref{subsection: derived categories and derived categories}, the homotopy band of $\za$ is exactly $\overline{\sigma}$, which consists of stings $\sigma_i:t(\ell_{i,0})\rightarrow t(\ell_{i,r_i})$ or the inverse strings $\sigma^{-1}_i:t(\ell_{i,r_i})\rightarrow t(\ell_{i,0})$.

\begin{figure}
	\[\scalebox{1}{
		\begin{tikzpicture}[>=stealth,scale=0.8]
\draw[thick](-2,0)--(2,0);
	
\draw[thick,red!50] (-6,0)--(-7,3.8)--(0,0)--(7,3.8)--(6,0);
			\draw[thick,red!50] (-4,5)--(0,0)--(4,5);
			\draw[thick,dashed,red!50] (0,0)--(0,6);

			\draw[very thick,dark-green]plot [smooth,tension=1] coordinates {(5.5,4.5) (3.5,4) (2,5.5)};
			\draw[very thick,dark-green,dashed]plot [smooth,tension=1] coordinates {(2,5.5) (0,4.5) (-2,5.5)};
			\draw[very thick,dark-green]plot [smooth,tension=1] coordinates {(-5.5,4.5) (-3.5,4) (-2,5.5)};
			
			\node[] at (6,1.7) {\tiny$\za$} ;
			\node[red!50] at (-7,1.2) {\tiny$\ell_{i+1,0}^*$} ;		
			\node[red!50] at (-3,1.2) {\tiny$\ell_{i,r_{i-1}}^*$} ;
			\node[dark-green] at (4.2,2.8) {\tiny$t(\ell_{i,0})$};
%			\node[red!50] at (-1.6,2.8) {\tiny$\ell_{i,r_{i-1}-1}$};
%			\node[red!50] at (1.9,2.8) {\tiny$\ell_{i,1}$};
			\node [dark-green] at (-4.2,2.8){\tiny$t(\ell_{i,r_{i}})$};
			\node [red!50] at (3,1.2) {\tiny$\ell_{i,0}^*$};
			\node [red!50] at (7.3,1.2) {\tiny$\ell_{i-1,r_{i-1}}^*$};

\draw[dark-green,thick,bend left,->](5.3,3.8)to(4.93,4.2);
\node [dark-green] at (4.7,3.6) {\tiny$\aaa_{i,1}$};
\draw[dark-green,thick,bend left,->](-4.93,4.2)to(-5.3,3.8);
\node [dark-green] at (-4.7,3.6) {\tiny$\aaa_{i,r_{i}}$};
\draw[red!50,thick,bend right,->](6.5,3.5)to(6.8,3.1);
\node [red!50] at (6,2.7) {\tiny$\aaa_{i-1,r_{i-1}}$};
\draw[red!50,thick,bend right,->](-6.8,3.1)to(-6.5,3.5);
\node [red!50] at (-6,2.8) {\tiny$\aaa_{i+1,1}$};

\draw[dark-green,very thick](4.5,1)--(5.5,4.5);	
\draw[dark-green,very thick](-4.5,1)--(-5.5,4.5);		
\draw[very thick](-8,2)--(8,2);	

\draw[red!50,thick,bend right,->](-.6,.75)to(-.86,.45);
\node [red!50] at (-1.05,.83) {\tiny$\aaa_{i,r_i}$};										
\draw[red!50,thick,bend right,->](.86,.44)to(.55,.65);
\node [red!50] at (1.07,.8) {\tiny$\aaa_{i,1}$};

\draw[red,thick,fill=red] (0,0) circle (0.08);
\draw[red,thick,fill=red] (6,0) circle (0.08);
\draw[red,thick,fill=red] (-6,0) circle (0.08);
%\draw[thick,fill=white] (-3,0) circle (0.08);
%\draw[thick,fill=white] (3,0) circle (0.08);
\draw[red,thick,fill=red] (4,5) circle (0.08);
\draw[red,thick,fill=red] (-4,5) circle (0.08);
\draw[red,thick,fill=red] (0,6) circle (0.08);
\draw[red,thick,fill=red] (-7,3.8) circle (0.08);
\draw[red,thick,fill=red] (7,3.8) circle (0.08);

\draw[red,thick,fill=red] (2,5.5) circle (0.08);
\draw[red,thick,fill=red] (-2,5.5) circle (0.08);
\draw[red,thick,fill=red] (5.5,4.5) circle (0.08);
\draw[red,thick,fill=red] (-5.5,4.5) circle (0.08);
\draw[red,thick,fill=red] (4.5,1) circle (0.08);
\draw[red,thick,fill=red] (-4.5,1) circle (0.08);
	\end{tikzpicture}}\]
	\begin{center}
		\caption{The homotopy band associated to a closed curve $\za$ is $\overline{\sigma}$, which consists of stings $\sigma_i=\aaa_{i,1}\cdots\aaa_{i,r_i}:t(\ell_{i,0})\rightarrow t(\ell_{i,r_i})$ or the inverse strings $\sigma^{-1}_i$.}\label{figure:band}
	\end{center}
\end{figure}

\end{proof}
%\begin{figure}[ht]
% \[\scalebox{1}{
%\begin{tikzpicture}[>=stealth,scale=1]
%
%\draw[red!50,thick](0,0)--(-1.5,3);
%\draw[red!50,thick](0,0)--(1.5,3);
%\draw[thick,bend right,->](0.2,0.35)to(-0.2,0.35);
%\node [] at (0,.6) {\tiny$\aaa$};
%\draw[thick,blue] (0,0) circle (1);
%  \node [blue] at (.8,0) {\tiny$\za$};
% \draw[thick]plot [smooth,tension=1] coordinates {(0,3) (-1.5,0) (0,-1.5) (1.3,0) (0,1.5) (-2,0) (0,-2) (2,0) (0,3)};
%   \node [] at (1.6,0) {\tiny$\zb$};
%\draw[thick,purple]plot [smooth,tension=1] coordinates {(0,3) (-2.5,0) (0,-2.5) (2.5,0) (0,3)};
%  \node [purple] at (2.8,0) {\tiny$\zg$};
%  \node [] at (-1.4,1.4) {\tiny$q$};
%  \node [] at (-1.17,1.2) {\tiny$\bullet$};
%  \draw[thick,fill=white] (0,3) circle (0.06);
%\draw[red,thick,fill=red] (0,0) circle (0.06);
%\end{tikzpicture}}\]
%        \begin{center}
%			\caption{For the zigzag curves $\zg$, $\za$ and $\zb$, the strings associated to them are respectively $\omega$, $\omega \aaa$ and $\omega \aaa \omega$. After smoothing $\zb$ at the canonical intersection $q$, we obtain $\za$ and $\zg$. The marked $\bpoint$-point may on the boundary or in the interior (i.e. a puncture).}
%\label{figure:object-band}
%\end{center}
%\end{figure}

\emph{Proof of Theorem \ref{theorem:object}:} This follows from Propositions \ref{proposition:object-any arc} and \ref{proposition:object-closed curve}.

\subsection{Morphisms and extensions as intersections}\label{subsection:morph}

In this subsection, we interpret the intersections of zigzag curves on the marked surface as the morphisms and extensions in the module category of the associated gentle algebra. We will fix $(\cals,\calm,\zD^*_s)$ to be a marked surface with a simple coordinate, and let $A$ be the associated gentle algebra. When we say zigzag curves we mean the zigzag curves on $(\cals,\calm)$ with respect to $\zD^*_s$.

For a zigzag arc or a zigzag closed curve $\za$, we still use $f_0$ to denote the \emph{canonical grading} such that $\P_{(\za,f_0)}$ or $\P_{(\za,f_0,\lambda,m)}$ is the projective resolution of $M_\za$ or $M_{(\za,\lambda,m)}$ respectively.
Let's start with the definition of the weight for an oriented intersection at boundary.

\begin{definition}\label{definition: weight}
Let $\za$ be a zigzag arc with an endpoint $q$ in $\calm_{\bpoint}$, where $q$ belongs to a polygon $\bbp=\{\ell^*_1,\ell^*_2,\cdots,\ell^*_n\}$ of $\zD^*_s$ with the arcs labeled clockwise, see Figure \ref{figure:weight and co-weight}.
Assume that starting from $q$, $\ell^*_t$ is the first arc that $\za$ intersects, for some $1\leq t \leq n$. We call $n-t$ the \emph{weight} of $\za$ at $q$, which is denoted by $w_q(\za)$, and we call $t-1$ the \emph{co-weight} of $\za$ at $q$, which is denoted by $cw_q(\za)$.
\begin{figure}
	\begin{center}
	\begin{tikzpicture}[>=stealth,scale=1]
				\draw[red!50,thick] (1,0)--(3,1)--(3,3)--(1,4);
				\draw[red!50,thick,dashed] (-1,4)--(1,4);
				\draw[red!50,thick] (-1,0)--(-3,1)--(-3,3)--(-1,4);

				\draw[very thick] (-2.5,4)--(0,0)--(2.5,4);
				\draw[thick] (-2,0)--(2,0);
				
				\draw[thick,fill=white] (0,0) circle (0.06);
				\draw (0,-.3) node {\tiny$q$};
				\node[red!50] at (-1.5,3.5) {\tiny$\ell^*_{t}$};
				\node[red!50] at (1.5,3.5) {\tiny$\ell^*_{t+\omega}$};
				\node at (-1.5,2) {\tiny$\za$};
				\node at (1.5,2) {\tiny$\zb$};
				\node[red!50] at (-2.5,1) {\tiny$\ell^*_{1}$};
				\node[red!50] at (2.5,.5) {\tiny$\ell^*_{n}$};
				\node at (0,3) {\tiny$\bbp$};
\draw[thick](0,0)to(-3.5,.8);
\node [] at (-1.2,.55) {\tiny$t^{-1}(\ell_1^*)$};	

\draw[thick,bend left,->](-.2,.3)to(.2,.3);
\node [] at (0,.6) {\tiny$\aaa$};	
\draw[red,thick,fill=red] (1,0) circle (0.06);
\draw[red,thick,fill=red] (-1,0) circle (0.06);
				\draw[red,thick,fill=red] (3,1) circle (0.06);
				\draw[red,thick,fill=red] (3,3) circle (0.06);
				\draw[red,thick,fill=red] (1,4) circle (0.06);
				\draw[red,thick,fill=red] (1,0) circle (0.06);
\draw[red,thick,fill=red] (-3,1) circle (0.06);
\draw[red,thick,fill=red] (-3,3) circle (0.06);
\draw[red,thick,fill=red] (-1,4) circle (0.06);
\end{tikzpicture}
\end{center}
\begin{center}
			\caption{For a zigzag arc $\za$ with endpoint $q$ which intersects $\ell^*_t$, the weight  $w_q(\za)$ of $\za$ at $q$ equals $n-t$. The weight $\omega(\aaa)$ of an oriented intersection $\aaa$ from $\za$ to $\zb$ is defined as $w_q(\za)-w_q(\zb)$, which equals $\omega$.}\label{figure:weight and co-weight}
\end{center}
\end{figure}
\end{definition}

\begin{definition}\label{definition:weighted intersections1}
	Let $\za$ and $\zb$ be two zigzag arcs which share a common endpoint $q$ in $\calm_{\bpoint}$. A \emph{weighted-oriented-intersection} from $\za$ to $\zb$ is an oriented intersection $\aaa$ from $\za$ to $\zb$ which arises from $q$, with \emph{weight} $w(\aaa)=w_q(\za)-w_q(\zb)$.
\end{definition}

Note that by Definition \ref{definition:oriented intersections}, there is an oriented intersection $\aaa$ from $\za$ to $\zb$ implies that $\zb$ follows $\za$ clockwise at $q$, cf. Figure \ref{figure:weight and co-weight}, and thus $w_q(\za)\geq w_q(\zb)$ by our convention used in Definition \ref{definition: weight}.
Therefore the weight of an oriented intersection arising from a boundary intersection is always non-negative.

The following proposition explicitly describes the minimal projective resolution $\P_\za$ for any zigzag arc $\za$ with endpoints on the boundary. In the following we write the complex as a line, where the maps between the projective modules arise from the associated arrows in the quiver. After putting the projective modules with the same degree together, we get the final projective resolution. Note that this result can be viewed as a geometric explanation of \cite[Thoerem 2.8, Corollary 2,12]{CPS21}, where the minimal projective resolution of any indecomposable module over a gentle algebra is given, by using combinatorics of strings.
We also mention that in \cite{LGH22}, the authors also use the surface model to construct the projective resolution of a module on a gentle algebra, where they use the geometric model established in \cite{BC21}. 
\begin{proposition}\label{prop:proj-res}
Let $\za$ be a zigzag $\bpoint$-arc with endpoint $p$ and $q$ on the boundary. 	
Let $\bbp=\{\ell^*_1,\ell^*_2,\cdots,\ell^*_n\}$ and $\bbp'=\{\iota^*_1,\iota^*_2,\cdots,\iota^*_m\}$ be the polygons of $\zD^*_s$ which contains $p$ and $q$ respectively.
We denote by $\sigma=\sigma_1\sigma_2\cdots\sigma_r$ the string associated to $\za$, and by $\nu_i^*$ the endpoint of $\sigma_i$ for each $1\leqslant i \leqslant r$. 

(1) If $\sigma_1$ is direct and $\sigma_r$ is inverse, then $\P_\za$ is given by

\begin{center}
\begin{tikzpicture}[scale=.5,xscale=3,yscale=2,ar/.style={->,thick}]
		\draw(-6.1,0)node(v1){$P_{t(\ell_n)}$}
		(-5,0)node(v2){$\cdots$}
		(-3.7,0)node(v3){$P_{t(\ell_{\omega_p(\za)})}$}
		(-2.3,0)node(v4){$P_{t(\nu_1)}$}
		(-1,0)node(v5){$P_{t(\nu_2)}$}
		(.3,0)node(v6){$P_{t(\nu_3)}$}
		(1.5,0)node(v7){$\cdots$};		
		\draw[ar](v1)edge(v2)
		(v2)edge(v3)
		(v4)edge(v3)
		(v4)edge(v5)
		(v6)edge(v5)
    	(v6)edge(v7);
	\end{tikzpicture}
	
	\begin{tikzpicture}[scale=.5,xscale=3,yscale=2,ar/.style={->,thick}]
	\draw(-6.1,0)node(v1){$\cdots$}
	(-4.8,0)node(v2){$P_{t(\nu_{m-1})}$}
	(-3.4,0)node(v3){$P_{t(\nu_m)}$}
	(-2,0)node(v4){$P_{t(\iota_{\omega_q(\za)})}$}
	(-.7,0)node(v5){$\cdots$}
	(.5,0)node(v6){$P_{t(\iota_m)}$};		
	\draw[ar](v1)edge(v2)
	(v3)edge(v2)
	(v3)edge(v4)
	(v5)edge(v4)
	(v6)edge(v5);
\end{tikzpicture}
\end{center}

(2) If $\sigma_1$ and $\sigma_r$ are both direct, then $\P_\za$ is given by
\begin{center}
	\begin{tikzpicture}[scale=.5,xscale=3,yscale=2,ar/.style={->,thick}]
		\draw(-6.1,0)node(v1){$P_{t(\ell_n)}$}
		(-5,0)node(v2){$\cdots$}
		(-3.7,0)node(v3){$P_{t(\ell_{\omega_p(\za)})}$}
		(-2.3,0)node(v4){$P_{t(\nu_1)}$}
		(-1,0)node(v5){$P_{t(\nu_2)}$}
		(.3,0)node(v6){$P_{t(\nu_3)}$}
		(1.5,0)node(v7){$\cdots$};		
		\draw[ar](v1)edge(v2)
		(v2)edge(v3)
		(v4)edge(v3)
		(v4)edge(v5)
		(v6)edge(v5)
		(v6)edge(v7);
	\end{tikzpicture}
	
	\begin{tikzpicture}[scale=.5,xscale=3,yscale=2,ar/.style={->,thick}]
		\draw(-6.1,0)node(v1){$\cdots$}
		(-4.8,0)node(v2){$P_{t(\nu_{m-1})}$}
		(-3.4,0)node(v3){$P_{t(\nu_m)}$}
		(-2,0)node(v4){$P_{t(\iota_{\omega_q(\za)-1})}$}
		(-.7,0)node(v5){$\cdots$}
		(.5,0)node(v6){$P_{t(\iota_m)}$};		
		\draw[ar](v2)edge(v1)
		(v2)edge(v3)
		(v4)edge(v3)
		(v5)edge(v4)
		(v6)edge(v5);
	\end{tikzpicture}
\end{center}
	
(3) If $\sigma_1$ and $\sigma_r$ are both inverse, then $\P_\za$ is given by

\begin{center}
	\begin{tikzpicture}[scale=.5,xscale=3,yscale=2,ar/.style={->,thick}]
		\draw(-6.1,0)node(v1){$P_{t(\ell_n)}$}
		(-5,0)node(v2){$\cdots$}
		(-3.7,0)node(v3){$P_{t(\ell_{\omega_p(\za)-1})}$}
		(-2.3,0)node(v4){$P_{t(\nu_1)}$}
		(-1,0)node(v5){$P_{t(\nu_2)}$}
		(.3,0)node(v6){$P_{t(\nu_3)}$}
		(1.5,0)node(v7){$\cdots$};		
		\draw[ar](v1)edge(v2)
		(v2)edge(v3)
		(v3)edge(v4)
		(v5)edge(v4)
		(v5)edge(v6)
		(v7)edge(v6);
	\end{tikzpicture}
	
	\begin{tikzpicture}[scale=.5,xscale=3,yscale=2,ar/.style={->,thick}]
	\draw(-6.1,0)node(v1){$\cdots$}
	(-4.8,0)node(v2){$P_{t(\nu_{m-1})}$}
	(-3.4,0)node(v3){$P_{t(\nu_m)}$}
	(-2,0)node(v4){$P_{t(\iota_{\omega_q(\za)})}$}
	(-.7,0)node(v5){$\cdots$}
	(.5,0)node(v6){$P_{t(\iota_m)}$};		
	\draw[ar](v1)edge(v2)
	(v3)edge(v2)
	(v3)edge(v4)
	(v5)edge(v4)
	(v6)edge(v5);
\end{tikzpicture}
\end{center}

(4) If $\sigma_1$ is inverse and $\sigma_r$ is direct, then $\P_\za$ is given by
\begin{center}
	\begin{tikzpicture}[scale=.5,xscale=3,yscale=2,ar/.style={->,thick}]
	\draw(-6.1,0)node(v1){$P_{t(\ell_n)}$}
	(-5,0)node(v2){$\cdots$}
	(-3.7,0)node(v3){$P_{t(\ell_{\omega_p(\za)-1})}$}
	(-2.3,0)node(v4){$P_{t(\nu_1)}$}
	(-1,0)node(v5){$P_{t(\nu_2)}$}
	(.3,0)node(v6){$P_{t(\nu_3)}$}
	(1.5,0)node(v7){$\cdots$};		
	\draw[ar](v1)edge(v2)
	(v2)edge(v3)
	(v3)edge(v4)
	(v5)edge(v4)
	(v5)edge(v6)
	(v7)edge(v6);
\end{tikzpicture}

	\begin{tikzpicture}[scale=.5,xscale=3,yscale=2,ar/.style={->,thick}]
	\draw(-6.1,0)node(v1){$\cdots$}
	(-4.8,0)node(v2){$P_{t(\nu_{m-1})}$}
	(-3.4,0)node(v3){$P_{t(\nu_m)}$}
	(-2,0)node(v4){$P_{t(\iota_{\omega_q(\za)-1})}$}
	(-.7,0)node(v5){$\cdots$}
	(.5,0)node(v6){$P_{t(\iota_m)}$};		
	\draw[ar](v2)edge(v1)
	(v2)edge(v3)
	(v4)edge(v3)
	(v5)edge(v4)
	(v6)edge(v5);
\end{tikzpicture}
\end{center}
\end{proposition}
\begin{proof}
By Proposition \ref{proposition:object-any arc}, the projective resolution $\P_\za$ is given by the complex associated to $\za$ with respect to the projective coordinate $\zD_p^*$. When $\sigma_1$ is a direct string, the left part of $\P_\za$ is as follows, which can be seen directly from the picture in Figure \ref{fig:proj-res1}:
\begin{center}
	\begin{tikzpicture}[scale=.5,xscale=3,yscale=2,ar/.style={->,thick}]
		\draw(-6.1,0)node(v1){$P_{t(\ell_n)}$}
		(-5,0)node(v2){$\cdots$}
		(-3.7,0)node(v3){$P_{t(\ell_{\omega_p(\za)})}$}
		(-2.3,0)node(v4){$P_{t(\nu_1)}$}
		(-1,0)node(v5){$P_{t(\nu_2)}$}
		(.3,0)node(v6){$P_{t(\nu_3)}$}
		(1.5,0)node(v7){$\cdots$};		
		\draw[ar](v1)edge(v2)
		(v2)edge(v3)
		(v4)edge(v3)
		(v4)edge(v5)
		(v6)edge(v5)
		(v6)edge(v7);
	\end{tikzpicture}
\end{center}	
\begin{figure}[ht]\centering
	\begin{center}
		\begin{tikzpicture}[scale=.6,>=stealth]
			\draw[thick](-12.5+.5,3.5)to(-12.5+.5,-3.5);	
			\draw[red!50,thick](-12.5+3,3.5)to(-15+3,3.5);
			\draw[red!50,thick](-12.5+3,-3.5)to(-15+3,-3.5);
			
			\draw[red!50,thick](-9.7+2,1.6)to(-12.5+3,3.5);
			\draw[red!50,thick](-9.7+2,-1.6)to(-12.5+3,-3.5);
			\draw[red!50,thick](-9.7+2,-1.6)to(-9.7+2,1.6);

			\draw[red!50,thick](-9.7+2,1.6)to(-7+2,-3);
			\draw[red!50](-9.7+2,1.6)to(-8.5+2,-3);
			\draw[red!50](-9.7+2+2,3.5)to(-9+2+2,-3);
			\draw[thick,red!50](-9.7+2+4,3.5)to(-9+2+2,-3);
			
			\draw[red!50] (-8.5+1.7,1) node {\tiny$\nu_1^*$};
			\draw[red!50] (-8.5+4.8,1) node {\tiny$\nu_2^*$};
			\draw (-9,-.5) node {$\za$};
			\draw[red!50] (-8.7+2.8,-2.5) node {$\cdots$};
			\draw[red!50] (-8.7+4.2,3.2) node {$\cdots$};
			
			\draw[bend right,thick,->](-8.15,2)to(-9+2,.3);
			\draw (-11+2.8,1) node {\tiny$\sigma_1$};
			
			\draw[thick,dark-green](-12.5+.5,3.5)to(-10.5,-4.5);
			\draw[dark-green] (-10.5,-5) node {\tiny$t(\ell)=t(\ell_n)$};
			
			\draw[thick,dark-green](-12.5+.5,3.5)to(-8.5,-4.5);
			%				\draw[dark-green] (-8,-5) node {\tiny$t(\ell_{n-1})$};
			
			\draw[bend left,thick,dark-green](-12.5+.5,3.5)to(-12.5+5.2,-2.5);
			\draw[dark-green] (-9.7,.8) node {\tiny$t(\ell_{\omega_p(\za)})$};

			\draw[thick,dashed,dark-green](-12.5+6.4,-2)to(-12.5+5.2,-2.5);     				\draw[thick,dashed,dark-green](-12.5+6,3)to(-12.5+7.5,2.6); 
			\draw[thick,dark-green](-12.5+7.5,2.6)to(-12.5+9,-2.6);       	       
			
			\draw[thick,dark-green](-12.5+6,3)to(-12.5+6.4,-2);
			\draw[dark-green] (-8.5+1.5,2) node {\tiny$t(\nu_1)$};  
			\draw[dark-green] (-8.5+5.5,-2) node {\tiny$t(\nu_2)$};                  
			%%%%%%
			\draw[red,fill=red] (-7+2,-3) circle (0.1);
			\draw[red,fill=red] (-9.7+2,-1.6) circle (0.1);
			\draw[red,fill=red] (-9.7+2,1.6) circle (0.1);
			\draw[red,fill=red] (-12.5+.5,3.5) circle (0.1);
			\draw[red,fill=red] (-12.5+3,3.5) circle (0.1);
			\draw[red,fill=red] (-12.5+3,-3.5) circle (0.1);
			\draw[red,fill=red] (-12.5+.5,-3.5) circle (0.1);		
			\draw[red,fill=red] (-12.5+5.2,-2.45) circle(0.1);
			\draw[red,fill=red] (-12.5+6.4,-2) circle(0.1);	
			\draw[red,fill=red] (-12.5+6,3)	 circle(0.1);
			\draw[red,fill=red] (-12.5+7.5,2.6)	 circle(0.1);					
			\draw[red,fill=red] (-12.5+9,-2.6)	 circle(0.1);				
			%%%%%%%%%%%%%%%%%%%%%%%%%%%%%%%%%%%%%%%%%%%%%%%%%%%%%%%%%%%%%%%%%%%%%%%%%%%
			
			\draw[very thick](-12,0)to(-2,0);
			\draw[thick,fill=white] (-12,0) circle (0.1);
			\draw (-12.5,0) node {\tiny$p$};
		\end{tikzpicture}
	\end{center}
	\caption{If $\sigma_1$ is a direct string, then the left part of $\P_\za$ is obtained by successively connecting the projective modules associated to the arcs in $\zD_p^*$ that $\za$ intersects: $t(\ell_n)$, $t(\ell_{n-1})$, $\cdots$, $t(\ell_{\omega_p(\za)})$, $t(\nu_1)$, $t(\nu_2)$, $\cdots$.}\label{fig:proj-res1}
\end{figure}
When $\sigma_1$ is an inverse string, the left part of $\P_\za$ is as follows, which can be seen directly from the picture in Figure \ref{fig:proj-res2}, noticing that in this case, $P_{t(\ell_{\omega_p(\za)})}$ does not appear in the complex, since $\za$ does not intersect $t(\ell_{\omega_p(\za)})$:
\begin{center}
	\begin{tikzpicture}[scale=.5,xscale=3,yscale=2,ar/.style={->,thick}]
		\draw(-6.1,0)node(v1){$P_{t(\ell_n)}$}
		(-5,0)node(v2){$\cdots$}
		(-3.7,0)node(v3){$P_{t(\ell_{\omega_p(\za)-1})}$}
		(-2.3,0)node(v4){$P_{t(\nu_1)}$}
		(-1,0)node(v5){$P_{t(\nu_2)}$}
		(.3,0)node(v6){$P_{t(\nu_3)}$}
		(1.5,0)node(v7){$\cdots$};		
		\draw[ar](v1)edge(v2)
		(v2)edge(v3)
		(v3)edge(v4)
		(v5)edge(v4)
		(v5)edge(v6)
		(v7)edge(v6);
	\end{tikzpicture}
\end{center}
At last, after combining the left part and the right part, we get the whole complex $\P_\za$ in the statement.		
	
	\begin{figure}[ht]\centering
		\begin{center}
			\begin{tikzpicture}[scale=0.6,>=stealth]
				\draw[thick](-12.5+.5,3.5)to(-12.5+.5,-3.5);	
				
				\draw[red!50,thick](-12.5+3,3.5)to(-15+3,3.5);
				\draw[red!50,thick](-12.5+3,-3.5)to(-15+3,-3.5);
				
				\draw[red!50,thick](-9.7+2,1.6)to(-12.5+3,3.5);
				\draw[red!50,thick](-9.7+2,-1.6)to(-12.5+3,-3.5);
				\draw[red!50,thick](-9.7+2,-1.6)to(-9.7+2,1.6);

				\draw[red!50,thick](-9.7+2,-1.6)to(-7+2,3);
				\draw[red!50](-9.7+2,-1.6)to(-9+2.7,3);
				
				\draw[thick,red!50](-9.7+2+4,-4)to(-9+2+2,3);
				\draw[red!50](-9.7+4,-4)to(-9+2+2,3);
				
				\draw[red!50] (-8.5+1.7,-1) node {\tiny$\nu_1^*$};
				\draw (-10.5,-.5) node {$\za$};
				\draw[red!50] (-8.7+3,2.7) node {$\cdots$};
				\draw[red!50] (-8.7+4,-3) node {$\cdots$};				
				\draw[bend right,thick,->](-9+2,-.3)to(-8.15,-2);
				\draw (-10.7+2.2,-1) node {\tiny$\sigma^{-1}_1$};

				\draw[thick,dark-green](-12.5+.5,3.5)to(-10.5,-4.5);
				\draw[dark-green] (-10.5,-5) node {\tiny$t(\ell)=t(\ell_n)$};
				
				\draw[thick,dark-green](-12.5+.5,3.5)to(-7.5,-4.5);
				\draw[dark-green] (-7.5,-5) node {\tiny$t(\ell_{\omega_p(\za)-1})$};
				
				\draw[thick,dark-green](-12.5+.5,3.5)to(-12.5+5.2,.8);
				\draw[dark-green] (-12.5+3.4,.9) node {\tiny$t(\ell_{\omega_p(\za)})$};

				\draw[thick,dashed,dark-green](-12.5+5.2,.8)to(-12.7+6.5,2);             
\draw[thick,dashed,dark-green](-12.5+6,-3)to(-12.5+7.5,-2.6); 
\draw[thick,dark-green](-12.5+7.5,-2.6)to(-12.5+9,2.6);  				
				\draw[thick,dark-green](-12.7+6.5,2)to(-12.5+6,-3);
				\draw[dark-green] (-8.5+1.35,-2.3) node {\tiny$t(\nu_1)$};   
				\draw[dark-green] (-8.5+5.5,1.5) node {\tiny$t(\nu_2)$};   				
				%%%%%%
				\draw[red,fill=red] (-7+2,3) circle (0.1);
				\draw[red,fill=red] (-9.7+2,-1.6) circle (0.1);
				\draw[red,fill=red] (-9.7+2,1.6) circle (0.1);
				\draw[red,fill=red] (-12.5+.5,3.5) circle (0.1);
				\draw[red,fill=red] (-12.5+3,3.5) circle (0.1);
				\draw[red,fill=red] (-12.5+3,-3.5) circle (0.1);
				\draw[red,fill=red] (-12.5+.5,-3.5) circle (0.1);			
				\draw[red,fill=red] (-12.7+6.5,2) circle (0.1);
				\draw[red,fill=red] (-12.5+6,-3) circle (0.1);	                
				\draw[red,fill=red] (-12.5+5.2,.8) circle (0.1);
				
            \draw[red,fill=red] (-12.5+7.5,-2.6)	 circle(0.1);	
            \draw[red,fill=red] (-12.5+9,2.6)	 circle(0.1);	
				%%%%%%%%%%%%%%%%%%%%%%%%%%%%%%%%%%%%%%%%%%%%%%%%%%%%%%%%%%%%%%%%%%%%%%%%%%%
				
				\draw[very thick](-12,0)to(-2,0);
				\draw[thick,fill=white] (-12,0) circle (0.1);
				\draw (-12.5,0) node {\tiny$p$};			
			\end{tikzpicture}
%			\begin{tikzpicture}[>=stealth,scale=.8]
%				\draw[red!50,thick] (0,-1)--(0,1)--(2,2)--(4,1)--(4,-1)--(7,1.5);
%				\draw[red!50] (4,-1)--(5.5,1.5);
%				\draw[red!50] (8,-1)--(7,1.5);
%				\draw[thick](0,-1)--(4,-1);
%				\draw[thick,dark-green](0,-1)--(4.5,.8);
%				\draw[thick,dashed,dark-green](5.5,.8)--(4.5,.8);			
%				\draw[thick,dark-green](5.5,.8)--(6.3,-1);				
%				\node at (2,-1.5) {\tiny$q$};
%				
%				\draw[red!50] (6,1) node {$\cdots$};
%				
%				\draw[very thick](2,-1)to(8,1);
%				\draw[] (7,.3) node {$\za$};
%				
%				\draw[dark-green](6.3,-1.4) node {\tiny$t(\ell)=t(\nu)$};   
%				\draw[red!50] (5,-.5) node {\tiny$\nu^*$};	
%				
%				
%				\draw[bend right,thick,->](4.6,-.5)to(3.5,-1);
%				\draw (3.6,-.1) node {\tiny$\sigma^{-1}_1$};
%				\draw[dark-green] (1.5,0) node {\tiny$t(\ell_n)$};
%				
%				\draw[thick,fill=white] (2,-1) circle (0.06);
%				\draw[red,thick,fill=red] (0,1) circle (0.06);
%				\draw[red,thick,fill=red] (0,1) circle (0.06);
%				\draw[red,thick,fill=red] (2,2) circle (0.06);
%				\draw[red,thick,fill=red] (4,1) circle (0.06);
%				\draw[red,thick,fill=red] (4,-1) circle (0.06);
%				\draw[red,thick,fill=red] (0,-1) circle (0.06);
%				\draw[red,thick,fill=red] (4.5,.8) circle (0.06);
%				\draw[red,thick,fill=red] (5.5,.8) circle (0.06);
%				\draw[red,thick,fill=red] (6.3,-1) circle (0.06);			\draw[red,thick,fill=red] (7,1.5) circle (0.06);	
%				\draw[] (-1.5,-3.5) node {};		
%			\end{tikzpicture}
		\end{center}
		\caption{If $\sigma_1$ is an inverse string, then the left part of $\P_\za$ is given by successively connecting the projective modules associated to the arcs in $\zD_p^*$ that $\za$ intersects: $t(\ell_n)$, $t(\ell_{n-1})$, $\cdots$, $t(\ell_{\omega_p(\za)-1})$, $t(\nu_1)$, $t(\nu_2)$, $\cdots$. Note that in this case $\za$ does not intersect $t(\ell_{\omega_p(\za)})$.}\label{fig:proj-res2}
	\end{figure}
\end{proof}
The above proposition derives the following
\begin{corollary}\label{lem:proj-res2}
	Let $\za$ be a zigzag $\bpoint$-arc with endpoint $p$ on the boundary. Assume that starting from $p$, $t(\ell)$ is the first arc in $\zD_p^*$ that $\za$ intersects, where we denote the intersection point by $\rho$, then we have $f_0(\rho)=-\omega_p(\za)$.
\end{corollary}
\begin{proof}
We still denote by $\sigma=\sigma_1\sigma_2\cdots\sigma_m$ the string associated to $\za$, and by $\nu_1^*$ the endpoint of $\sigma_1$. Without loss generality, we assume that $\sigma_1$ is the string corresponding the part of $\za$ starting from $p$.  If $\sigma_1$ is a direct string, then $t(\ell_n)$ is the first arc in $\zD_p^*$ that $\za$ intersects, that is $t(\ell)=t(\ell_n)$, see the picture in Figure \ref{fig:proj-res1}. In this case, since $\P_\za$ is the projective resolution of $M_\za$, it follows from the description of $\P_\za$ given in Proposition \ref{prop:proj-res} that $P_{t(\nu)}$ is located in the $(-1)$-th position of $\P_{(\za,f_0)}$, and thus we have $f_0(t(\nu))=-1$. Therefore $f_0(t(\ell))=-\omega_q(\za)$, which again follows from Proposition \ref{prop:proj-res}.

Now let $\sigma_1$ be an inverse string. 
If $\omega_q(\za)\geq 1$, then $t(\ell_n)$ is the first arc in $\zD_p^*$ that $\za$ intersects, that is $t(\ell)=t(\ell_n)$, see the picture in Figure \ref{fig:proj-res2}. Therefore, similarly we have $f_0(t(\ell))=-\omega_q(\za)$. 
If $\omega_q(\za)=0$, then $t(\nu_1)$ is the first arc in $\zD_p^*$ that $\za$ intersects, that is $t(\ell)=t(\nu_1)$. Since in this case $P_{t(\nu_1)}$ is located in the $0$-th position of $\P_{(\za,f_0)}$, we have $f_0(t(\ell))=0$, which again equals $-\omega_q(\za)$.
%
%The proof has done.
\end{proof}

\begin{proposition-definition}\label{definition:extension in interior}
Let $\za$ and $\zb$ be two zigzag curves which intersect at a non-puncture interior point $q$ on the surface locally depicted as follows, where $\aaa$ and $\bbb$ are the oriented intersections associated to $q$:
\begin{center}
\begin{tikzpicture}[>=stealth,scale=0.5]
            \draw (-.5,7) node {\tiny$\zb$};

            \draw [bend right,thick] (8.7,5) to (8.7,0);
%            \draw (9.25,1) node {\tiny$\rpoint$};
%			\draw (9.25,4) node {\tiny$\rpoint$};

%%			\draw (3,6.5) node {\tiny$\rpoint$};
%			\draw[thick,dark-green] (7,6) -- (4,7.5);
%
%%			\draw (3,-1.5) node {\tiny$\rpoint$};
%			\draw[thick,dark-green] (7,-1) -- (4,-2.5);
			\draw[dashed,thick,dark-green!50] (0,5) -- (0,-.2);

			\draw[thick,dark-green!50] (0,5) -- (4,7.5);
			\draw[thick,dark-green!50] (0,-.2) -- (4,-2.5);

			\draw[dashed,thick,dark-green!50] (8,4) -- (4,7.5);
			\draw[dashed,thick,dark-green!50] (8,1) -- (4,-2.5);

			\draw (-.5,-2) node {\tiny$\za$};
			
			\draw (1.3,6.5) node {\tiny$p_2$};
			\draw[dark-green!50] (3.4,6.2) node {\tiny$t(\ell_2)$};
			\draw (1.3,-1.6) node {\tiny$p_1$};
			\draw[dark-green!50] (3.4,-1.2) node {\tiny$t(\ell_1)$};
			\draw (5.2,2.5) node {\tiny$q$};
			\draw (2.5,2.5)node {\tiny$\bbp$};

            \draw [bend left,thick,->] (4.2,2) to (4.2,3);
			\draw (3.5,2.5) node {\tiny$\aaa$};
            \draw [bend left,thick,->] (5.4,3) to (5.4,2);
			\draw (6.1,2.5) node {\tiny$\aaa$};
            \draw [bend left,thick,->] (4.2,3) to (5.4,3);
			\draw (4.7,3.7) node {\tiny$\bbb$};
            \draw [bend left,thick,->] (5.4,2) to (4.2,2);
			\draw (4.7,1.3) node {\tiny$\bbb$};

            \draw[thick,fill=white] (7.95,2.5) circle (0.2);
			\draw[very thick] (8,-.5) -- (0.15,6.85);
			\draw[very thick] (8,5.5) -- (0.15,-1.85);
			\end{tikzpicture}
\end{center}
Then we have $f_0(p_1)=g_0(p_2)$ or $f_0(p_1)=g_0(p_2)+1$, where $p_1$ and $p_2$ are the intersection points between the arcs in $\zD_p^*$ and $\za$ and $\zb$ respectively which nearest to the point $q$, and $f_0$ and $g_0$ are respectively the canonical gradings associated to $\za$ and $\zb$, as recalled at the beginning of subsection \ref{subsection:morph}. 
 \begin{enumerate}[\rm(1)]
  \item if $f_0(p_1)=g_0(p_2)$, then we call $\aaa$ a \emph{weighted-oriented-intersection} from $\za$ to $\zb$ with weight $w(\aaa)=0$, and call $\bbb$ a \emph{weighted-oriented-intersection} from $\zb$ to $\za$ with weight $w(\bbb)=1$;

  \item if $f_0(p_1)=g_0(p_2)+1$, then we call $\aaa$ a \emph{weighted-oriented-intersection} from $\za$ to $\zb$  with weight $w(\aaa)=1$, and call $\bbb$ a \emph{weighted-oriented-intersection} from $\zb$ to $\za$ with weight $w(\bbb)=0$.
 \end{enumerate}
  \end{proposition-definition}
\begin{proof}
%By the descriptions of $\P_\za=\P_{(\za,f_0)}$ and $\P_\zb=\P_{(\zb,g_0)}$ given in Theorem \ref{thm:main-projective-res}, both $f_0(p_1)$ and $g_0(p_2)$ belongs to $\{0,-1\}$.
By the equality \eqref{equ:morph2}, there is a morphism from $\P_{(\za,f_0)}$ to $\P_{(\zb,g_0)}[g_0(p_2)-f_0(p_1)]$ arising from $\aaa$, as well as a morphism from $\P_{(\zb,g_0)}$ to $\P_{(\za,f_0)}[f_0(p_1)-g_0(p_2)+1]$ arising from $\bbb$.
On the other hand, since both $\P_{(\za,f_0)}$ and $\P_{(\zb,g_0)}$ are projective resolutions of modules, we have $g_0(p_2)-f_0(p_1)\geq 0$ and $f_0(p_1)-g_0(p_2)+1\geq 0$.
So we have $f_0(p_1)=g_0(p_2)$ or $f_0(p_1)=g_0(p_2)+1$.
%the projective modules $P_{\ell^*_1}$ and $P_{\ell^*_2}$ are always at the $(-2)-th$ position or at the $(-1)-th$ position of $\P_\za$ and $\P_\zb$.
%Then a similar argument as above shows that the equality \eqref{equ:morph2} guarantees that we have $f_0(p_1)=g_0(p_2)$ or $f_0(p_1)=g_0(p_2)+1$.
\end{proof}

\begin{definition}\label{definition:weighted intersections}
Let $\za$ and $\zb$ be two zigzag arcs on the surface which share an endpoint $q\in \calp_{\bpoint}$. For any $m\in \mathbb{N}$, let $(\aaa,m)$ and $(\bbb,m)$ be the associated oriented intersections from $\za$ to $\zb$ and from $\zb$ to $\za$ respectively, which are introduced in subsection \ref{subsection: derived categories and derived categories}. Assume that there are $n$ arcs $t(\ell_i), 1\leq i \leq n$ connected to $q$, where $\za$ and $\zb$ connect to $q$ between arcs $t(\ell_{u-1})$ and $t(\ell_u)$, and between arcs $t(\ell_{v-1})$ and $t(\ell_v)$ respectively, see the following diagram
 \[\scalebox{1}{
\begin{tikzpicture}[>=stealth,scale=0.9]

\draw[dark-green!50,thick](2,0)--(5,-1.6);
\draw[dark-green!50,thick](2,0)--(-1,-1.6);
\draw[dark-green!50,thick](2,0)--(2,-2.5);
\draw[dark-green!50,thick](2,0)--(5,1.6);
\draw[dark-green!50,thick](2,0)--(-1,1.6);
\draw[dark-green!50,thick](2,0)--(2,2.5);

\draw[very thick](2,0)--(5,0);
\draw[very thick](2,0)--(-1,0);
\node [] at (0,-.6) {\tiny$\bbp$};

\node [] at (-.8,0.25) {\tiny$\za$};
\node [] at (4.8,-0.25) {\tiny$\zb$};

\node [] at (.95,.9) {\tiny$p_1$};
\node [] at (1.25,.2) {\tiny$p_2$};

            \draw[thick] (2,0) circle (.3);
            \draw [thick,->] (2.27,.1) to (2.31,0);
            \draw [thick,->] (1.73,-.1) to (1.69,0);

			\draw (2.3,.5) node {\tiny$\aaa$};
			\draw (2.3,-.5) node {\tiny$\bbb$};

\node [dark-green!50] at (-.8,1.1) {\tiny$t(\ell_u)$};
%\node [dark-green] at (1.75,2) {\tiny$\ell_v$};
\node [dark-green!50] at (4.8,-1.15) {\tiny$t(\ell_v)$};
\node [dark-green!50] at (-1,-1.2) {\tiny$t(\ell_{u-1})$};
\node [dark-green!50] at (4.8,1.1) {\tiny$t(\ell_{v-1})$};
			    \draw[dotted,very thick]plot [smooth,tension=1] coordinates {(0,0) (.6,0.3) (2,1.2) (3,0.8)};
			    \draw[dotted,very thick]plot [smooth,tension=1] coordinates {(+4,0) (-.6+4,-0.3) (-2+4,-1.2) (.8,-.2) (1.5,.6)};
   \draw[thick,fill=white] (2,0) circle (0.09);
			\draw (6,0) node {};
\end{tikzpicture}}\]
Then we call $(\aaa,m)$ a \emph{weighted-oriented-intersection} from $\za$ to $\zb$ with weight $w(\aaa,m)=(v-u)+mn$, and call $(\bbb,m)$ a \emph{weighted-oriented-intersection} from $\zb$ to $\za$ with weight $w(\bbb,m)=(u-v)+mn$.
\end{definition}

In above definition, we always choose $v-u$ and $u-v$ in a way such that $0 \leq v-u, u-v \leq n-1$ by reducing modulo $n$ if necessary.

\begin{remark}\label{remark:inj}
	Similar to Proposition \ref{prop:proj-res}, for any zigzag $\bpoint$-arc $\za$, if the endpoint is a puncture, rather than a marked point on the boundary, one can still  directly give the associated projective resolution $\P_\za$. The subtle thing is that the complex will be infinite on the left, where the projective module $P_{t(\ell_i)}$ periodically appears, cf. Figure \ref{figure:simple to projective}. 
	
	For a finite dimensional algebra $A$, we know that the bounded derived category $\cald^b(A)$ is triangle equivalent to the homotopy categories $\kaa$ and $\iaa$. The constructions and results obtained above can dually be established for $\iaa$. In particular, one can define the {\em injective coordinate} $\zD^*_i$ to be $t^{-1}(\zD_s)$, which can be used as the coordinate to compute the minimal injective resolution of a module over gentle algebra $A(\zD^*_s)$.
	We will not discuss this in detail, but an example will be given in Section \ref{section:example}.
\end{remark}

Proposition \ref{prop:proj-res} and the descriptions in Remark \ref{remark:inj} directly yield the following corollary, where the third one is well-known from the perspective of the representation theory, that is, both the projective dimension $\pd M$ and the injective dimension $\id M$ of a band module $M$ are one.

\begin{corollary}\label{corollary:proj-res-string}
	Let $\za$ be a zigzag curve, then
	\begin{enumerate}[\rm(1)]
		\item $\pd M_\za=max\{w_p(\za),w_q(\za)\}$, and $\id M_\za=max\{cw_p(\za),cw_q(\za)\}$ if $\za$ is a zigzag $\bpoint$-arc with both endpoints $p$ and $q$ in $\calm_{\bpoint}$;
		\item $\pd M_\za=\id M_\za=\infty$, if $\za$ is a zigzag $\bpoint$-arc with at least one endpoint in $\calp_{\bpoint}$;
		\item  $\pd M_\za=\id M_\za=1$, if $\za$ is a zigzag closed curve.
	\end{enumerate}
\end{corollary}

The above corollary allows us to compute the finitistic dimension $\fd A$ of a gentle algebra $A$, where $\fd A$ is defined as
\begin{equation*}
	\fd A := \sup \{\pd M: \text{~for~} M\in \ma \text{~with~} \pd M < \infty\}.
\end{equation*}

\begin{corollary}\label{corollary:fdim}
	Let $A$ be a gentle algebra, then
	\begin{equation}\label{eq:fd}
		\fd A={\rm max}\{\pd I \text{~for~injective~module~} I\},
	\end{equation}
	which also equals the number of arrows in the longest sequence with consecutive relations, where the sequence contains no arrow from an oriented cycle with full relations.
\end{corollary}
\begin{proof}
%	Denote by $\zD_s^*$ the simple coordinate associated to $A$.
	Let $\bbp$ be a polygon in $\zD_s^*$ which has a boundary marked point $q$ and has maximal number of edges in the set of polygons containing boundary segments. Assume that $\bbp$ has $n$ edges, cf. Figure \ref{figure:weight and co-weight}.
	Then by Corollary \ref{corollary:proj-res-string}, $\pd M \leq n-1$ for any indecomposable $A$-module $M$ with finite projective dimension, and thus $\fd A \leq n-1$.
	
	Let $\ell^*_1$ be the first edge in $\bbp$. Denote by $I_1$ the injective module associated to $t^{-1}(\ell^*_1)$, that is, $I_1=M_{t^{-1}(\ell^*_1)}$, cf. Figure \ref{figure:dual-twist} for the construction of the arcs correspond injective modules.
	Denote by $p$ the another endpoint of $t^{-1}(\ell^*_1)$ which belongs to a polygon $\bbp'$.
	Note that $p$ is also a boundary marked point and $w_p(t^{-1}(\ell^*_1))\leq w_q(t^{-1}(\ell^*_1))$, since $\bbp$ is the maximal one. On the other hand, note that $w_q(t^{-1}(\ell^*_1))=n-1$, see Figure \ref{figure:weight and co-weight}. Then by Corollary \ref{corollary:proj-res-string}, $\pd I_1=n-1$.
	
	Since any sequence with consecutive relations in $A$ is connected by those arrows arising from the same polygon in $\zD_s^*$, thus the polygon $\bbp$ gives the longest such sequence, and the second statement follows directly from above discussion.
\end{proof}

The following theorem is the main result of this subsection.
Recall that when we represent the oriented intersections of curves as morphisms in $\kaa$, we denote a band object $\P_{(\za,f,\lambda,m)}$ by $\P_{(\za,f)}$ in order to state the results easily, if $m=1$ and the statement does not depend on the choice of $\lambda$, see subsection \ref{subsection: derived categories and derived categories}.
In the following we will continue to use  this notation. So for two such modules $M_\za$ and $M_\zb$, they are isomorphic if and only if $\za$ and $\zb$ are homotopic and the associated parameters coincide.

\begin{theorem}\label{theorem:main-extensions}
	Let $(\cals,\calm,\zD^*_s)$ be a marked surface with a simple coordinate, and let $\za$ and $\zb$ be two zigzag curves on the surface. Then for any \emph{oriented intersection} from $\za$ to $\zb$ with weight $\omega\geq 0$, there is a morphism in $\Ext^\omega(M_\za,M_\zb)$ associated to it.
	
	Furthermore, all of such morphisms form a basis of the space $\Ext^\omega(M_\za,M_\zb)$, unless
	\begin{enumerate}[\rm(1)]
		\item $\za$ and $\zb$ are the same $\bpoint$-arc. In this case, the identity map is the extra basis map of $\Hom(M_\za,M_\zb)$;
		\item $\za$ and $\zb$ are the same closed curve and $M_\za\cong M_\zb$. In this case, the identity map is the extra basis map in $\Hom(M_\za,M_\zb)$, and the self-extension corresponding to the Auslander-Reiten sequence of $M_\za$ is the extra basis map of $\Ext^1(M_\za,M_\zb)$.
	\end{enumerate}
\end{theorem}
\begin{proof}
Firstly we show that each oriented intersection from $\za$ to $\zb$ with weight $\omega$ gives rise to a morphism in $\Ext^\omega(M_\za,M_\zb)$.
When the oriented intersection arises from an interior intersection $q$, this directly follows from the equalities \eqref{equ:morph2} and \eqref{equ:morph3} and the definition of the weight of an oriented intersection given in Proposition-Definition \ref{definition:extension in interior} and Definition \ref{definition:weighted intersections} respectively.

Now we assume that $\aaa$ is an oriented intersection from $\za$ to $\zb$ arising from a boundary intersection $q$, whose weight is $\omega$.
Let $\bbp=\{\ell^*_1,\cdots,\ell^*_m\}$ be the polygon of $\zD^*_s$ which contains $q$, where the arcs in $\bbp$ are labeled clockwise, cf. the left picture in Figure
\ref{figure:simple to projective}.
%Assume that starting from $q$, $\za$ and $\ab$ firstly intersect $\ell^
Assume that starting from $q$, $p_1$ and $p_2$ are the first intersection points that $\za$ and $\zb$ respectively intersect with the arcs in $\zD^*_p$. Then we have $f_0(p_1)=-w_q(\za)$ and $g_0(p_2)=-w_q(\zb)$ by Corollary \ref{lem:proj-res2}. 
%We recall that $f_0$ and $g_0$ are respectively the gradings over $\za$ and $\zb$ such that $\P_\za=\P_{(\za,f_0)}$ and $\P_\zb=\P_{(\zb,g_0)}$. 
Then by the equality \eqref{equ:morph1}, there is a morphism from $\P_{(\za,f_0)}$ to $\P_{(\zb,g_0)}[g_0(p_2)-f_0(p_1)]=\P_{(\zb,g_0)}[w_q(\za)-w_q(\zb)]=\P_{(\zb,g_0)}[\omega]$. Such morphism corresponds to a morphism in $\Ext^\omega(M_\za,M_\zb)$, that is, the oriented intersection $\aaa$ gives rise to a morphism in $\Ext^\omega(M_\za,M_\zb)$.
%When $(\aaa,m)$ is an oriented intersection from $\za$ to $\zb$ arising from a puncture intersection $q$, whose weight is $\omega=(v-u)+mn$ given in Definition \ref{definition:weighted intersections}, the proof is similar.

The second part of the theorem that the morphisms form a basis of the extension space follows from Proposition \ref{prop:obj-in-derived-cat} and the one-to-one correspondence between the spaces $\Ext^{\omega}(M_\za,M_\zb)$ and $\Hom(\P_\za,\P_\zb[\omega])$. In particular, since both the projective dimension and the injective dimension of a band module are one, the AR-triangle starts at this band module can be lifted as an AR-triangle in the derived category, thus the statement (2) follows from the corresponding statement (2) of Proposition \ref{prop:obj-in-derived-cat}.
\end{proof}

\begin{remark}
To sum up, let $\za$ and $\zb$ be two zigzag curves which intersect at $q$. If $q$ is a boundary intersection, then there is only one morphism between $M_\za$ and $M_\zb$ arising from $q$. If $q$ is a non-punctured interior intersection, then there is a morphism and a 1-extension between $M_\za$ and $M_\zb$ arising from $q$. If $q$ is a puncture, then there are infinite many of morphisms or extensions between $M_\za$ and $M_\zb$ arising from $q$.

In particular, an extension arising from a non-punctured interior intersection is always an $1$-extension, while all the higher extensions arises from a boundary intersection or an intersection which is a puncture. We mention that this phenomenon is also pointed out in \cite{BS21}, where for gentle algebras arising from triangulations of surfaces, an explicit basis of higher extension spaces between indecomposable modules is given.
\end{remark}

\subsection{Higher Yoneda-extensions as distinguished polygons}\label{section:yonada}

In this subsection, we explain the higher Yoneda-extensions as some special polygons, distinguished polygons, on the surface, and explain the Yoneda-product as gluing of these polygons.

Let $(\cals,\calm,\zD^*_s)$ be a marked surface with a simple coordinate.
Let $\za$ and $\zb$ be two $\bpoint$-arcs on the surface sharing an endpoint $q$ in $\calm_{\bpoint}$. Assume that $q$ belongs to a polygon of $\zD^*_s$ such that $\za$ intersects an arc $\ell^*_{t}\in \zD^*_s$ and $\zb$ intersects an arc $\ell^*_{t+\omega}\in \zD^*_s$.
Denote by $\aaa$ the associated oriented intersection from $\za$ to $\zb$ whose weight is $w(\aaa)=\omega$.
We construct a $(\omega+2)$-gon $\bbp(\aaa)=\{\zg_0,\zg_1,\cdots,\zg_{\omega+1}\}$ in the following way, see the picture in Figure \ref{figure:Mhigher-extension}.

We write $\zb=\zg_0$ and $\za=\zg_{\omega+1}$.
Let $\zg_1$ be the $\bpoint$-arc obtained from smoothing $\zg_0$ and $\ell_{t+\omega-1}$ at $q$, and let $\zg_{\omega}$ be the $\bpoint$-arc obtained from smoothing $\zg_{\omega+1}$ and $\ell_{t+1}$ at $q$. For $2\leq i \leq \omega-1$, let $\zg_i$ be the $\bpoint$-arc obtained from smoothing $\ell_{t+\omega-i+1}$ and $\ell_{t+\omega-i}$ at $q$.
For any $1\leq i \leq \omega+1$, we denote the common endpoint of $\zg_{i-1}$ and $\zg_{i}$ by $p_{i}$.
Then it is clear that each $p_i, 1\leq i \leq \omega,$ gives rise to an oriented intersection $\bbb_i$ from $\zg_{i-1}$ to $\zg_i$ of weight zero. We denote the map in $\Hom(M_{\zg_{i-1}},M_{\zg_i})$ associated to $\bbb_i$ still by $\bbb_i$.

When $q$ is a puncture, we can similarly construct the $(\omega+2)$-gon $\bbp(\aaa)$.
The subtle thing is that the oriented intersection may wrap around the $q$ several times, so the associated polygon is a kind of covering, see Figure \ref{figure:Mhigher-extension2} for an example.
\begin{figure}
	\[\scalebox{1}{
		\begin{tikzpicture}[>=stealth,scale=0.6]
			\draw[red!50,thick] (-6,0)--(-5.2,3.3)--(-2,5);
			\draw[red!50,thick] (6,0)--(5.2,3.3)--(2,5);
			\draw[red!50,dashed,thick] (-2,5)--(2,5);
			
			\draw[red!50,thick] (-6,-0)--(-5.2,-3.3);
			\draw[red!50,dashed,thick] (-5.2,-3.3)--(0,-5);
			\draw[red!50,thick] (6,-0)--(5.2,-3.3);
			\draw[red!50,dashed,thick] (5.2,-3.3)--(0,-5);
			%\draw[red!50,dashed,thick] (-2,-5)--(2,-5);
			
			\draw[thick] (-6.5,2)--(0,0)--(6.5,2);
			\draw[dashed,thick,black!50] (-4,5)--(0,0)--(4,5);
			\draw[thick,dashed,black!50] (0,0)--(0,6);
			
			\draw[thick]plot [smooth,tension=1] coordinates {(-6.5,2) (-4.5,3) (-4,5)};
			\draw[thick]plot [smooth,tension=1] coordinates {(6.5,2) (4.5,3) (4,5)};
			\draw[thick]plot [smooth,tension=1] coordinates {(0,6) (-1.5,4) (-4,5)};
			\draw[thick]plot [smooth,tension=1] coordinates {(0,6) (1.5,4) (4,5)};
			
			\draw[thick,dashed,black!50] (-6.5,-2)--(0,0)--(6.5,-2);
			\draw[dashed,thick,black!50] (-4,-5)--(0,0)--(4,-5);
			%\draw[thick,dashed,black!50] (0,0)--(0,-6);
			
			\draw[thick]plot [smooth,tension=1] coordinates {(-6.5,-2) (-4.5,-3) (-4,-5)};
			\draw[thick]plot [smooth,tension=1] coordinates {(6.5,-2) (4.5,-3) (4,-5)};
			%\draw[thick]plot [smooth,tension=1] coordinates {(0,-6) (-1.5,-4) (-4,-5)};
			%\draw[thick]plot [smooth,tension=1] coordinates {(0,-6) (1.5,-4) (4,-5)};
			\draw[thick]plot [smooth,tension=1] coordinates {(-4,-5) (0,-3) (4,-5)};
			
			\draw[thick]plot [smooth,tension=1] coordinates {(-6.5,2) (-5,0) (-6.5,-2)};
			\draw[thick]plot [smooth,tension=1] coordinates {(6.5,2) (5,0) (6.5,-2)};
			
			\draw[red,thick,fill=red] (5.2,3.3) circle (0.08);
			\draw[red,thick,fill=red] (2,5) circle (0.08);
			\draw[red,thick,fill=red] (-5.2,3.3) circle (0.08);
			\draw[red,thick,fill=red] (-2,5) circle (0.08);
			
			\draw[red,thick,fill=red] (6,0) circle (0.08);
			\draw[red,thick,fill=red] (5.2,-3.3) circle (0.08);
			%			\draw[red,thick,fill=red] (2,-5) circle (0.08);
			\draw[red,thick,fill=red] (-6,0) circle (0.08);
			\draw[red,thick,fill=red] (-5.2,-3.3) circle (0.08);
			%			\draw[red,thick,fill=red] (-2,-5) circle (0.08);
			\draw[red,thick,fill=red] (0,-5) circle (0.08);
			
			\node at (-1.5,.8) {\tiny$\za$};
			\node at (1.5,.8) {\tiny$\zb$};
			%			\node[black!50] at (-1.6,3) {\tiny$\ell_{t+\omega-1}$};
			%			\node[black!50] at (1.5,3) {\tiny$\ell_{t+1}$};
			%			\node [red!50] at (-5,1) {\tiny$\ell^*_{t+\omega}$};
			%			\node [red!50] at (6.3,1.2) {\tiny$\ell^*_{t}$};
			%			\node [red!50] at (4.8,4.2) {\tiny$\ell^*_{t+1}$};
			\node at (-3.4,3.2) {\tiny$\zg_{\omega-n}$};
			\node at (-3.92,2.6) {\tiny$\zg_{\omega+1}$};
			\node at (3.6,3.2) {\tiny$\zg_1$};
			\node at (3.92,2.6) {\tiny$\zg_{n+1}$};
			
			\node [] at (4.4,0) {\tiny$\zg_{n}$};	
			\node [] at (-3.9,0) {\tiny$\zg_{\omega-n+1}$};	
			
			\draw[thick,fill=white] (0,0) circle (0.08);
			\draw[thick,->]plot [smooth,tension=1] coordinates {(-.45,.15)(.1,.5)(.5,-.1)(-.1,-.7)(-.9,0)(-.1,.85)(.8,.245)};
			\node [] at (0,-1.2) {\tiny$\aaa$};
			\node at (0,-.45) {\tiny$q$};
			
			\draw[thick,fill=white] (4,5) circle (0.08);
			%			\node [] at (4.3,5.3) {\tiny$p_2$};
			\draw[thick,fill=white] (-4,5) circle (0.08);
			%			\node [] at (-4.3,5.3) {\tiny$p_{\omega}$};

			\draw[thick,fill=white] (-6.5,2) circle (0.08);

			\draw[thick,fill=white] (6.5,2) circle (0.08);
			\draw[thick,bend left,->](5,1.52)to(5,2.52);
			\node [] at (4.2,2) {\tiny$\bbb_1$};
			
			\draw[thick,bend left,->](-5.31,2.34)to(-5.5,1.12);
			\node [] at (-4.2,.9) {\tiny$\bbb_{\omega-n+1}$};	
			\draw[thick,bend left,->](-5,2.52)to(-5,1.52);
			\node [] at (-4,2) {\tiny$\bbb_{\omega+1}$};
			
			\draw[thick,bend left,->](5.5,1.12)to(5.31,2.34);
			\node [] at (4.56,.9) {\tiny$\bbb_{n+1}$};	
			
			%			\node [] at (6.4,2.6) {\tiny$p_1$};
			
			\draw[thick,fill=white] (6.5,-2) circle (0.08);			
			\draw[thick,fill=white] (-6.5,-2) circle (0.08);			
			\draw[thick,fill=white] (0,6) circle (0.08);			
			\draw[thick,fill=white] (4,-5) circle (0.08);			
			\draw[thick,fill=white] (-4,-5) circle (0.08);	
	\end{tikzpicture}}\]
	\begin{center}
		\caption{A distinguished $(\omega+2)$-gon $\bbp(\aaa)$ corresponds to an oriented intersection $\aaa$ from $\za$ to $\zb$ with weight $\omega$, where $q$ is a puncture and there are $n$ $\rpoint$-arcs in the polygon of $\Delta^*_s$ that $q$ belongs to.}\label{figure:Mhigher-extension2}
	\end{center}
\end{figure}

Note that the edges of a polygon (in a classical sense) are simple curves, however, as the edges of a $(\omega+2)$-gon $\bbp(\aaa)$, the arcs $\za$ and $\zb$ may have self-intersections, thus $\bbp(\aaa)$ is not necessarily a `real polygon'.
We call $\bbp(\aaa)$ a {\em distinguished polygon} on the surface, where the name is justified by the following

\begin{theorem}\label{theorem:yoneda}
	Let $\za$ and $\zb$ be two $\bpoint$-arcs on a marked surface $(\cals,\calm,\zD^*_s)$ which share a common endpoint $q\in \calm_{\bpoint}\cup \calp_{\bpoint}$. Denote by $\aaa$ the associated oriented intersection from $\za$ to $\zb$ with weight $w(\aaa)=\omega$, which gives rise to a map in $\Ext^\omega(M_\za,M_\zb)$. Then under above notations, the Yoneda $\omega$-extension $[\aaa]$ corresponding to the map $\aaa$ is as follows:
	
	\begin{gather}\label{eq:yoneda}
		\xymatrix@C=0.7cm
		{[\aaa]=0\ar[r]^{} & M_{\zb}\ar[r]^{\bbb_1} & M_{\zg_1}\ar[r]^{\bbb_2} & \cdots \ar[r]^{} & M_{\zg_{\omega}}\ar[r]^{\bbb_{\omega+1}} & M_{\za}\ar[r]& 0}.
	\end{gather}
\end{theorem}
\begin{proof}
	We prove the case for $\omega=2$. The general case can be inductively proved with respect to $\omega$ by using a similar method.
	
	Denote $\ell_{t+1}$ by $\ell$ for simplicity. We denote by $\aaa_1$ and $\aaa_2$ respectively the oriented intersection from $\za$ to $\ell$ and the oriented intersection from $\ell$ to $\zb$ arising from the intersection $q$. Then both of the weights $\omega(\aaa_1)$ and $\omega(\aaa_2)$ are one, which give rise to two morphisms, still denoted by $\aaa_1$ and $\aaa_2$, in $\Ext^1(M_\za,M_\ell)$ and $\Ext^1(M_\ell,M_\zb)$ respectively. It is straightforward to see that the Yoneda $1$-extension associated to $\aaa_1$ and $\aaa_2$ are respectively
	$$\xymatrix@C=0.7cm
	{[\aaa_1]=0\ar[r]^{} & M_{\ell}\ar[r]^{\bbb_2''} & M_{\zg_{2}}\ar[r]^{\bbb_3} & M_{\za}\ar[r]& 0}~~~\text{ and}$$
	$$\xymatrix@C=0.7cm
	{[\aaa_2]=0\ar[r]^{} & M_{\zb}\ar[r]^{\bbb_1} & M_{\zg_{1}}\ar[r]^{\bbb_2'} & M_{\ell}\ar[r]& 0},$$
	where $\bbb_2'$ and $\bbb_2''$ are the natural morphisms arising from the canonical boundary intersections obtained from smoothing of arcs.
	Denote by $\aaa_1*\aaa_2$ the Yoneda product of $\aaa_1$ and $\aaa_2$ in $\Ext^2(M_\za,M_\zb)$. Then the Yoneda $2$-extension associated to $\aaa_1*\aaa_2$ is
	\begin{gather}\label{eq:yoneda1}
		\xymatrix@C=0.7cm
		{[\aaa_1*\aaa_2]=0\ar[r]^{} & M_{\zb}\ar[r]^{\bbb_1} & M_{\zg_1}\ar[r]^{\bbb_2}& M_{\zg_{2}}\ar[r]^{\bbb_3} & M_{\za}\ar[r]& 0}.
	\end{gather}
	Since the composition of oriented intersections corresponds to the composition of associated morphisms, see for example in \cite[Remark 2.6]{CS23b}, thus we have equality of morphisms $\bbb_2=\bbb_2'\bbb_2''$, noticing that $\bbb_2=\bbb_2'\bbb_2''$ as oriented intersections.
	On the other hand, the Yoneda-product $\aaa_1*\aaa_2$ corresponds to the composition $\aaa_1(\aaa_2[1])$ in $\Hom(\P_\za,\P_\zb[2])$, which equals the map $\aaa$ again by \cite[Remark 2.6]{CS23b}. Thus \eqref{eq:yoneda1} is exactly the Yoneda extension associated to the map $\aaa$ in $\Ext^2(M_\za,M_\zb)$.
\end{proof}

\begin{remark}\label{remark:yoneda}
	The above theorem shows that the distinguished $(\omega+2)$-gon $\bbp(\aaa)$ represents a Yoneda $\omega$-extension of a map in $\Ext^\omega(M_\za,M_\zb)$.
	Furthermore, the Yoneda product can be explained as gluing of these polygons. More precisely, let $\aaa$ and $\bbb$ be two maps in  $\Ext^\omega(M_\za,M_\zb)$ and $\Ext^\delta(M_\zb,M_\zg)$ arising from a common marked point on the boundary. Let $\bbp(\aaa)$ and $\bbp(\bbb)$ be the polygons associated to $[\aaa]$ and $[\bbb]$ respectively. Then $\aaa\bbb$ is a non-zero map in $\Ext^{\omega+\delta}(M_\za,M_\zg)$, and the polygon associated to the Yoneda extension is the $(\omega+\delta+2)$-gon $\bbp(\aaa\bbb)$, which is obtained by gluing $\bbp(\aaa)$ and $\bbp(\bbb)$ along the arc $\zb$.

	Note that by using the string combinatoric calculation, the 1-extensions between string modules are described as so-called {\it overlap extensions} and {\it arrow extensions}, see for example in \cite{BDMTY20,CPS21,CS17,S99,Z14}. Note that for an overlap extension, the middle term of the associated exact sequence is a direct sum of two indecomposable string modules, while for an arrow extension, the middle term is indecomposable. 
	Therefore in the language of the surface model, an overlap extension corresponds to an oriented non-punctured interior intersection, while an arrow extension corresponds to an oriented intersection at a boundary/punctured-marked point.
\end{remark}

\section{A geometric model for the algebraic hearts of a gentle algebra}\label{section:heart}

In this section, by using the geometric model for the module category of a gentle algebra established in above section, we will give a geometric model for any algebraic heart in the derived category of a gentle algebra, where such a heart is always equivalent to the module category of a gentle algebra.

In the rest of the paper, let $A$ be a gentle algebra arising from a marked surface $(\cals,\calm,\zD^*_A)$, where $\zD^*_A$ is a projective coordinate. Here we leave the notations $\zD^*_p$, $\zD_s^*$, $\zD_p$ and $\zD_s$ to represent arbitrary (rather than the initial) coordinates and dissections respectively on $(\cals,\calm)$.

We introduce a notation which will be used in the following. Let $\za$ be an $\bpoint$-arc and let $q$ be an endpoint of $\za$, we denote by $s^\za_{q}$ the intersection in $\za \cap \zD_A^*$ which is nearest to $q$, where $\za \cap \zD_A^*$ is the set of intersections between $\za$ and the arcs in $\zD^*_A$.

\begin{definition}\label{definition:silting dissection}
A \emph{simple-minded dissection} of $(\cals,\calm,\zD^*_A)$ is a pair $(\zD_s,f)$, where $\zD_s$ is a simple dissection of $(\cals,\calm)$ and $f=\{f_\za, \za\in \zD_s\}$ is a set of gradings of arcs from $\zD_s$ such that
$$f_\za(s^\za_{q}) < f_\zb(s^\zb_{q}),$$
whenever $q$ is a common endpoint of arcs $\za$ and $\zb$ in $\zD_s$ such that $\zb$ follows $\za$ clockwise at $q$.
\end{definition}

Then we have the following geometric description of simple-minded collections in $\dba$, which can be proved similar to the proof of \cite[Theorem 3.2]{APS23} and its converse \cite[Proposition 3.7]{APS23}, where a geometric characterization for silting objects is given, noticing that there is a kind of duality between simple-minded collections in $\dba$ and silting objects in $\ka$, see details in \cite[Section 5]{KY14}.
For convenience of the reader, we provide a sketch of the proof after the theorem.

\begin{theorem}\label{thm:simple minded dissection}
Let $(\zD_s,f)$ be a simple-minded dissection of $(\cals,\calm,\zD^*_A)$.
Then
$$\P_{(\zD_s,f)}:=\{\P_{(\za,f_\za)}, (\za,f_\za)\in (\zD_s,f)\}$$
is a simple-minded collection in $\dba$.

Conversely, any simple-minded collection in $\dba$ arises from this way. That is, there is a bijection from the set of homotopy classes of simple-minded dissections on $(\cals,\calm,\zD^*_A)$ to the set of isomorphism classes of simple-minded collections in $\dba$.
\end{theorem}
\begin{proof}[{\bf Sketch of the proof:} ]
Assume we have a simple-minded dissection $(\zD_s,f)$, then any morphism arises from an oriented intersection from an arc $\za$ to an arc $\zb$ at a common endpoint $q$. Since we have $f_\za(s^\za_{q}) < f_\zb(s^\zb_{q}),$ the degree of this morphism is $f_\zb(s^\zb_{q})-f_\za(s^\za_{q})$, which is positive. Thus the conditions (1) and (2) in the definition of a simple-minded collection hold, recall the definition given in subsection 
\ref{subsection: t-structures and hearts}. 
To show the condition (3) holds, that is, $\dba$ can be generated by objects $\P_{(\za,f_\za)}, (\za,f_\za)\in (\zD_s,f)$, it suffices to show that a shift  $\P_{(\zg,g)}$ of a projective resolution of each simple module $M_{\zg}$ can be generated by these objects. Since smoothing of arcs corresponds to taking the cones of morphisms between the associated objects (up to a shift) in $\dba$ \cite[Theorem 4.1]{OPS18}, it is only need to show that $\zg$ can be obtained by iteratively smoothing of arcs in $\zD_s$.
To see this, note that the arcs in $\zD_s$ cut the surface into polygons. The arc $\zg$ crossing a polygon is homotopic to a smoothing of some of the segments forming the boundary of this polygon; applying this to all polygons crossed by $\zg$, we get that $\zg$ is a smoothing of the $\za$.

Conversely, if we have a simple-minded collection $\{S_1,\ldots,S_n\}$, then each object $S_i$ corresponds a graded $\bpoint$-arc $\za_i$. Note that there is no interior intersection between $\za_i$ and $\za_j$ for any $1\leqslant i,j \leqslant n$. Otherwise, there are two morphisms arising from the intersection with degrees $d$ and $1-d$, cf. \eqref{equ:morph2} and \eqref{equ:morph22}, which contradicts the conditions (1) and (2). Therefore the arcs can only intersect at the $\bpoint$-points. Furthermore, by a similar argument used in \cite[Lemma 5.6]{APS23}, the conditions (1) and (2) imply that the arcs $\za_i$ do not enclose any unpunctured surface. On the other hand, since there are $n$ arcs $\za_i$, they cut the surface into polygons each which has exactly one $\rpoint$-point. Therefore the  arcs $\za_i, 1\leqslant i \leqslant n,$ form a simple dissection.
Finally, the inequality in the definition of a simple-minded dissection follows from the conditions (1) and (2).
\end{proof}

Recall that a minimal oriented intersection in a simple coordinate is an anti-clockwise angle, see subsection \ref{subsection: Modules as curves}.
Now dually, we introduce a minimal oriented intersection for a simple dissection.
More precisely, a \emph{minimal oriented intersection}
on a simple dissection $\zD_s$ is a \emph{clockwise} angle locally from an arc $\ell_i$ to an arc $\ell_j$ in $\zD_s$ at a common endpoint, such that the angle is in the interior of the surface and there is no other arcs from $\zD_s$ in between.

Let $\zD_s$ be a simple dissection with a dual simple coordinate $\zD_s^*$ on the surface. Then note that for any minimal oriented intersection $\aaa$ of $\zD_s$ from $\ell_i$ to $\ell_j$, there is a unique minimal oriented intersection $\aaa^{*}$ of $\zD^*_s$ from $\ell_i^*$ to $\ell_j^*$, see Figure \ref{figure:graded dual algebra}.
Furthermore, this establishes bijections between the set of minimal oriented intersections on $\zD_s$ and the set of minimal oriented intersections on $\zD^*_s$.
We call $\aaa^{*}$ the \emph{dual} of $\aaa$.

\begin{definition}\label{definition:graded}
Let $A$ be a gentle algebra associated with a surface model $(\cals,\calm,\zD^*_A)$.
Let $(\zD_s,f)$ be a simple-minded dissection on $(\cals,\calm,\zD^*_A)$.
\begin{enumerate}[\rm(1)]
\item Let $F$ be a map from the set of minimal oriented intersections of $\zD_s$ to the integer set $\Z$, such that for each minimal oriented intersection $\aaa: \ell_i\rightarrow \ell_j$ arising from a common endpoint $q$ of $\ell_i$ and $\ell_j$, we have
$$F(\aaa)=f_{\ell_j}(s^{\ell_j}_q)-f_{\ell_i}(s^{\ell_i}_q).$$
We call $F(\aaa)$ the \emph{grading} of $\aaa$, and call the pair $(\zD_s,F)$ the \emph{graded simple dissection} induced from $(\zD_s,f)$.
\item Let $F^*$ be a map from the set of minimal oriented intersections of $\zD_s^*$ to the integer set $\Z$, such that for each minimal oriented intersection $\aaa^{*}$, we have
$$F^*(\aaa^{*})=1-F(\aaa),$$ where $\aaa^{*}$ is the dual of $\aaa$.
We call the pair $(\zD^*_s,F^*)$ the \emph{graded simple coordinate} induced from $(\zD_s,f)$.
\end{enumerate}
\end{definition}

Let's explain why we choose gradings as above. 
The first grading determined by $F$ actually maps into the set $\mathbb{Z}_{\geq 1}$,
which corresponds to simple objects in hearts having no non-positive extensions
between them. 
The dual grading $F^*$ is non-positive, meaning that the algebra $\tilde{\zG}$ we
are getting is going to arise from a silting object, see more details in the proof of Theorem \ref{theorem: last}.

\begin{figure}
 \[\scalebox{1}{
\begin{tikzpicture}[>=stealth,scale=0.6]
\draw [thick,bend right] (0,4.6)to(4,4.6);

\draw [thick] (-1.3,3.8)--(2,0);
\draw [thick](5.3,3.8)--(2,0);
\draw [thick,red] (-1.3,.2)--(2,4);
\draw [thick,red](5.3,.2)--(2,4);

\draw [thick,bend left,->](1.7,.35)to(2.3,.35);
\node at (2,.8) {\tiny$\aaa$};
\draw [thick,bend left,<-,red] (2.3,3.65)to(1.7,3.65);
\node[red] at (2,3.2) {\tiny$\aaa^{*}$};

\node at (1.2,.5) {\tiny$\ell_i$};
\node at (3,.5) {\tiny$\ell_j$};
\node[red] at (-1.5,-.3) {\tiny$\ell^*_i$};
\node[red] at (5.5,-.3) {\tiny$\ell^*_j$};
\node at (2,-.4) {\tiny$q$};
\node[red] at (2,4.4) {\tiny$q^*$};

\draw[red,thick,fill=red] (2,4) circle (0.1);
\draw[thick,fill=white] (2,0) circle (0.1);
\end{tikzpicture}
}\]
\begin{center}
\caption{The dual $\aaa^{*}$ of a minimal oriented intersection $\aaa$ of $\zD_s$. We set $F^*(\aaa^{*}) = 1-F(\aaa)$ for the gradings.}\label{figure:graded dual algebra}
\end{center}
\end{figure}

\begin{definition}\label{definition:graded2}
Let $A$ be a gentle algebra associated with a surface model $(\cals,\calm,\zD^*_A)$.
Let $(\zD_s,f)$ be a simple-minded dissection on $(\cals,\calm,\zD^*_A)$, whose induced graded-simple coordinate is $(\zD^*_s,F^*)$.
\begin{enumerate}[\rm(1)]
\item A \emph{zigzag curve} on the quadruple $(\cals,\calm,\zD^*_s,F^*)$ is a zigzag curve on $(\cals,\calm,\zD^*_s)$ which is compatible with the gradings of $\zD^*_s$, that is, $F^*(\aaa^*)=0$ for any minimal oriented intersection $\aaa^*$ that the curve goes through. We call an $\bpoint$-arc a \emph{zigzag arc} if it is a zigzag curve on $(\cals,\calm,\zD^*_s,F^*)$. We call a (primitive) closed curve a \emph{closed zigzag curve} if it is a zigzag curve on $(\cals,\calm,\zD^*_s,F^*)$.
\item The algebra of the graded-simple coordinate $(\zD^*_s,F^*)$ is defined to be the graded algebra $\tilde{\zG}=A(\zD^*_s)$ with the grading induced by the grading $F^*$, that is, the grading of each arrow $\aaa^*\in Q(\zD^*_s)$ equals $F^*(\aaa^*)$. Denote by $\zG$ the subalgebra of $\tilde{\zG}$ consisting of elements with zero gradings.
\end{enumerate}
\end{definition}

\begin{lemma}
The algebras $\tilde{\zG}$ and $\zG$ are gentle algebras.
\end{lemma}
\begin{proof}
We know that $\tilde{\zG}$ is a graded gentle algebra, since $\zD^*_s$ is a simple coordinate. Then $\zG$ is also a gentle algebra, since it is obtained from $\tilde{\zG}$ by deleting the arrows with non-zero gradings, and by forgetting the relations which contain these arrows. Note that $\zG$ may not be connected.
\end{proof}
\begin{theorem}\label{theorem: last}
Let $A$ be a gentle algebra associated with a surface model $(\cals,\calm,\zD^*_A)$.
Let $\calh$ be a heart of $\dba$ arising from a simple-minded collection $\{S_1,\cdots,S_n\}$, with associated simple-minded dissection $(\zD_s,f)$. Then $\calh$ is equivalent to the module category of the gentle algebra $\zG$. Furthermore, any indecomposable object in $\calh$ is of the form $\P_{(\za,g)}$ for some zigzag arc $\za$ on $(\cals,\calm,\zD^*_s,F^*)$ and some grading $g$, or of the form $\P_{(\za,g,\lambda,m)}$ for some zigzag closed curve $\za$ on $(\cals,\calm,\zD^*_s,F^*)$ and some grading $g$ and $\lambda\in k^*$, $m\in \mathbb{N}$.
\end{theorem}
\begin{proof}

By \cite[Lemma 5.3]{KY14}, the heart $\calh$ is equivalent to the module category of the endomorphism algebra $B=kQ/I$ of the associated silting object, which is a gentle algebra by \cite[Proposition 3.7]{APS23}.
Since $S_1,\cdots,S_n$ are simples in $\calh$, the vertices of $Q$ correspond to the arcs in $\zD_s$. The arrows of $Q$ correspond to a basis of $\Ext^1_{\calh}(S,S)=\Hom_{\dba}(S,S[1])$, where $S=\bigoplus_{1\leqslant i\leqslant n}S_i$, and thus correspond to oriented intersections with grading one, between graded arcs in $(\zD_s,f)$. Moreover, for two arrows $\aaa:\P_{(\za,f_\za)}\rightarrow \P_{(\zb,f_\zb)}$ and $\bbb:\P_{(\zb,f_\zb)}\rightarrow \P_{(\zg,f_\zg)}$, the composition $\aaa\bbb$ belongs to $I$ if and only if  the common endpoint of $\za$ and $\zb$, and the common endpoint of $\zb$ and $\zg$ that respectively gives rise to $\aaa$ and $\bbb$ are different.

On the other hand, by the construction, $\tilde{\zG}$ is a non-positive graded algebra whose zero part $\zG$ is isomorphic to the algebra $B$, noticing that the grading $F^*(\aaa^*)$ is defined as $1-F(\aaa)$, where $\aaa$ gives rise to an arrow in $Q$ if and only if $F(\aaa)=1$ if and only if $F^*(\aaa^*)=0$.

By the construction, the simples $S_i\in \calh$ is of the form $\P_{(\ell_i,f_{\ell_i})}$ for a graded arc $(\ell_i,f_{\ell_i})$ with $\ell_i\in \zD_s$.
Furthermore, any indecomposable object in $\calh$ is iteratively generated by simples $\P_{(\ell_i,f_{\ell_i})}$, that is, obtained by iterated smoothing the arcs in $\zD_s$. Thus it essentially follows from Theorem \ref{theorem:main arcs and objects} that the indecomposable objects in $\calh$ are always arising from a zigzag curve on $(\cals,\calm,\zD^*_s,F^*)$. That is, the second statement follows.
\end{proof}

The following diagram depicts the constructions appearing in Theorem \ref{theorem: last}.
$$
\xymatrix{
\calh\ar@{<->}[rrr]^{\text{Proposition~} \ref{prop:simpleminded-to-t-str}}&&&\P_{(\zD_s,f)}\\
&&&\\
(\cals,\calm,\zD^*_s,F^*)\ar@{->}[uu]^{\text{Definition~} \ref{definition:graded2}}
\ar@{<-}[rrr]^{\text{Definition~} \ref{definition:graded}}
&&&{(\zD_s,f)}\ar@{<->}[uu]_{\text{Definition~} \ref{thm:simple minded dissection}}}
$$

\begin{remark}
For an indecomposable object $\P_{(\za,g)}$ or $\P_{(\za,g,\lambda,m)}$ in $\calh$, the grading $g$ is determined by the extensions that generates it, as well as the gradings of the arcs $\ell_i$ appearing in extensions.
%However, it seems that hard to give a direct characterization for this grading.

As a direct corollary of above theorem, we get a bijection between the homotopy classes of graded simple coordinates on the surface $(\cals,\calm,\zD^*_A)$ and the equivalent classes of the hearts in $\dba$ up to shift.
\end{remark}

\section{An example}\label{section:example}

In this section, we give a concrete example to illustrate the results we have obtained. We consider the following gentle algebra $A$:
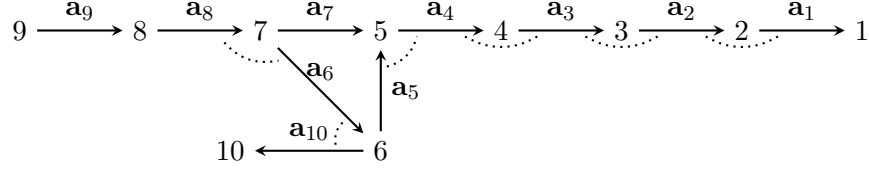
\begin{center}
	{\begin{tikzpicture}[scale=.8,xscale=1,ar/.style={->,thick,>=stealth}]
			\draw(0,0)node(9){$9$}(2,0)node(8){$8$}(4,0)node(7){$7$}(6,0)node(5){$5$}
			(8,0)node(4){$4$}(10,0)node(3){$3$}(12,0)node(2){$2$}(14,0)node(1){$1$}
			(3.5,-2)node(10){$10$}(6,-2)node(6){$6$};
			
			\draw[ar](9) to node[above] {$\aaa_9$} (8);
			\draw[ar](8)to node[above] {$\aaa_8$} (7);
			\draw[ar](7)to node[above] {$\aaa_7$} (5);
			\draw[ar](5)to node[above] {$\aaa_4$} (4);
			\draw[ar](4)to node[above] {$\aaa_3$} (3);
			\draw[ar](3)to node[above] {$\aaa_2$} (2);
			\draw[ar](2)to node[above] {$\aaa_1$} (1);
			
			\draw[ar](7)to node[above] {$\aaa_6$} (6);
			\draw[ar](6)to node[right] {$\aaa_5$} (5);
			\draw[ar](6)to node[above] {$\aaa_{10}$} (10);
			
			\draw[bend right,dotted,thick]($(8)!.7!(7)-(0,.2)$) to ($(7)!.2!(6)-(.1,.1)$);
			\draw[bend left,dotted,thick]($(10)!.7!(6)+(0,.1)$) to ($(7)!.7!(6)-(0,.1)$);
			\draw[bend right,dotted,thick]($(6)!.7!(5)+(.1,0)$) to ($(5)!.3!(4)-(0,.1)$);
			\draw[bend right,dotted,thick]($(5)!.7!(4)-(0,.1)$) to ($(4)!.3!(3)-(0,.1)$);
			\draw[bend right,dotted,thick]($(4)!.7!(3)-(0,.1)$) to ($(3)!.3!(2)-(0,.1)$);
			\draw[bend right,dotted,thick]($(3)!.7!(2)-(0,.1)$) to ($(2)!.3!(1)-(0,.1)$);
	\end{tikzpicture}}
\end{center}

By Corollary \ref{corollary:fdim}, the finitistic dimension $\fd A$ is $5$, which equals the projective dimension of the injective module $I_6$, and also corresponds to the longest path with consecutive relations from $6$ to $1$.
The surface model associated to $A$ is an annulus $(\cals,\calm,\zD^*_s)$:
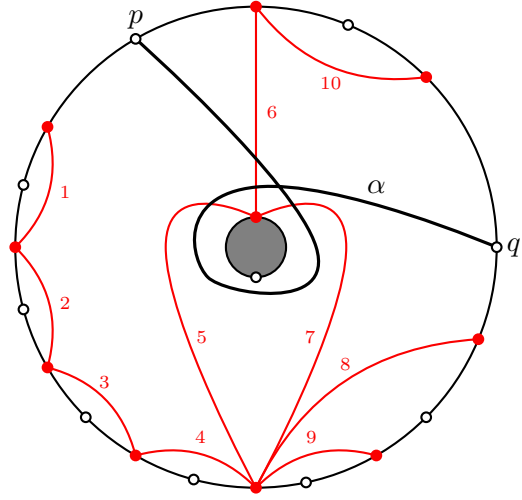
\begin{center}{
		\begin{tikzpicture}[scale=0.8]
			\draw[thick,fill=white] (0,0) circle (4cm);		
			\draw[thick,fill=gray] (0,0) circle (0.5cm);
			
			\path (0:4) coordinate (b1)
			(67.5:4) coordinate (b2)
			(120:4) coordinate (b3)
			(165:4) coordinate (b4)
			(195:4) coordinate (b5)
			(225:4) coordinate (b6)
			(255:4) coordinate (b7)
			(-78:4) coordinate (b8)
			(-45:4) coordinate (b9)
			(-90:.5) coordinate (b10);
			
			\path (45:4) coordinate (r1)
			(90:4) coordinate (r2)
			(150:4) coordinate (r3)
			(180:4) coordinate (r4)
			(210:4) coordinate (r5)
			(240:4) coordinate (r6)
			(270:4) coordinate (r7)
			(-60:4) coordinate (r8)
			(-22.5:4) coordinate (r9)
			(90:.5) coordinate (r10);

			\draw[red,thick]plot [smooth,tension=1] coordinates {(r7) (1.5,0) (r10)};
			\draw[red,thick]plot [smooth,tension=1] coordinates {(r7) (-1.5,0) (r10)};
			
			\draw[red,thick,bend left] (r1) to node[below]{\tiny$10$} (r2)
			(r3) to node[right]{\tiny$1$} (r4)
			(r4) to node[right]{\tiny$2$} (r5)
			(r5) to node[above]{\tiny$3$} (r6)
			(r6) to node[above]{\tiny$4$} (r7)
			(r7) to node[above]{\tiny$9$} (r8)
			(r7) to node[above]{\tiny$8$} (r9);
			\draw[red] (-.9,-1.5) node {\tiny$5$};
			\draw[red] (.9,-1.5) node {\tiny$7$};
			
			\draw[red,thick] (r10) to node[right]{\tiny$6$}  (r2);
			
			\draw (b1) node[right] {$q$};
			\draw (b3) node[above] {$p$};
			\draw (2,1) node {$\za$};
			
			\draw[very thick]plot [smooth,tension=1] coordinates {(b1) (0,1) (-.8,-.5) (1,0) (b3)};
			
			\draw[thick,fill=white] (b1) circle (.08cm)
			(b2) circle (.08cm)
			(b3) circle (.08cm)
			(b4) circle (.08cm)
			(b5) circle (.08cm)
			(b6) circle (.08cm)
			(b7) circle (.08cm)
			(b8) circle (.08cm)
			(b9) circle (.08cm)
			(b10) circle (.08cm);

			\draw[thick,red,fill=red] (r1) circle (.08cm)
			(r2) circle (.08cm)
			(r3) circle (.08cm)
			(r4) circle (.08cm)
			(r5) circle (.08cm)
			(r6) circle (.08cm)
			(r7) circle (.08cm)
			(r8) circle (.08cm)
			(r9) circle (.08cm)
			(r10) circle (.08cm);	
	\end{tikzpicture}}
\end{center}

(1) We consider an arc $\za$ with endpoints $p$ and $q$ depicted in above picture.
By {\bf Construction 1}, the associated module is
$$M_\za=\begin{array}{c}
	6~7 \\
	\text{~}~\text{~}5~6.
\end{array}$$
Note that $p$ and $q$ belong to the polygons $\bbp=\{1,2,3,4,5,6\}$ and $\bbp'=\{6,7,8,10\}$ respectively.
The weights of $\za$ at $p$ and $q$ are
$$w_p(\za)=6-1=5, w_q(\za)=2-1=1.$$
Therefore by Corollary \ref{corollary:proj-res-string}, the projective dimension of $M_\za$ is
$$\pd M_\za=max\{w_p(\za),w_q(\za)\}=5.$$
In fact, the string associated to $M_\za$ is
$$\omega=(\aaa_5)(\aaa_7^{-1})(\aaa_6).$$
After combining the projective modules at the same degree in the complex given in Proposition \ref{prop:proj-res}, the minimal projective resolution of $M_\za$ is as follows:
\[
\begin{tikzpicture}[scale=.5,xscale=3,yscale=3,ar/.style={->,thick}]
	\draw(-7,0)node(v8){$\P_\za=$}
	(-6.5,0)node(v7){$P_1$}
	(-5.5,0)node(v6){$P_2$}
	(-4.5,0)node(v5){$P_3$}
	(-3.5,0)node(v4){$P_4$}
	(-1,0)node(v2){$P_{10}\oplus P_5\oplus P_5$}
	(2,0)node(v1){$P_6\oplus P_7$,};
	\draw[ar](v7)edge(v6)
	(v6)edge(v5)
	(v5)edge(v4)
	(v4)edge(v2)
	(v2)edge(v1);
	\draw[black]
	($(v7)!.5!(v6)$) node[above] {$\aaa_1$}
	($(v6)!.5!(v5)$) node[above] {$\aaa_2$}
	($(v5)!.5!(v4)$) node[above] {$\aaa_3$}
	($(v4)!.4!(v2)$) node[above] {$\left(
		\begin{array}{ccc}
			0 & 0 & \aaa_4 \\
		\end{array}
		\right)
		$}
	($(v2)!.5!(v1)$) node[above] {$\left(
		\begin{array}{cc}
			\aaa_{10} & 0 \\
			\aaa_5 & \aaa_7 \\
			0 & \aaa_6\aaa_5 \\
		\end{array}
		\right)
		$};
\end{tikzpicture}
\]
%that is we have an exact sequence
%\[
%\begin{tikzpicture}[scale=.5,xscale=2.7,yscale=3,ar/.style={->,thick}]
%	\draw(-7,0)node(v8){$0$}
%	(-6.3,0)node(v7){$1$}
%	(-5.5,0)node(v6){$\begin{array}{c}
%			2 \\
%			1
%		\end{array}$}
%	(-4.5,0)node(v5){$\begin{array}{c}
%			3 \\
%			2
%		\end{array}$}
%	(-3.5,0)node(v4){$\begin{array}{c}
%			4 \\
%			3
%		\end{array}$}
%	(-1.8,0)node(v2){$10\oplus \begin{array}{c}
%			5 \\
%			4
%		\end{array}\oplus \begin{array}{c}
%			5 \\
%			4
%		\end{array}$}
%	(.5,0)node(v1){$\begin{array}{c}
%			~6~  \\
%			10~5
%		\end{array}
%		\oplus \begin{array}{c}
%			7 \\
%			5~6 \\
%			4~~~5
%		\end{array}
%		$}
%	(2.2,0)node(v0){$\begin{array}{c}
%			6~7 \\
%			\text{~}~\text{~}5~6
%		\end{array}
%		$}
%	(3.2,0)node(v){$0.$};
%	\draw[ar](v7)edge(v6)
%	(v8)edge(v7)
%	(v6)edge(v5)
%	(v5)edge(v4)
%	(v4)edge(v2)
%	(v2)edge(-.4,0)
%	(v1)edge(1.9,0);%
%	\draw[->,thick](2.6,0)to(3,0);
%\end{tikzpicture}
%\]
which can be read directly by using the projective coordinate $\zD^*_p=t(\zD_s)$ (green arcs):
\begin{center}{
		\begin{tikzpicture}[scale=0.8]
			\draw[thick,fill=white] (0,0) circle (4cm);		
			\draw[thick,fill=gray] (0,0) circle (0.5cm);
			
			\path (0:4) coordinate (b1)
			(67.5:4) coordinate (b2)
			(120:4) coordinate (b3)
			(165:4) coordinate (b4)
			(195:4) coordinate (b5)
			(225:4) coordinate (b6)
			(255:4) coordinate (b7)
			(-78:4) coordinate (b8)
			(-45:4) coordinate (b9)
			(-90:.5) coordinate (b10);
			
			\path (45:4) coordinate (r1)
			(90:4) coordinate (r2)
			(150:4) coordinate (r3)
			(180:4) coordinate (r4)
			(210:4) coordinate (r5)
			(240:4) coordinate (r6)
			(270:4) coordinate (r7)
			(-60:4) coordinate (r8)
			(-22.5:4) coordinate (r9)
			(90:.5) coordinate (r10);
			
			\draw[dark-green!50,thick]plot [smooth,tension=1] coordinates {(r9) (1.5,.3) (r10)};
			%
			%\draw[dark-green,thick]plot [smooth,tension=1] coordinates {(r2) (-.6,0) (.5,-.6) (.6,.3) (r10)};
			
			\draw [dark-green!50,thick,smooth] (r2) .. controls (-1,.5) and (-.8,-.8) .. (0,-.8) .. controls (.8,-.6) and (1,.5) ..(r10);
			
			\draw[dark-green!50,thick] (r2) to node[below]{\tiny$1$} (r3);
			\draw[dark-green!50,thick] (r1) to node[right]{\tiny$10$} (r9);
			\draw[dark-green!50,thick](r2) to node[below]{\tiny$2$} (r4);
			\draw[dark-green!50,thick] (r2) to node[left]{\tiny$3$} (r5);
			\draw[dark-green!50,thick] (r2) to node[left]{\tiny$4$} (r6);
			\draw[dark-green!50,thick] (r2) to node[above]{\tiny$6$} (r9);
			\draw[dark-green!50,thick, bend left] (r8) to node[above]{\tiny$8$} (r9);
			\draw[dark-green!50,thick, bend left] (r7) to node[above]{\tiny$9$} (r8);

			\draw[very thick]plot [smooth,tension=1] coordinates {(b1) (.3,1) (-.5,-1) (1.2,0) (b3)};
			
			\draw (b1) node[right] {$q$};
			\draw (b3) node[above] {$p$};
			\draw (2.8,.8) node {$\za$};
			\draw (-.4,1.4)[dark-green!50] node {\tiny$5$};
			\draw (1.6,0)[dark-green!50] node {\tiny$7$};

			\draw[thick,fill=white] (b1) circle (.08cm)
			(b2) circle (.08cm)
			(b3) circle (.08cm)
			(b4) circle (.08cm)
			(b5) circle (.08cm)
			(b6) circle (.08cm)
			(b7) circle (.08cm)
			(b8) circle (.08cm)
			(b9) circle (.08cm)
			(b10) circle (.08cm);

			\draw[thick,red,fill=red] (r1) circle (.08cm)
			(r2) circle (.08cm)
			(r3) circle (.08cm)
			(r4) circle (.08cm)
			(r5) circle (.08cm)
			(r6) circle (.08cm)
			(r7) circle (.08cm)
			(r8) circle (.08cm)
			(r9) circle (.08cm)
			(r10) circle (.08cm);	
	\end{tikzpicture}}
\end{center}

(2) The co-weights of $\za$ at $p$ and $q$ are
$$cw_p(\za)=1-1=0, cw_q(\za)=3-1=2.$$
Therefore by Corollary \ref{corollary:proj-res-string}, the injective dimension of $M_\za$ is
$$\id M_\za=max\{cw_p(\za),cw_q(\za)\}=2.$$
The minimal injective resolution of $M_\za$ is:
\[
\begin{tikzpicture}[scale=.5,xscale=3,yscale=3,ar/.style={->,thick}]
	\draw(-5.8,0)node(v1){$\I_\za=$}
	(-5,0)node(v2){$P_5\oplus P_6$}
	(-2.5,0)node(v3){$P_7\oplus P_7$}
	(-1,0)node(v4){$P_8$};
	\draw[ar](v2)edge(v3)
	(v3)edge(v4);
	\draw[black]
	($(v2)!.5!(v3)$) node[above] {\tiny$\left(
		\begin{array}{cc}
			\aaa_6\aaa_5 & \aaa_7 \\
			0 & \aaa_6 \\
		\end{array}
		\right)
		$}
	($(v3)!.6!(v4)$) node[above] {\tiny$\aaa_8,$};
\end{tikzpicture}
\]
%that is we have an exact sequence
%\[
%\begin{tikzpicture}[scale=.5,xscale=2.7,yscale=3,ar/.style={->,thick}]
%	\draw(-5,0)node(v8){$0$}
%	(-4,0)node(v7){$\begin{array}{c}
%			6~7 \\
%			\text{~}~\text{~}5~6
%		\end{array}$}
%	(-1.8,0)node(v2){$\begin{array}{c}
%			\text{~}~\text{~}\text{~}~\text{~}~9 \\
%			7\text{~}~\text{~}8\\
%			6~7\\
%			5
%		\end{array}\oplus \begin{array}{c}
%			7 \\
%			6
%		\end{array}$}
%	(.5,0)node(v1){$\begin{array}{c}
%			9  \\
%			8\\
%			7
%		\end{array}
%		\oplus \begin{array}{c}
%			9  \\
%			8\\
%			7
%		\end{array}     $}
%	(2.2,0)node(v0){$\begin{array}{c}
%			9 \\
%			8
%		\end{array}$}
%	(3.2,0)node(v){$0,$};
%	
%	\draw[->,thick](-4.8,0)to(-4.3,0);
%	\draw[->,thick](-3.5,0)to(-2.6,0);
%	\draw[->,thick](-.9,0)to(-.1,0);
%	\draw[->,thick](1.1,0)to(2,0);
%	\draw[->,thick](2.4,0)to(3,0);
%\end{tikzpicture}
%\]
which can be seen by using the injective coordinate $\zD^*_i=t^{-1}(\zD_s)$ (cf. Remark \ref{remark:inj}):
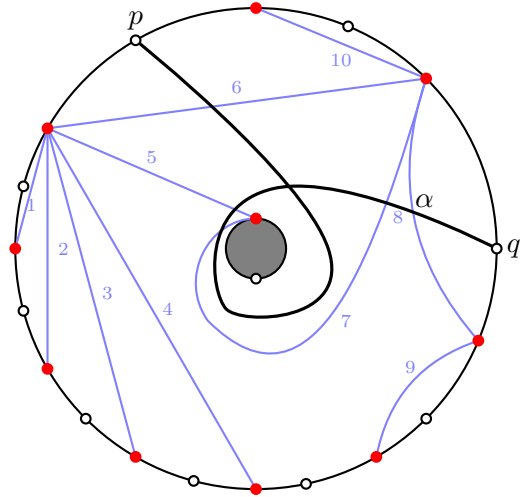
\begin{center}{
		\begin{tikzpicture}[scale=0.8]
			\draw[thick,fill=white] (0,0) circle (4cm);		
			\draw[thick,fill=gray] (0,0) circle (0.5cm);
			
			\path (0:4) coordinate (b1)
			(67.5:4) coordinate (b2)
			(120:4) coordinate (b3)
			(165:4) coordinate (b4)
			(195:4) coordinate (b5)
			(225:4) coordinate (b6)
			(255:4) coordinate (b7)
			(-78:4) coordinate (b8)
			(-45:4) coordinate (b9)
			(-90:.5) coordinate (b10);
			
			\path (45:4) coordinate (r1)
			(90:4) coordinate (r2)
			(150:4) coordinate (r3)
			(180:4) coordinate (r4)
			(210:4) coordinate (r5)
			(240:4) coordinate (r6)
			(270:4) coordinate (r7)
			(-60:4) coordinate (r8)
			(-22.5:4) coordinate (r9)
			(90:.5) coordinate (r10);
			\draw[blue!50,thick] (r3) to node[below]{\tiny$1$} (r4);
			\draw[blue!50,thick] (r2) to node[below]{\tiny$10$} (r1);
			\draw[blue!50,thick] (r10) to node[above]{\tiny$5$} (r3);
			\draw[blue!50,thick](r3) to node[right]{\tiny$2$} (r5);
			\draw[blue!50,thick] (r3) to node[right]{\tiny$3$} (r6);
			\draw[blue!50,thick] (r3) to node[right]{\tiny$4$} (r7);
			\draw[blue!50,thick] (r3) to node[above]{\tiny$6$} (r1);
			\draw[blue!50,thick, bend left] (r9) to node[left]{\tiny$8$} (r1);
			\draw[blue!50,thick, bend left] (r8) to node[above]{\tiny$9$} (r9);
			%\draw[blue,thick]plot [smooth,tension=1] coordinates {(r1) (.8,-1.2) (-1.2,-.2) (r10)};
			
			\draw [blue!50,thick,smooth] (r1) .. controls (2,0) and (1,-3) .. (-.8,-1.2) .. controls (-1.3,-.5) and (-.8,.5) ..(r10);
			
			\draw (b1) node[right] {$q$};
			\draw (b3) node[above] {$p$};
			\draw (2.8,.8) node {$\za$};
			\draw[blue!50] (1.5,-1.2) node {\tiny$7$};
			
			\draw[very thick]plot [smooth,tension=1] coordinates {(b1) (.3,1) (-.5,-1) (1.2,0) (b3)};
			
			\draw[thick,fill=white] (b1) circle (.08cm)
			(b2) circle (.08cm)
			(b3) circle (.08cm)
			(b4) circle (.08cm)
			(b5) circle (.08cm)
			(b6) circle (.08cm)
			(b7) circle (.08cm)
			(b8) circle (.08cm)
			(b9) circle (.08cm)
			(b10) circle (.08cm);

			\draw[thick,red,fill=red] (r1) circle (.08cm)
			(r2) circle (.08cm)
			(r3) circle (.08cm)
			(r4) circle (.08cm)
			(r5) circle (.08cm)
			(r6) circle (.08cm)
			(r7) circle (.08cm)
			(r8) circle (.08cm)
			(r9) circle (.08cm)
			(r10) circle (.08cm);	
	\end{tikzpicture}}
\end{center}

(3) Now let $\zb$ be the arc associated to the simple module $S_1$, see the following picture.
Then $\za$ and $\zb$ intersect at the point $p$, which gives rise to an oriented intersection $\aaa$ from $\za$ to $\zb$ of weight $5$.
Therefore $\aaa$ corresponds to a map in $\Ext^5(M_\za,M_\zb)$ by Theorem \ref{theorem:main-extensions}.

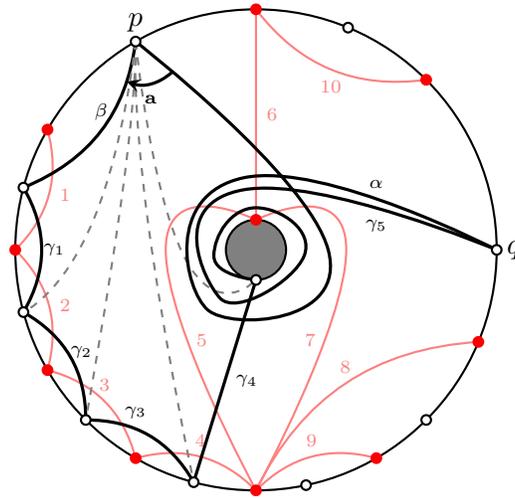
\begin{center}{
		\begin{tikzpicture}[scale=0.8]
			\draw[thick,fill=white] (0,0) circle (4cm);		
			\draw[thick,fill=gray] (0,0) circle (0.5cm);
			
			\path (0:4) coordinate (b1)
			(67.5:4) coordinate (b2)
			(120:4) coordinate (b3)
			(165:4) coordinate (b4)
			(195:4) coordinate (b5)
			(225:4) coordinate (b6)
			(255:4) coordinate (b7)
			(-78:4) coordinate (b8)
			(-45:4) coordinate (b9)
			(-90:.5) coordinate (b10);
			
			\path (45:4) coordinate (r1)
			(90:4) coordinate (r2)
			(150:4) coordinate (r3)
			(180:4) coordinate (r4)
			(210:4) coordinate (r5)
			(240:4) coordinate (r6)
			(270:4) coordinate (r7)
			(-60:4) coordinate (r8)
			(-22.5:4) coordinate (r9)
			(90:.5) coordinate (r10);
			
			\draw[red!50,thick]plot [smooth,tension=1] coordinates {(r7) (1.5,0) (r10)};
			\draw[red!50,thick]plot [smooth,tension=1] coordinates {(r7) (-1.5,0) (r10)};
			
			\draw[red!50,thick,bend left] (r1) to node[below]{\tiny$10$} (r2)
			(r3) to node[right]{\tiny$1$} (r4)
			(r4) to node[right]{\tiny$2$} (r5)
			(r5) to node[above]{\tiny$3$} (r6)
			(r6) to node[above]{\tiny$4$} (r7)
			(r7) to node[above]{\tiny$9$} (r8)
			(r7) to node[above]{\tiny$8$} (r9);
			\draw[red!50] (-.9,-1.5) node {\tiny$5$};
			\draw[red!50] (.9,-1.5) node {\tiny$7$};
			\draw[red!50,thick] (r10) to node[right]{\tiny$6$}  (r2);
			
			\draw (b1) node[right] {$q$};
			\draw (b3) node[above] {$p$};
			\draw (2,1.1) node {\tiny$\za$};
			\draw (2,.4) node {\tiny$\zg_5$};
			
			\draw[black,very thick,bend left] (b3) to node[above]{\tiny$\zb$} (b4)
			(b4) to (b5)
			(b5) to (b6)
			(b6) to (b7);
			\draw[black,very thick] (b7) to node[right]{\tiny$\zg_4$} (b10);                         \draw[] (-3.35,0) node {\tiny$\zg_1$};
			\draw[] (-2.9,-1.7) node {\tiny$\zg_2$};
			\draw[] (-2,-2.7) node {\tiny$\zg_3$};
			%\draw [very thick,smooth,black!60] (b1) .. controls (2,0) and (1,-3) .. (-.8,-1.2) .. controls (-1.3,-.5) and (-.8,.5) ..(b10);
			
			\draw[thick,dashed,gray]plot [smooth,tension=1] coordinates {(b3) (-2.7,.5) (b5)};
			\draw[thick,dashed,gray]plot [smooth,tension=1] coordinates {(b3) (-2.2,1.1) (b6)};
			\draw[thick,dashed,gray]plot [smooth,tension=1] coordinates {(b3) (-1.8,.5)(b7)};
			\draw[thick,dashed,gray]plot [smooth,tension=1] coordinates {(b3) (-1.2,-.1)(b10)};
			
			\draw[black,very thick,bend left,->,>=stealth] (-1.4,2.95) to node[below]{\tiny$\aaa$} (-2.15,2.8);
			
			\draw[very thick]plot [smooth,tension=1] coordinates {(b1) (-.2,1) (-.6,-.8) (.8,-.1) (0,.7) (-.7,-.1) (b10)};
			
			\draw[very thick]plot [smooth,tension=1] coordinates {(b1) (-.2,1.2) (-.8,-1) (1.2,0) (b3)};
			
			\draw[thick,fill=white] (b1) circle (.08cm)
			(b2) circle (.08cm)
			(b3) circle (.08cm)
			(b4) circle (.08cm)
			(b5) circle (.08cm)
			(b6) circle (.08cm)
			(b7) circle (.08cm)
			(b8) circle (.08cm)
			(b9) circle (.08cm)
			(b10) circle (.08cm);

			\draw[thick,red,fill=red] (r1) circle (.08cm)
			(r2) circle (.08cm)
			(r3) circle (.08cm)
			(r4) circle (.08cm)
			(r5) circle (.08cm)
			(r6) circle (.08cm)
			(r7) circle (.08cm)
			(r8) circle (.08cm)
			(r9) circle (.08cm)
			(r10) circle (.08cm);
	\end{tikzpicture}}
\end{center}

Furthermore, by Theorem \ref{theorem:yoneda}, the associated Yoneda $5$-extension is as follows:
\begin{gather*}
	\xymatrix@C=0.7cm
	{[\aaa]=0\ar[r]^{} & M_{\zb}\ar[r] & M_{\zg_1}\ar[r] & M_{\zg_2}\ar[r] & M_{\zg_3}\ar[r] & M_{\zg_4}\ar[r] & M_{\zg_5}\ar[r] & M_{\za}\ar[r]& 0},
\end{gather*}
that is,
\begin{gather*}
	\xymatrix@C=0.7cm
	{0\ar[r] & 1 \ar[r] & {\begin{array}{c}
				2 \\
				1
		\end{array}}
		\ar[r] & {\begin{array}{c}
				3 \\
				2
		\end{array}}
		\ar[r] & {\begin{array}{c}
				4 \\
				3
		\end{array}}
		\ar[r] & {\begin{array}{c}
				5 \\
				4
		\end{array}}
		\ar[r] & {\begin{array}{c}
				6~7 \\
				\text{~}~\text{~}5~6\\
				\text{~}~\text{~}\text{~}~\text{~}~\text{~}5
		\end{array}}
		\ar[r] & {\begin{array}{c}
				6~7 \\
				\text{~}~\text{~}5~6
		\end{array}}
		\ar[r]& 0}.
\end{gather*}

(4) Now we give an example of the geometric model of a heart in $\dba$.
As used in Section \ref{section:heart}, in the following we denote by $\zD^*_A$ the initial simple coordinate associate to $A$.
We consider a simple-minded collection
$$\P_{(\zD_s,f)}=\{S_i, i\neq 4,5\}\cup\{S_5[-1], \begin{array}{c}
	5 \\
	4
\end{array}
\},$$
in $\dba$. By the construction given in Definition \ref{definition:graded}, the associated graded simple coordinate $(\zD^*_s,F^*)$ is depicted in red arcs as follows
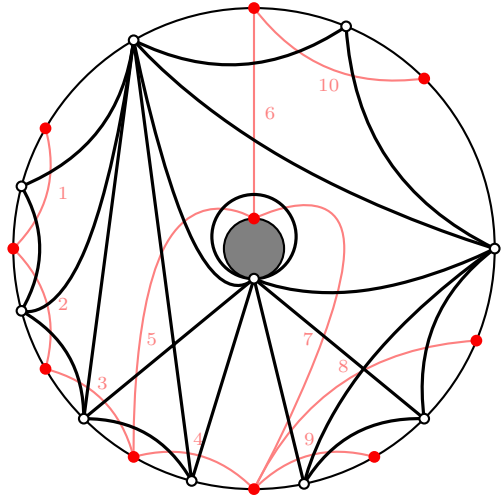
\begin{center}{
		\begin{tikzpicture}[scale=0.8]
			
			\draw[thick,fill=white] (0,0) circle (4cm);
			
			\path (0:4) coordinate (b1)
			(67.5:4) coordinate (b2)
			(120:4) coordinate (b3)
			(165:4) coordinate (b4)
			(195:4) coordinate (b5)
			(225:4) coordinate (b6)
			(255:4) coordinate (b7)
			(-78:4) coordinate (b8)
			(-45:4) coordinate (b9)
			(-90:.5) coordinate (b10);
			
			\path (45:4) coordinate (r1)
			(90:4) coordinate (r2)
			(150:4) coordinate (r3)
			(180:4) coordinate (r4)
			(210:4) coordinate (r5)
			(240:4) coordinate (r6)
			(270:4) coordinate (r7)
			(-60:4) coordinate (r8)
			(-22.5:4) coordinate (r9)
			(90:.5) coordinate (r10);

			\draw[red!50,thick]plot [smooth,tension=1] coordinates {(r7) (1.5,0) (r10)};
			\draw[red!50,thick]plot [smooth,tension=1] coordinates {(r6) (-1.5,0) (r10)};
			
			\draw[red!50,thick,bend left] (r1) to node[below]{\tiny$10$} (r2)
			(r3) to node[right]{\tiny$1$} (r4)
			(r4) to node[right]{\tiny$2$} (r5)
			(r5) to node[above]{\tiny$3$} (r6)
			(r6) to node[above]{\tiny$4$} (r7)
			(r7) to node[above]{\tiny$9$} (r8)
			(r7) to node[above]{\tiny$8$} (r9);
			\draw[red!50] (-1.7,-1.5) node {\tiny$5$};
			\draw[red!50] (.9,-1.5) node {\tiny$7$};
			
			\draw[red!50,thick] (r10) to node[right]{\tiny$6$}  (r2);
			
			\draw[black,very thick,bend left] (b3) to (b4)
			(b5) to (b6)
			(b6) to (b7)
			(b4) to (b5)
			(b8) to (b9)
			(b9) to (b1)
			(b1) to (b2)
			(b2) to (b3);
			\draw[black,very thick]    (b6) to (b10)
			(b7) to (b10)
			(b8) to (b10)
			(b9) to (b10)
			(b3) to (b6)
			(b3) to (b7);
			\draw[black,very thick]plot [smooth,tension=1] coordinates {(b1) (.5,1.5) (b3)};
			\draw[black,very thick]plot [smooth,tension=1] coordinates {(b8) (1.9,-1.9) (b1)};
			\draw[black,very thick]plot [smooth,tension=1] coordinates {(b3) (-2.8,0) (b5)};
			\draw[black,very thick]plot [smooth,tension=1] coordinates {(b3) (-1,.1) (b10)};
			\draw[black,very thick]plot [smooth,tension=1] coordinates {(b10) (1.8,-.7) (b1)};
			
			\draw[very thick] (0,.2) circle (.7cm);			
			\draw[thick,fill=gray] (0,0) circle (0.5cm);
			
			\draw[thick,fill=white] (b1) circle (.08cm)
			(b2) circle (.08cm)
			(b3) circle (.08cm)
			(b4) circle (.08cm)
			(b5) circle (.08cm)
			(b6) circle (.08cm)
			(b7) circle (.08cm)
			(b8) circle (.08cm)
			(b9) circle (.08cm)
			(b10) circle (.08cm);

			\draw[thick,red,fill=red] (r1) circle (.08cm)
			(r2) circle (.08cm)
			(r3) circle (.08cm)
			(r4) circle (.08cm)
			(r5) circle (.08cm)
			(r6) circle (.08cm)
			(r7) circle (.08cm)
			(r8) circle (.08cm)
			(r9) circle (.08cm)
			(r10) circle (.08cm);	
	\end{tikzpicture}}
\end{center}
where for an arrow $\aaa^*$, the associated grading $F^*(\aaa^*)$ equals zero, excepting for the arrows $\aaa^*_1: 6 \longrightarrow 5$ and $\aaa^*_2: 7 \longrightarrow 4$, where $F^*(\aaa^*_1)=-1$ and $F^*(\aaa^*_2)=1$.
Furthermore, by Theorem \ref{theorem: last}, the indecomposable objects in the heart $\calh$ generated by $\P_{(\zD_s,f)}$ are zigzag arcs on $(\cals,\calm)$ with respect to the graded simple coordinate $(\zD^*_s,F^*)$, which are the black arcs in the picture.

\begin{remark}\label{remark:last}
	Note that the simple coordinate $\zD^*_s$ is obtained from $\zD^*_A$ by certain flip, that is, we replace
	the arc $5$ in $\zD^*_A$ by the smoothing of the arcs $5$ and $4$, and we keep the other arcs in $\zD^*_A$.
	In fact the heart $\calh$ is the (simple) HRS-tilting \cite{HRS96,W10} of the torsion pair $(S_5,S_5^{\perp})$ in the module category $\ma$, or equivalently, the simple-minded collection $\P_{(\zD_s,f)}$ is the mutation of the canonical simple-minded collection consisting of simples in $\ma$.
	
	The above phenomenon that explains the (simple) HRS-tilting as a certain flip of the associated graded simple coordinate is a general phenomenon.
	In fact, the authors explain the silting mutations as flips of dissections in \cite{CS23b}, noticing that there is a kind of duality between silting objects and simple-minded collections, see for example in \cite{KY14}.
\end{remark}


\begin{thebibliography}{99}

\newcommand{\au}[1]{\textrm{#1},}
\newcommand{\ti}[1]{\textrm{#1},}
\newcommand{\jo}[1]{\textit{#1}}
\newcommand{\vo}[1]{\textbf{#1}}
\newcommand{\yr}[1]{(#1)}
\newcommand{\pp}[2]{#1--#2.}
\newcommand{\arxiv}[1]{\href{http://arxiv.org/abs/#1}{arXiv:#1}}


\bibitem[A21]{A21}
C. Amiot,
\newblock Indecomposable objects in the derived category of a skew-gentle algebra using orbifolds.
\newblock {\em EMS Ser. Congr. Rep.} European Mathematical Society (EMS), Z\"{u}rich, 2023, 1-24. 

\bibitem[AB22]{AB22}
C. Amiot and T. Br\"{u}stle,
\newblock Derived equivalences between skew-gentle algebras using orbifolds.
\newblock  {\em Doc. Math.} 27 (2022), 933-982.

\bibitem[ABCJP10]{ABCJP10}
I. Assem, T. Br\"ustle, C-J. Gabrielle and P-G. Plamondon,
\newblock Gentle algebras arising from surface triangulations.
\newblock {\em Algebra Number Theory} 4 (2010), no. 2, 201-229.

\bibitem[AG16]{AG16}
C. Amiot and Y. Grimeland,
\newblock Derived invariants for surface algebras.
\newblock {\em J. Pure Appl. Algebra} 220 (2016), no. 9, 3133-3155.

\bibitem[AH81]{AH81}
I. Assem and D. Happel,
\newblock Generalized tilted algebras of type {$A_{n}$}.
\newblock {\em Comm. Algebra} 9 (1981), no. 20, 2101-2125.

\bibitem[ALP16]{ALP16}
K-K. Arnesen, R. Laking and D. Pauksztello,
\newblock Morphisms between indecomposable complexes in the bounded derived category of a gentle algebra.
\newblock {\em J. Algebra} 467 (2016), 1-46.

\bibitem[ALP21]{ALP21}
C. Amiot, D. Labardini-Fragoso and P-G. Plamondon,
\newblock Derived invariants for surface cut algebras of global dimension 2 II: the punctured case.
\newblock {\em Comm. Algebra} 49 (2021), no. 1, 114-150.

\bibitem[APS23]{APS23}
C. Amiot, P-G. Plamondon and S. Schroll,
\newblock A complete derived invariant for gentle algebras via winding numbers and Arf invariants. 
\newblock {\em Selecta Math. (N.S.)} 29 (2023), no. 2, Paper No. 30, 36 pp.


\bibitem[AS87]{AS87}
I. Assem and A. Skowro\'{n}ski,
\newblock Iterated tilted algebras of type affine $A_n$.
\newblock {\em Math. Z.} 195 (1987), no. 2, 269-290.

\bibitem[B07]{B07}
  \newblock T.~Bridgeland,
  \newblock Stability conditions on triangulated categories.
  \newblock {\em Ann. Math.} (2) 166 (2007), no. 2, 317-345.

\bibitem[B11]{B11}
\newblock G. Bobi{\'n}ski,
\newblock The almost split triangles for perfect complexes over gentle algebras.
\newblock {\em J. Pure Appl. Algebra}  215 (2011), no. 4, 642-654.

\bibitem[B16]{B16}
R. Bocklandt,
\newblock Noncommutative mirror symmetry for punctured surfaces. With an appendix by Mohammed Abouzaid.
\newblock {\em Trans. Amer. Math. Soc.} 368 (2016), no. 1, 429-469.

\bibitem[BBD82]{BBD82}
A.~A. Beilinson, J. Bernstein and P. Deligne,
\newblock Analyse et topologie sur les espaces singuliers.
\emph{Ast{\'e}risque}, vol. 100, Soc. Math. France, 1982 (French).


\bibitem[BBM14]{BBM14}
K. Baur, A. B. Buan and R. J. Marsh,
\newblock Torsion Pairs and Rigid Objects in Tubes.
\newblock {\em Algebr. Represent. Theory}, volume 17, pages 565-591 (2014)


\bibitem[BPP16]{BPP16}
N. Broomhead, D. Pauksztello and D. Ploog,
\newblock Discrete derived categories II: the silting pairs
CW complex and the stability manifold.
\newblock {\em J. London Math. Soc.},  93 (2016), 273–300.


\bibitem[BPP17]{BPP17}
N. Broomhead, D. Pauksztello and D. Ploog,
\newblock Discrete derived categories I: homomorphisms,
autoequivalences and t-structures.
\newblock {\em Math. Z.},  285 (2017), 39–89.


\bibitem[BC21]{BC21}
K. Baur and R. Coelho Sim\~oes,
\newblock A geometric model for the module category of a gentle algebra.
\newblock {\em Int. Math. Res. Not. IMRN} (2021), no. 15, 11357-11392.

\bibitem[BD18]{BD18}
I. Burban and Y. Drozd,
\newblock Non-commutative nodal curves and derived tame algebras.
\newblock (2018) Preprint
  \href{https://arxiv.org/abs/1805.05174}{\texttt{arXiv:1805.05174}}.

\bibitem[BDMTY20]{BDMTY20}
T. Br\"{u}stle, G. Douville, K. Mousavand, H. Thomas and E. Y\i ld\i r\i m,
\newblock On the combinatorics of gentle algebras.
\newblock {\em Can. J. Math.} 72 (6) (2020) 1551-1580.

\bibitem[BM03]{BM03}
V. Bekkert and H A. Merklen,
\newblock Indecomposables in derived categories of gentle algebras.
\newblock {\em Algebr. Represent. Theory} 6 (2003), no. 3, 285-302.

\bibitem[BR87]{BR87}
M. C. R. Butler and C. M. Ringel,
\newblock Auslander-Reiten sequences with few middle terms and applications to string algebras.
\newblock {\em Comm. Algebra} 15 (1987), no. 1-2, 145-179.


\bibitem[BS21]{BS21}
K. Baur and S. Schroll,
\newblock Higher extensions for gentle algebras.
\newblock {\em Bull. Sci. Math.} 170 (2021), Paper No. 103010, 21 pp.


\bibitem[BPT]{BPT}
T. Br\"{u}stle, D. Pauksztello and H. R. Tyler,
\newblock Cohomology of complexes for gentle algebras via
geometric models. In preparation.

\bibitem[BZ11]{BZ11}
T. Br\"{u}stle and J. Zhang,
\newblock On the cluster category of a marked surface without punctures.
\newblock {\em Algebra Number Theory} 5 (2011), no. 4, 529-566.

\bibitem[CCS06]{CCS06}
P. Caldero, F. Chapoton and R. Schiffler,
\newblock Quivers with relations arising from clusters (An case).
\newblock {\em Trans. Amer. Math. Soc.} 358 (2006), no. 3, 1347-1364.


\bibitem[CD20]{CD20}
A. Chan and L. Demonet,
\newblock Classifying torsion classes of gentle algebras.
\newblock (2020) Preprint
  \href{https://arxiv.org/abs/2009.10266}{\texttt{arXiv:2009.10266}}.

\bibitem[CJS22]{CJS22}
W. Chang, H. B. Jin and S. Schroll,
\newblock Recollements of derived categories of graded gentle algebras and surface cuts.
\newblock (2022) Preprint
  \href{https://arxiv.org/abs/2206.11196}{\texttt{arXiv:2206.11196}}.



\bibitem[CHS23]{CHS23}
W. Chang, F. Haiden and S. Schroll,
\newblock Braid group actions on branched coverings and full exceptional sequences.
\newblock (2023) Preprint
  \href{https://arxiv.org/abs/2301.04398}{\texttt{arXiv:2301.04398}}.


\bibitem[CP16]{CP16}
R. Coelho Sim\~oes and D. Pauksztello,
\newblock Torsion pairs in a triangulated category generated by a spherical object.
\newblock {\em J. Algebra} 448, (2016): 1-47.

\bibitem[CPS19]{CPS19}
\.{I}.~\c{C}anak\c{c}i, D. Pauksztello and S. Schroll,
\newblock Mapping cones in the bounded derived category of a gentle algebra.
\newblock {\em J. Algebra} 530 (2019), 163-194.

\bibitem[CPS21]{CPS21}
\.{I}.~\c{C}anak\c{c}i, D. Pauksztello and S. Schroll,
\newblock On extensions for gentle algebras.
\newblock {\em Canad. J. Math.} 73 (2021), no. 1, 249-292.

\bibitem[CS17]{CS17}
\.{I}.~\c{C}anak\c{c}i and S. Schroll,
\newblock Extensions in Jacobian algebras and cluster categories of marked surfaces.
\newblock {\em Adv. Math.} 313 (2017) 1-49.

\bibitem[CS23a]{CS23a}
W. Chang and S. Schroll,
\newblock Exceptional sequences in the derived category of a gentle algebra. 
\newblock {\em Selecta Math. (N.S.)} 29 (2023), no. 3, Paper No. 33. 41 pp. 

\bibitem[CS23b]{CS23b}
W. Chang and S. Schroll,
\newblock A geometric realization of silting theory for gentle algebras.
\newblock {\em Math. Z.}  303 (2023), no. 3, Paper No. 67, 37 pp.

\bibitem[D14]{D14}
L. David-Roesler,
\newblock Derived Equivalence of Surface Algebras in Genus 0 via Graded Equivalence.
\newblock {\em  Algebr. Represent. Theory} Vol. 17 (2014), No. 1, 1-30.

\bibitem[DS12]{DS12}
L. David-Roesler and R. Schiffler,
\newblock Algebras from surfaces without punctures.
\newblock {\em J. Algebra}  350 (2012), 218-244.



\bibitem[FGLZ23]{FGLZ23}
C. J. Fu, S. F. Geng, P. Liu and Y. Zhou,
\newblock On support $\tau$-tilting graphs of gentle algebras.
\newblock {\em J. Algebra} 628 (2023), 189–211. 


\bibitem[FST08]{FST08}
  \au{S.~Fomin, M.~Shapiro \and D.~Thurston}
  \ti{Cluster algebras and triangulated surfaces, part I: Cluster complexes}
  \jo{Acta Mathematica} \vo{201} \yr{2008} \pp{83}{146}



\bibitem[HJR11]{HJR11}
T. Holm, P. Jorgensen and M. Rubey,
\newblock Ptolemy diagrams and torsion pairs in the cluster category of Dynkin type An.
\newblock {\em J. Algebraic Combin.} 34, no. 3 (2011): 507-23.



\bibitem[HKK17]{HKK17}
F. Haiden, L. Katzarkov and M. Kontsevich,
\newblock Flat surfaces and stability structures.
\newblock {\em Publ. Math. Inst. Hautes \'Etudes Sci.}
126 (2017), 247-318.

\bibitem[HRS96]{HRS96}
D. Happel, I. Reiten and S. Smal\o,
\newblock Tilting in abelian categories and quasitilted algebras.
\newblock {\em Mem. Amer. Math. Soc.}
120 (1996) viii+88.

\bibitem[HS05]{HS05}
B. Huisgen-Zimmermann and S. O. Smal\o,
\newblock The homology of string algebras. I.
\newblock {\em J. Reine Angew. Math.}
580 (2005), 1-37.

\bibitem[HZZ23]{HZZ23}
P. He, Y. Zhou and B. Zhu,
\newblock A geometric model for the module category of a skew-gentle algebra.
\newblock {\em Math. Z.} 304 (2023), no. 1, Paper No. 18, 41 pp.

\bibitem[GHZ85]{GHZ85}
E. L. Green, D. Happel and D. Zacharia,
\newblock Projective resolutions over Artin algebras with zero relations.
\newblock {\em Illinois J. Math.} 29 (1985), no. 1, 180-190.


\bibitem[K15]{K15}
M. Kalck,
\newblock Singularity categories of gentle algebras.
\newblock {\em Bull. Lond. Math. Soc.} 47 (1) (2015) 65-74.

\bibitem[KY14]{KY14}
S. Koenig and D. Yang,
\newblock Silting objects, simple-minded collections, t-structures and co-t-structures for finite-dimensional algebras.
\newblock {\em Doc. Math.} 19 (2014), 403-438.


\bibitem[L09]{L09}
  \au{D. Labardini-Fragoso}
  \ti{Quivers with potentials associated to triangulated surfaces}
  \jo{Proc. Lond. Math. Soc.} (3) 98:3 \yr{2009} \pp{797}{839}

\bibitem[LGH22]{LGH22}
Y. Z. Liu, H. P. Gao and Z. Y. Huang ,
\newblock Homological dimensions of gentle algebras via geometric models.
\newblock (2022) Preprint \arxiv{2208.13180}.


\bibitem[LP20]{LP20}
Y. Lekili and A. Polishchuk,
\newblock Derived equivalences of gentle algebras via Fukaya categories.
\newblock {\em Math. Ann.} 376 (2020), no. 1-2, 187-225.

\bibitem[LSV22]{LSV22}
D. Labardini-Fragoso, S. Schroll and Y. Valdivieso,
\newblock Derived categories of skew-gentle algebras and orbifolds.
\newblock {\em Glasg. Math. J.}  64 (2022), no. 3, 649–674. 

\bibitem[O19]{O19}
S. Opper,
\newblock On auto-equivalences and complete derived invariants of gentle algebras.
\newblock (2019) Preprint
  \href{https://arxiv.org/abs/1904.04859}{\texttt{arXiv:1904.04859}}.

\bibitem[OPS18]{OPS18}
\au{S. Opper, P-G Plamondon \and S. Schroll}
\ti{A geometric model for the derived category of gentle algebras}
\yr{2018} Preprint
\arxiv{1801.09659}.

\bibitem[OZ22]{OZ22}
\au{S. Opper \and A. Zvonareva}
\ti{Derived equivalence classification of Brauer graph algebras}
{\em Adv. Math.} 402 (2022), Paper No. 108341, 59 pp.

\bibitem[PPP19]{PPP19}
Y. Palu, V. Pilaud and P-G Plamondon,
\newblock Non-kissing and non-crossing complexes for locally gentle algebras.
\newblock  {\em J. Comb. Algebra} 3 (2019), no. 4, 401-438.

\bibitem[PPP21]{PPP21}
Y. Palu, V. Pilaud and P-G Plamondon,
\newblock Non-kissing complexes and tau-tilting for gentle algebras.
\newblock  {\em Mem. Amer. Math. Soc.} 274 (2021), no. 1343, vii+110 pp.

\bibitem[QZZ22]{QZZ22}
  \au{Y.~Qiu, C. Zhang \and Y. Zhou}
  \ti{Two geometric models for graded skew-gentle algebras}
  \yr{2022} Preprint
  \arxiv{2212.10369}.

\bibitem[S99]{S99}
J. Schröer,
\newblock Modules without self-extensions over gentle algebras.
\newblock {\em J. Algebra}, 216 (1999), 1, 178-189.

\bibitem[W10]{W10}
J. Woolf,
\newblock Stability conditions, torsion theories and tilting.
\newblock {\em J. Lond. Math. Soc.}, (2) 82 (2010), no. 3, 663-682.

\bibitem[Z14]{Z14}
J. Zhang,
\newblock On indecomposable exceptional modules over gentle algebras.
\newblock {\em Comm. Algebra}, 42 (7) (2014) 3096-3119.


\end{thebibliography}
\end{document}